\documentclass[12pt,a4paper]{article}
\usepackage[utf8]{inputenc}
\usepackage[english]{babel}
 \usepackage{mathrsfs}
\usepackage[T1]{fontenc}
\usepackage{amsmath}
\usepackage{amsfonts}
\usepackage{amssymb}
\usepackage{makeidx}
\usepackage{color}
\usepackage{amsthm}
\usepackage{tikz}
\usepackage{graphicx}
\usepackage{hyperref}
\usepackage[left=2cm,right=2cm,top=2cm,bottom=2cm]{geometry}
\author{Dang Nguyen-Bac}
  \usepackage[all]{xy}

\newcounter{counter_thm_int}

\newtheorem{thm_int}[counter_thm_int]{Theorem}
\newtheorem{conjecture}[counter_thm_int]{Conjecture}

\newtheorem{cor_int}[counter_thm_int]{Corollary}

\newtheorem{thm}{Theorem}[section]

\newtheorem{lem}[thm]{Lemma}
\newcounter{counter_conj-lem}
\newtheorem{conj-lem}[thm]{Conjecture-Lemma}

\newtheorem{prop}[thm]{Proposition}
\newcounter{counter_conj-prop}
\newtheorem{conj-prop}[thm]{Conjecture-Proposition}

\newcounter{counter_conj-thm}
\newtheorem{conj-thm}[thm]{Conjecture-Theorem}

\newtheorem{cor}[thm]{Corollary}
\newtheorem{defi}[thm]{Definition}

\theoremstyle{remark}

\newtheorem{rem}[thm]{Remark}

\newtheorem{ex}[thm]{Example}

\newcommand{\black}{\color{black}}
\newcommand{\red}{\color{red}}

\DeclareMathOperator{\PGL}{PGL}
\DeclareMathOperator{\PSO4}{PSO_4}

\DeclareMathOperator{\Bir}{Bir}

\DeclareMathOperator{\CAT}{CAT}
\DeclareMathOperator{\Conj}{Conj}
\DeclareMathOperator{\M}{M}

\DeclareMathOperator{\C}{\K}
\DeclareMathOperator{\K}{k}

\DeclareMathOperator{\divi}{div}
\DeclareMathOperator{\ord}{ord}

\DeclareMathOperator{\Ks}{K}
\DeclareMathOperator{\degtr}{trdeg}

\DeclareMathOperator{\Der}{Der}

\DeclareMathOperator{\Min}{Min}

\DeclareMathOperator{\Id}{Id}
\DeclareMathOperator{\Ker}{Ker}
\DeclareMathOperator{\Aut}{Aut}

\DeclareMathOperator{\Spec}{Spec}
 
\DeclareMathOperator{\stab}{Stab}

\DeclareMathOperator{\SL2}{\mathcal{Q}}
\DeclareMathOperator{\CSL2}{\C[\SL2]}

\DeclareMathOperator{\Jac}{Jac}
\DeclareMathOperator{\O4}{O_4}

\DeclareMathOperator{\tame}{Tame(Q)}

\DeclareMathOperator{\Pg}{\mathbb{P}}

\newcommand{\dtypeI}[1]{\draw #1 node[blue] {$\circ$};}
\newcommand{\dtypeII}[1]{\draw #1 node[violet] {$\bullet$};}

\DeclareMathOperator{\EV}{E_V}
\DeclareMathOperator{\EH}{E_H}

\DeclareMathOperator{\stame}{STame(Q)}

\DeclareMathOperator{\GL}{GL}
\DeclareMathOperator{\SO_4}{SO_4}

\title{Degree growth for tame automorphisms of an affine quadric threefold}
\begin{document}

\maketitle

\abstract{ In this paper, we consider the degree sequences of the tame automorphisms preserving an affine quadric threefold. Using some valuatives estimates derived from the work of Shestakov-Umirbaev and the action of this group on a $\CAT(0)$, Gromov-hyperbolic square complex constructed by Bisi-Furter-Lamy, we prove that the dynamical degrees of tame elements avoid any value strictly between 1 and 4/3. As an application, these methods allow us to characterize when the growth exponent of the degree of a random product of finitely many tame automorphisms is positive.}

\section*{Introduction}


Fix a projective variety $X$ of dimension $n$ defined over an algebraically closed field  $\C$ of characteristic zero and a rational map $f$ on $X$. 
We are interested in the complexity associated to the dynamical system induced by $f$, more precisely on the growth of the degrees of the $p$-fold composition  $f^p = f \circ \ldots \circ f$. 
This general problem was addressed in the work of Russakovski-Shiffman (\cite{russakovskii_shiffman}) when $X =\Pg^n$ in which they related the asymptotic behavior of the images by $f$ of the linear subvarieties of $\mathbb{P}^n$ with the degree sequences. 
The asymptotic ratios of these sequences, denoted $\lambda_i(f)$ for $i \leqslant n$, and refered as dynamical degrees, control the topological entropy of those maps (\cite{dinh_sibony_une_borne_sup}) and are crucial for the construction of an invariant measure of maximal entropy (\cite{bedford_smillie_3,guedj_ergodic,bedford_diller_energy_invariant}). 

When $f$ is a birational surface map, the situation is completely classified (\cite{gizatullin_80}, \cite{diller_favre}, \cite{cantat_bir_surfaces}, \cite{blanc_cantat}).
For general  rational maps on surfaces, the behavior of the degree is known for morphisms of the affine plane (\cite{favre_jonsson_dynamical_compactification}) and  when $\lambda_1(f)^2 > \lambda_2(f)$ (\cite{boucksom_favre_jonsson_deggrowth}).  

From the dimension three on, only the degree growth of monomial maps (\cite{lin_algebraic_stability_and_degree_growth}, \cite{favrewulcandegree}), regular morphisms, pseudo-automorphisms (\cite{oguiso_truong_salem_threefold,truong_pseudo_aut,
bedford_invertible_blow_ups,
oguiso_truong_explicit_rational_cy_threefolds_positive_entropy,
truong_aut_blowup_threefolds}), birational maps on hyperk\"ahler varieties (\cite{bianco}) and sporadic examples (\cite{maillard_baxter_birational,
maillard_growth_complexity,maillard_complexity,bedford_truong_matrix_inversion,bedford_kim_dynamics_pseudo_3}) were studied. The general problem of understanding the degree of the iterates of birational transformations of $\mathbb{P}^3$ remains  open. 
The  main reason is that we usually rely on the construction of a good birational model (e.g. an algebraically stable model in the sense of Fornaess and Sibony~\cite{fornaess_sibony_II}) to find the degree sequences, but the structure of the set of birational models of threefolds is far more complicated than its analog for surfaces. 
It is thus natural to ask whether we can find a large class of birational transformations of  $\mathbb{P}^3$ for which this sequence is fully understood.

A first natural choice would be the group of polynomial automorphisms of the three dimensional affine space. 
Even though there has been some recent work on particular subgroups of this group (\cite{wright_tame_A3}, \cite{lamy_combinatorics}, \cite{lamy_przytycki}), their dynamical degrees were not computed. We have thus turned our attention to a simpler situation, namely the subgroup of tame automorphisms of the affine quadric threefold.
\medskip

We denote by $(x,y,z,t)$ the affine coordinates in $\mathbb{A}^4$ and consider the affine quadric $\SL2$ given by:
\begin{equation*}
\SL2 = V( xt - yz - 1 ).
\end{equation*}
Observe that the Picard group of the closure $\overline{\SL2}$  of $\SL2$ in $\Pg^4$ is generated by $H = c_1(\mathcal{O}(1)_{| \overline{\SL2}})$ so that one can define the algebraic degree of an automorphism by:
%
\begin{equation*}
\deg(f) := \deg_1(f) =  (\pi_1^* H^2 \cdot \pi_2^{*} H),
\end{equation*}
where $\pi_1$ and $\pi_2$ are the projections of the graph of the birational map induced by $f $ in $ \overline{\SL2} \times \overline{\SL2}$ onto the first and the second factor respectively.
Observe that by definition $\deg_2(f) = (\pi_1^* H \cdot \pi_2^* H^2) = \deg(f^{-1})$ since $f$ is an automorphism. 
%

The subgroup of tame automorphisms, denoted $\tame$,  is defined as the subgroup generated by affine automorphism and transformations induced by:
\begin{equation*}
(x,y,z,t) \mapsto (x, y, z+ xP(x,y) , t + yP(x,y)),
\end{equation*}
with $P\in \C[x,y]$.


\begin{thm_int} \label{thm_int_degree_growth}
Let $f$ be a tame automorphism, then one of the following possibilities occur:
\begin{enumerate}
\item[(i)] The sequence $(\deg(f^n), \deg(f^{-n}))$ is bounded and $f$ is conjugated to  linear map; or $f^2$ is conjugated to an automorphism of the form 
\[(x,y,z,t) \mapsto (a x, b y + xR(x) , b^{-1}z + xP(x,y), a^{-1}(t + y P(x,y) + z R(x) + x R(x) P(x,y)))\] with $a,b \in \C^*$, $P\in\C[x,y]$ and $R \in \C[x]$.
\item[(ii)] There exists a constant $C>0$ such that:
 $$ \dfrac{1}{C} n \leqslant \deg(f^{\epsilon n}) \leqslant Cn ,$$
for all $\epsilon \in \{ +1, -1 \}$ and $f$ is conjugated  to an automorphism of the form:
\begin{equation*}
(x,y,z,t) \mapsto (ax ,b^{-1} (z + xR(x)) , b(y + xP(x) z ), a^{-1}(t + z^2P(x) + yR(x) +xzP(x)R(x))),
\end{equation*}
with $a,b \in \C^*$, $R \in \C[x]$ and $P\in \C[x]\setminus \C$.
\item[(iii)] The sequences $\deg(f^n)$ and $\deg(f^{-n})$ grow at least exponentially and there exists a constant $C(f)>0$  such that:
\begin{equation*}
\min( \deg(f^{-n}), \deg(f^n)) \geqslant C(f) \left ( \dfrac{4}{3} \right )^n.
\end{equation*}
\end{enumerate}
\end{thm_int}

Theorem \ref{thm_int_degree_growth} is a first step towards an understanding of the dynamical degrees of  these particular automorphisms. 

\begin{cor_int} \label{cor_spectral_gap}
The following inclusion is satisfied:
\begin{equation*}
\{ \lambda_1(f) \ | \ f \in \tame \} \subset \  \{ 1\} \cup [4/3, +\infty[.
\end{equation*}
\end{cor_int}

This result is reminiscent of a theorem of Blanc and Cantat~\cite[Corollary 2.7]{blanc_cantat} stating that the set of first dynamical degrees of any birational surface maps is included in $\{1\} \cup [\lambda_L, \infty)$ where $\lambda_L\simeq 1.176280$ denotes the Lehmer number.  We conjecture however that  the dynamical degrees of a tame automorphism is always an integer.

 Another immediate consequence of Theorem \ref{thm_int_degree_growth} is the following corollary.
\begin{cor_int}\label{cor_int_lambda_1_fibration} 
Any tame automorphism $f \in \tame$  satisfying $\lambda_1 (f) =1$ preserves a fibration or belongs to $\O4$ and both sequences $\deg(f^n),\deg(f^{-n})$ are either bounded or linear.
\end{cor_int} 
The above result gives a positive answer to a question by Urech~\cite[Question 4]{urech}  in this special situation. 
 
 \medskip


The proof of Theorem \ref{thm_int_degree_growth} exploits extensively the structure of the group  of tame automorphisms. 
We use the natural action of $\tame$ on a square complex  $\mathcal{C}$ which was introduced and studied by Bisi-Furter-Lamy
 in~\cite{bisi_furter_lamy}. This action is faithful, transitive on squares, and isometric. The complex $\mathcal{C}$ plays the same role for $\tame$ as the Bass-Serre tree  for $\mathrm{Aut}[\C^2]$. 

One of the main result of [BFL14] is that $\mathcal{C}$ is a geodesic space which is both $\CAT(0)$ and Gromov-hyperbolic. 
%
As a result, a tame automorphism induces an action on the complex which is rather constrained: either it is elliptic and fixes a vertex in the complex $\mathcal{C}$; or it is hyperbolic and acts by translation on an invariant geodesic line.

Using an explicit description of the stabilizer subgroups of each vertices, we compute the degree sequences of all elliptic tame automorphisms. 

  The crucial point of the proof is the study in \S\ref{section_comparison} of the degree growth of hyperbolic automorphisms.
In this case, we show that  the sequence of degrees is bounded from below by $C\, (4/3)^n$  for some positive constant $C>0$ and where $n$ depends on the distance of translation on an invariant geodesic line. 
 Let us state a weaker statement which summarizes the overall idea of our proof and which relates the degree with the displacement by $f$ of a vertex $v_0$ fixed by the linear group.

\begin{thm_int} \label{thm_int_degree_versus_distance} For any tame automorphisms $f\in \tame$ for which $f$ is not affine, the following inequality holds:
\begin{equation*}
\log (\deg(f)) \geqslant  \dfrac{\log(4/3)}{2\sqrt{2}}  d_{\mathcal{C}}( f \cdot  v_0,   v_0)  - 2 \log(4/3),
\end{equation*}
where $d_\mathcal{C}$ denotes the distance in the complex.
\end{thm_int}
This phenomenon already appears in the case of plane automorphisms since one can bound from below the logarithm of the degree of a plane automorphism by $\log(2)$ multiplied by the  distance between two vertices in the Bass-Serre tree associated to the group $\Aut(\mathbb{A}^2)$. Also in the case of $\Bir(\Pg^2)$, there is a relationship between the degree and the distance on a suitable hyperbolic space. 
The above result does not imply Theorem \ref{thm_int_degree_growth} and one needs to prove a more refined statement to obtain that the degree of $f^n$ is indeed larger than $(4/3)^n$.
  Let us explain how this is done. 
 \bigskip

Let $f\in\tame$ be any hyperbolic automorphism. 
First we show that by conjugating with an appropriate automorphism, we can suppose that $v_0$ lies at distance $\le 2$ of an $f$-invariant geodesic line. 
Suppose that $v_0$ is contained in an invariant geodesic of $f$. Our goal is to prove that 
\begin{equation}\label{WIN}
\deg(f^n) \geqslant (4/3)^{d_\mathcal{C}(v_0, f^n \cdot v_0)} \text{ for all } n \in \mathbb{N}.
\end{equation}
 The sequence of large squares (i.e isometric to $[0,2]^2$) cut by the geodesic segment $[v_0, f^n \cdot v_0]$ allows us to write
 \begin{equation}\label{eq:decomp}
 f^n = g_p \circ g_{p-1} \cdots \circ g_1 
 \end{equation}
as a composition of elementary automorphisms and affine transformations preserving the quadric. 
 This decomposition is not unique in general and ideally, one would hope to prove that the degree is multiplicative so that $\deg(f^n) \ge \prod_{i=1}^p \deg(g_i)$. 
 The obstruction to this property is the presence of resonances, which are explained as follows. 
 Two regular functions $P,R \in \C[Q]$ are \textit{resonant} if there exists $\lambda \in \C^*$ and two integers $p,q$ such that $ \deg(P^p - \lambda R^q) < p \deg(P) = q \deg(R)$ and they are called \textit{critical} if $p = 1$ or $q=1$.
   
  When these resonances are 
 not \emph{critical}, 
  we show that one can apply 
 the so-called parachute inequalities (recalled in \S \ref{subsection_parachute_ineq}) to deduce~\eqref{WIN}. 
 These  inequalities are elementary valuative estimates on the values of partial derivatives of suitable polynomials, and are derived from the proof of Nagata's conjecture by Shestakov-Umirbaev (see \cite{shestakov}, \cite{kuroda}, \cite{lamy_venereau}).

To get around the appearance of critical resonances, we exploit the structure of the tame group to prove that $f^n$ always admits an appropriate factorization for which the parachute inequalities can be applied inductively. In other words, we write $f^n =g'_p \circ \ldots \circ g'_1$ where $g'_i$ are tame automorphisms such that for each $i \leqslant p$,  $g'_{i+1}$ and  $(g'_i \circ \ldots \circ g'_1)$ do not have critical resonances. 

Using the correspondence between the factorizations of $f$ and the sequences of large squares  cut out by the invariant geodesic, we are reduced to proving that one can modify inductively our initial sequence of large  squares to avoid critical resonances. 
The essential point is to choose a valuation $\nu$ of monomial type (i.e with different weights on the coordinate axis $x,y,z,t$) such that one of the vertex of our initial large square has $\nu$-value
strictly less than the three others. The dissymmetry induced by $\nu$ will be propagated along  any sequence of squares following our geodesic. We then  argue that this minimality property on each large square  allows us to choose another square with no critical resonances.
As a result, the core of our approach relies deeply on the structure of the tame group which is reflected by the geometric properties of the square complex. Our proof is presented using purely combinatorial arguments. 

\bigskip

In the last part of this paper, we shall give a random version of Theorem \ref{thm_int_degree_growth}. Consider a finitely generated subgroup $G$ of the tame group and an atomic probability measure $\mu$ on $G$ such that:
\begin{equation*}
\int_G \log (\deg(g)) d\mu(g) < +\infty.
\end{equation*} 
The random walk on $G$ with transition law $\mu$ is the Markov chain starting at $\Id$ with transition law $\mu$. 
The state of the Markov chain $g_n$ at the time $n$ is equal to  the product of $n$ independent, identically distributed random variable on $G$ with distribution law $\mu$. Its distribution law $\nu_n$ is the $n$-fold convolution of $\mu$. 
Since the degree is submultiplicative, Kingman's subadditivity asserts that the degree exponents given by 
\begin{equation*}
\lambda_1(\mu) := \limsup_{n \rightarrow +\infty} \dfrac{1}{n}\int_G \log(\deg(g)) d\nu_n(g)   
\end{equation*} 
and 
\begin{equation*}
\lambda_2(\mu) := \limsup_{n \rightarrow +\infty} \dfrac{1}{n}\int_G \log (\deg(g^{-1})) d\nu_n(g)   
\end{equation*}
are finite. 
 These numbers measure the complexity of our random walk and one recovers the first and second dynamical degrees of $f$ when $\mu$ is equal to the Dirac measure at $f$.
Since the degree is equal to the norm of the pullback operator induced by $f$ on the Neron-Severi group of the quadric, these quantities play the same role as the Lyapounov exponents of a random products of matrices (\cite{furstenberg_kesten_product_random_matrices,furstenberg_non_commuting_random_products}) for this group and  
 the existence of these exponents can thus be interpreted as a law of large number (\cite[Theorem 0.6]{quint_random_walk_reductive}).

We now state the following result on the behavior of any symmetric random walks on this particular group.
\begin{thm_int} \label{thm_random_walk} Let $G$ be a finitely generated subgroup of the tame group and let $\mu$ be a symmetric atomic measure on $G$ satisfying the condition:
\begin{equation*}
\int_G \log( \deg(g)) d\mu(g) < + \infty. 
\end{equation*}
Then the degree exponents $\lambda_1(\mu) = \lambda_2(\mu) $ are positive if and only if $G$ contains two automorphisms with dynamical degree strictly larger than $1$ generating a free group of rank $2$.  
\end{thm_int}

Moreover, we also obtain the following classification.
\begin{cor_int} \label{cor_int_random_cor}When $\lambda_1 (\mu) = \lambda_2 (\mu) = 0$ then $G$ satisfies one of the following properties.
\begin{enumerate}
\item[(i)] The group $G$ is conjugated to a subgroup of the linear group $\O4$.
\item[(ii)] There exists a $G$-equivariant morphism $\varphi: Q \to \mathbb{A}^2 \setminus \{(0,0)\}$ where $G$ acts on $\mathbb{A}^2 \setminus \{(0,0) \}$ linearly.
\item[(iii)] The group $G$ contains an automorphism $h$ with $\lambda_1(h)> 1$ and there exists an integer $M$ such that  any automorphism $f \in G$ can be decomposed into $ g \circ h^p$ where $p$ is an integer and $g$  has a degree bounded by $M$.
\item[(iv)]  There exists a $G$-equivariant morphism $\varphi: Q \to \mathbb{A}^1$ where $G$ acts on $\mathbb{A}^1$ by multiplication and any automorphism of $G$ has dynamical degree $1$.
\end{enumerate} 
\end{cor_int}

In other words, the degree exponents detect whenever the random walk has a chaotic behavior. 
 These last two results essentially follow from a classification of the finitely generated subgroups of the tame group and a theorem due to Maher-Tiozzo \cite[Theorem 1.2]{maher_tiozzo} which asserts that a random walk on a subgroup $G$ of isometries of a $\CAT(0)$ space will drift to the boundary whenever $G$ contains two non-commuting hyperbolic elements. 
When this happens, we obtain using Theorem \ref{thm_int_degree_versus_distance} that the degree exponent is bounded below by a multiplicative factor of the drift and  is thus positive. 
Otherwise, we prove that $G$ preserves a vertex in the complex or a geodesic line. We then determine the degree sequences explicitly and conclude.
\medskip

If we pursue the analogy with the random walk on groups, it is natural to ask whether one can obtain a central limit theorem analog to the one for random products of  matrices (\cite[Theorem 0.7]{quint_random_walk_reductive}) or for random products of  mapping classes (\cite{dahmani_horbez}). We state it as follows.
\begin{conjecture} Take $\mu$ a symmetric atomic measure on the tame group. 
Then the limit  $$   \sigma^2 :=\lim_{n \rightarrow + \infty}   \dfrac{1}{n} \int_G  (\log \deg (g) - \lambda(\mu) n )^2 d\nu_n(g)  $$
exists where $\nu_n= \mu^{*n}$ denotes the $n$-fold convolution of $\mu$ and  the sequence of random variables 
$$ \dfrac{\log \deg (g_n) - \lambda(\mu) n}{\sqrt{n}} $$ converges to the normal distribution law $\mathcal{N}(0 , \sigma^2 )$.
\end{conjecture}

\subsection*{Structure of the paper}

In \S\ref{section_1}, we recall some general facts on the tame group  and then   review in \S\ref{section_square} the construction of the associated square complex.  In \S\ref{section_hyperbolic}, we focus on the global properties of the complex and exploit them to describe the degree sequences of particular automorphisms whose action  fix a vertex on the complex. We then state in \S\ref{section_estimates} the main valuative estimates needed for our proofs of Theorem \ref{thm_int_degree_growth} and Theorem \ref{thm_int_degree_versus_distance} which are presented in \S\ref{section_comparison}. Finally, we apply the previous result to deduce Theorem \ref{thm_random_walk} and Corollary \ref{cor_int_random_cor} in the last section. 

\subsection*{Acknowledgements}

The content of this paper is taken from part of the author's PhD thesis supervised by C. Favre. We are thanking him for his thorough reading. We are also grateful to C. Bisi, S. Cantat, JP. Furter, S. Lamy, A. Martin, M. Ruggiero, C. Urech for their comments and suggestions on the present paper and on the previous versions.
  

\section{General facts on the tame group of the quadric}
\label{section_1}

We work over an algebraically closed field $\K$ of characteristic zero. 
Take some affine coordinates $(x,y,z,t) \in \mathbb{A}^4$ and consider the smooth affine quadric threefold $\SL2$ given by:
\begin{equation*}
\SL2 := V( xt-yz - 1) \subset \mathbb{A}^4.
\end{equation*}
Let us also fix an open embedding $\mathbb{A}^4 \subset \mathbb{P}^4$ so that $\mathbb{A}^4 = V( w )$ in the homogeneous coordinates $[x,y,z,t,w]\in \Pg^4$.
\smallskip

In this section, we briefly describe the geometry of the affine quadric and give some preliminary properties of its elementary and orthogonal group of automorphism. 

%
%

\subsection{The geometry of a quadric threefold and its compactification in $\Pg^4$}
%
%

The affine variety $\SL2 \subset \mathbb{A}^4$ is a smooth quadric threefold. 
The Zariski closure $\overline{\SL2}$ of the affine quadric is also smooth  in $\Pg^4$ 
and has Picard rank one by Lefschetz hyperplane theorem. 
A birational map from $\overline{\SL2} $ to $\Pg^3$ is given by choosing a point $p_0 \in \overline{\SL2}$ and sending a point $p \in \overline{\SL2}$ to the intersection of the line $(pp_0)$ with a hyperplane in $\Pg^4$ which does not contain $p_0$. 
 
 We denote by $H_\infty := \overline{\SL2} \setminus \SL2$ the hyperplane section at infinity. 
 It is a smooth quadric surface given in homogeneous coordinates by:
\begin{equation*}
H_\infty := V( xt-yz) \subset \mathbb{P}^4.
\end{equation*}
We identify $H_\infty$ with $\Pg^1 \times \Pg^1$ by the isomorphism induced by the composition of the Segre embedding $\Pg^1 \times \Pg^1 \hookrightarrow \mathbb{P}^3  $ with the inclusion $\mathbb{P}^3 = V( w ) \hookrightarrow \mathbb{P}^4$.
In homogeneous coordinates, it is given by:
\begin{equation}
([\xi_0,\xi_1], [\eta_0, \eta_1]) \mapsto [\xi_0\eta_0, \xi_0\eta_1, \xi_1\eta_0, \xi_1\eta_1, 0]. 
\end{equation}
\smallskip 


Any line in $H_\infty$ of the form $ \{ \lambda\} \times \Pg^1$ (resp. $\Pg^1 \times \{ \lambda\}$) where $\lambda \in \Pg^1$ is said to be vertical (resp. horizontal).
\begin{figure}[!h]
\begin{tikzpicture} 
\draw[dotted] (0,0)-- (0,4) -- (5,4) -- (5, 0) --(0,0); 
\draw (0,0) node[below left ] {$ ([0,1], [0,1])=[0,0,0,1,0] \in \overline{\SL2}$};
\draw (0,4) node[above left ] {$ ([0,1], [1,0]) = [0, 0 ,1 ,0 ,0]$};
\draw (5,4) node[above right ] {$ ([1,0],[1,0]) = [1,0,0,0,0]$};
\draw (5,0) node[below right ] {$ ([1,0], [0,1])= [0,1,0,0,0]$};
\draw (0,3) -- (5,3);
\draw (5,3) node[right] {horizontal line $\Pg^1 \times \{ \lambda \}$} ;
\draw (3,0) -- (3,4);
\draw (3,0) node[below] {vertical line $\{ \lambda \} \times \Pg^1$} ;

%
%
%

\end{tikzpicture}
\end{figure}

\medskip


The two projection maps
$\pi_x: \SL2 \to \mathbb{A}^1$ and $\pi_y : \SL2 \to \mathbb{A}^1$ given by:
\begin{align*}
\pi_x :& (x,y,z,t) \in \SL2 \mapsto x  ,\\
\pi_y :& (x,y,z,t) \in \SL2 \mapsto y,
\end{align*}
induce algebraic fibrations which are trivial over $\mathbb{A}^1 \setminus \{ 0 \}$
 such that $\pi_x^{-1}(\mathbb{A}^1\setminus \{ 0\}) $ and $\pi_y^{-1}(\mathbb{A}^1\setminus \{ 0\})$ are isomorphic to $\mathbb{A}^1\setminus \{ 0\} \times \mathbb{A}^2$. 
 Observe that the fibers over
  $0$ are both isomorphic to $\mathbb{A}^1 \times \mathbb{A}^1\setminus \{ 0\}$ 
so that the fibrations are not locally trivial over a neighborhood of the origin. 
Observe that the intersection with $H_\infty $ of the closure of the fiber over $0$ in $\overline{\SL2}$ is the union of a vertical line and a horizontal line. The projection on the two components:
\begin{equation*}
\pi_{x,y} : (x,y,z,t) \rightarrow (x,y)
\end{equation*}
induces a surjective morphism $\pi_{x,y} : \SL2 \to \mathbb{A}^2 \setminus {(0,0)}$ which is also trivial over $\mathbb{A}^2 \setminus \{x= 0 \}$. 
\medskip

The affine quadric $Q$ carries naturally a volume form $\Omega$ which is the Poincaré residue of the rational $4$-form $dx \wedge dy \wedge dz \wedge dt/f$ along $\SL2$. More explicitly, $\Omega$ is defined by:
\begin{equation*}
 \Omega = \dfrac{dx \wedge dy \wedge dz}{x} {|_{\SL2}} =\dfrac{ dy \wedge d z \wedge dt}{ t } {|_{\SL2}}  =  \dfrac{dx \wedge dz \wedge dt}{ z } {|_{\SL2}}. 
\end{equation*}
One checks that $\Omega$ extends as a rational $3$-form $\overline{\Omega}$ on $\overline{\SL2}$ such that its divisors of poles and zeros satisfies $$\divi(\overline{\Omega}) = - 3 [H_\infty].$$




%
%

%
%
%
%
%
%
\subsection{The orthogonal group} \label{section_orthogonal}

A regular automorphism $f$ of $\SL2$ is determined by a morphism $f^\sharp$ of the $\C$-algebra $\C[Q]$  and hence by its image on the four regular functions $x,y,z,t$. 
If we denote by $f_x, f_y ,f_z,f_t \in \CSL2$ the image of $x,y,z,t$ by $f^\sharp$, it is convenient to adopt a matrix-like notation for $f$ as follows:
\begin{equation*}
f= \left ( \begin{array}{ll}
f_x & f_y \\
f_z & f_t
\end{array} \right ). 
\end{equation*}
Observe that $f_x f_t - f_z f_y = 1$ since $f^\sharp$ is a morphism of the $\C$-algebra $\C[Q]$ and that any such automorphism preserves the volume form $\Omega$ (up to a constant).

Denote by $q (x,y,z,t) = xt -yz$ the quadratic form defined on the vector space  $V=\C^4$. 
The group $\O4$ is the subgroup of linear automorphisms of $\C^4$ which leave the quadratic form $q$ invariant:
\begin{equation*}
\O4 = \{ f \in \GL_4(\C) \ | \ q \circ f = q \}. 
\end{equation*}  
An element of $\O4$ naturally defines an automorphism of the quadric hypersurface $Q$. As a consequence, we have that for any $f \in \O4$, $$f^* \Omega = \epsilon(f) \Omega,$$ where $\epsilon : \O4 \to \C^*$ is a morphism of groups. 
Since $\Omega$ is the Poincare residue of the form $dx \wedge dy \wedge
 dz \wedge dt / (xt-yz - 1)$ to $\SL2$, this implies that for any $f\in \O4$,   $\epsilon(f) $ is equal to the determinant of the endomorphism of $\C^4$ associated to $f$, hence $\epsilon(f) \in \{+1,-1 \}$.
 The subgroup $\SO_4$ is  the kernel of $\epsilon$ and has  index $2$ in $\O4$.

Observe that every element of $\O4$ extends as regular automorphism of $\overline{\SL2}$ which leaves the hyperplane at infinity invariant. 
In particular, the restriction map onto $H_\infty$ induces a morphism of groups from $\O4$ onto $\Aut(\Pg^1\times \Pg^1)$. 
The main properties of $\O4$ and $\SO_4$ are summarized in the following proposition. 
\begin{prop} \label{prop_O4_action_infinity} The following properties are satisfied:
\begin{enumerate}
\item[(i)] The group $\SO_4$ acts transitively on the set of horizontal and vertical lines at infinity respectively, and on the set of points at infinity.
\item[(ii)] Any element of $f \in \O4 $ which does not belong to $\SO_4$ exchanges the horizontal lines at infinity with the vertical lines at infinity. 
\item[(iii)] The following  sequence is exact. 
 \begin{equation*}
 \xymatrix{1 \ar[r]& \{ +1, -1\} \ar[r] & \O4\ar[r] & \Aut(\Pg^1\times \Pg^1)  \ar[r]& 1. }
 \end{equation*}
 
\item[(iv)] For any element $f \in \O4$, we have:
\begin{equation*}
f^* \Omega = \epsilon(f) \Omega,
\end{equation*}
where $\epsilon(f) \in \{ +1,-1\}$ and $\Ker(\epsilon) = \SO_4$. 
\end{enumerate}
\end{prop}

\begin{proof}
Observe that $(iii)$ follows directly from the following exact sequence:
\begin{equation*}
 \xymatrix{1 \ar[r]& \{ +1, -1\} \ar[r] & \O4\ar[r] & \PSO4  \ar[r]& 1, }
\end{equation*}
and the fact that $\PSO4 \simeq \PGL_2 \times \PGL_2$ which is given in \cite[Section 23.1]{fulton_harris}.

In particular, $(iii)$ directly implies $(i)$. 
\end{proof}

\subsection{Elementary transformations} \label{section_elementary}

The group $\EV$ (resp. $\EH$) of vertical (resp. horizontal) elementary transformations is defined by
\begin{align*}
\EV := &\left \{ \left ( \begin{array}{ll}
a x & b y \\
b^{-1} (z + x P(x,y)) & a^{-1}(t + yP(x,y)) 
\end{array} \right ) \ | \ P\in \C[x,y], a,b \in \C^* \right  \} ,\\
\EH :=& \left  \{\left ( \begin{array}{ll}
a x & b (y + x P(x,z)) \\
b^{-1} z  & a^{-1}(t + zP(x,z)) 
\end{array} \right ) \ | \ P\in \C[x,y], a,b \in \C^* \right  \}.
\end{align*}
 The terminology comes from the fact that these transformations are restrictions to the quadric of transformations of $\mathbb{A}^4$ of the form $$(x,y,z,t) \rightarrow (x, y + P(x) , z + R(x,y) , t + S(x,y,z))$$ where $P \in \C[x], R \in \C[x,y], S\in \C[x,y,z]$, 
which are elementary in the sense of \cite{shestakov}.

\medskip
Any automorphism in $\EV$ fix the two fibrations $\pi_x : (x,y,z,t) \rightarrow x$ and $\pi_y : (x,y,z,t) \rightarrow y$ and this geometric property characterizes the group $\EV$ (see \cite[Proposition 3.2.3.1]{these_hyperbolique}). An explicit computation proves that any elementary automorphism $f$ preserves the volume form $\Omega$:
\begin{equation*}
f^* \Omega = \Omega.
\end{equation*} 

We will not focus on the action of these elementary transformations on the compactification $\overline{\SL2}$.  For more details on the study of the birational transformations induced by these transformations, we refer to \cite[Chapter 3, Section 3.2.3]{these_hyperbolique}.

\smallskip


\medskip

\section{The square complex associated to the tame group}
\label{section_square}

The tame group, denoted $\tame$, is the subgroup of $\Aut(\SL2)$ generated by $\EV$ and $\O4$. 
It is naturally included in $\Bir(\Pg^3)$ since the variety $\overline{\SL2}$ is rational.

Observe that any tame automorphism $f$ fixes the volume form $\Omega$ up to a sign, i.e there exists a group morphism $\epsilon : \tame \to \{+1,-1 \}$ such that:
\begin{equation*}
f^* \Omega = \epsilon(f) \Omega.
\end{equation*}
This allows us to identify the kernel $\stame$ of $\epsilon$ as the group generated by $\SO_4$ and $\EV$. It has index $2$ in $\tame$.

The tame group $\tame$ is a strict subgroup of $\Aut(\SL2)$ (\cite{lamy_venereau}) and satisfies the Tits alternative  (see \cite[Theorem C]{bisi_furter_lamy}). 
The proof of this last fact is due to Bisi-Furter-Lamy and relies on the construction of a square complex on which the group acts by isometry. 
\medskip

The plan of this section is as follows. In \S \ref{subsection_complex} we detail the construction of the square complex due to Bisi-Furter-Lamy. Then, following the presentation in \cite{bisi_furter_lamy} we shall review in \S \ref{subsection_stab_III_II}, \S \ref{subsection_bass_serre} and \S \ref{subsection_link_type_I} the properties  of the stabilizer  of each vertex of this complex.
We will focus particularly on the stabilizer of the vertices which we call of type I in \S \ref{subsection_bass_serre} and \S \ref{subsection_link_type_I}, for which the analysis is more involving.
Finally, we state in \S \ref{subsection_squares_def} five technical lemmas on how four squares glue together near each vertices. 
As before, the situation is also more delicate near the vertices of type I and we need to introduce more terminology to describe the local geometry at those vertices. 
For a more detailed explanation of the results in this section, we refer to \cite[\S 2, \S 3.1]{bisi_furter_lamy} and to \cite[Chapter 3, Section 3.3]{these_hyperbolique}. 

%
%
%

%
\subsection{Construction of the square complex}
\label{subsection_complex}

The square complex, denoted $\mathcal{C}$, is a $2$-dimensional polyhedral complex where the cells of dimension $2$ are squares and where the cells of dimension $0$ and $1$ have some special markings. 
\medskip

We say that a regular function $f_1 \in \CSL2$
 is a component of an automorphism if there exists $f_{2} , f_3,f_4 \in \CSL2$ 
such that $f = (f_1 , f_2,f_3 , f_4)$ defines an automorphism of the quadric. 
One similarly defines the notion of components for a pair $(f_1,f_2)$ or for a triple $(f_1,f_2,f_3)$ of regular functions on $\SL2$ when they can be completed to a $4$-tuple defining an automorphism of the affine quadric.
\smallskip

We distinguish three types of vertices for the complex $\mathcal{C}$:
\begin{enumerate}
\item[$\bullet$] Type I vertices are equivalence classes of components $f_1 \in \CSL2$ of an automorphism, where two components $f_1$ and $f_2$ are identified if there exists an element $a \in \K^*$ such that $f_1 = a f_2$.
A vertex induced by a component $f_1 \in \CSL2$ is denoted by $[f_1]$.
\item[$\bullet$] Type II vertices are equivalence class of components $(f_1,f_2)$ of an automorphism  where $f_1=x \circ f,f_2= y \circ f \in \CSL2$ for $f \in \tame$ and where one identifies two components $(f_1,f_2)$ with $(g_1,g_2)$ if $(g_1,g_2) = (af_1+ b f_2, cf_1 + d f_2 )$ for some matrix: 
\begin{equation*}
\left (\begin{array}{ll}
a & b \\
c & d
\end{array} \right ) \in \GL_2.
\end{equation*}
A vertex induced by a component $(f_1,f_2)$ is denoted by $[f_1,f_2]$.
Denote by $f_3 = z \circ f$ and $f_4 = t \circ f$, the vertices $[f_1,f_2]$, $[f_1,f_3]$, $[f_2,f_4]$, $[f_3 , f_4]$ are well-defined since the automorphisms $(f_1,f_3,f_2,f_4), (-f_2,-f_4,f_1 , f_3) $ and $(-f_3,f_4 , -f_1,f_2)$ are also tame. 
Moreover, given a component $(f_1,f_2)$ and an invertible matrix $\left ( \begin{array}{ll}
a& b \\
c & d
\end{array} \right ) \in \GL_2$, there exists an automorphism $g$ such that $x\circ g = a f_1 + b f_2$ and $y\circ g = c f_1 + df_2$. 
Let us insist on the fact that on the contrary, there are no vertices of the form $[f_1,f_4]$ or $[f_2 , f_3]$. 

\item[$\bullet$] Type III vertices are equivalence classes of automorphisms $f \in \tame$ where two tame automorphisms $f$ and $g$ are equivalent if there exists $h \in \O4$ such that $f = h \circ g$. 
An equivalence class of $f \in \tame$ is denoted by $[f]$.
\end{enumerate}
\smallskip

The edges of the complex $\mathcal{C}$ are of two types:
\begin{enumerate}
\item[$\bullet$] Type I edges join a vertex of type I of the form $[f_1]$ with a vertex of type II of the form $[f_1,f_2]$ where $(f_1,f_2)$ are the components of a tame automorphism. 
\item[$\bullet$] Type III edges join a vertex of type II of the form $[f_1,f_2]$ with a vertex of type III $[f]$ where $(f_1,f_2)$ are the components of the automorphism given by $f$. 
\end{enumerate}
\smallskip 

The cells of dimension $2$ are squares containing two type II vertices of the form $[f_1,f_2]$, $[f_1,f_3]$, one vertex of type I given by $[f_1]$ and one vertex of type III given by $[f]$ where $(f_1,f_2,f_3)$ are the components of the automorphism $f \in \tame$. 
We have the following figure of a square. As in \cite{bisi_furter_lamy}, we adopt the following convention for the pictures:
the vertices of type I, II and III are represented by the symbol $\circ $, $\bullet$ and $\blacksquare$ respectively.
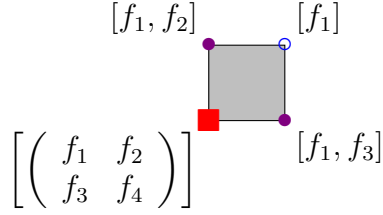
\begin{figure}[h!]  \label{fig_elem}
\centering
\begin{tikzpicture}
\draw[fill=lightgray] (0,0) -- (0,1) -- (1,1) --(1,0) -- (0,0);
\draw  (1,1 ) node[above right]    {$[f_1]$} node[blue]{$\circ$};
\draw  (0,0 ) node[below left ]    {$\left [ \left ( \begin{array}{ll}
f_1& f_2 \\
f_3 & f_4
\end{array} \right ) \right ]$} node[red]{$\blacksquare$};
\draw  (1,0 ) node[below right ]    {$[f_1,f_3]$} node[violet]{$\bullet$};
\draw  (0,1 ) node[above left ]    {$[f_1,f_2]$} node[violet]{$\bullet$};
\end{tikzpicture}
\end{figure}

The square complex $\mathcal{C}$ is obtained by the quotient of the disjoint union of all cells by the equivalence relation $ \sim $ where any two cells $C_1,C_2$ are identified along $C_1 \cap C_2$.
\medskip

Each square of the complex is endowed with the euclidean metric $d$ so that each square is isometric to $[0,1] \times [0,1]$. 
For any points $p$ and $q$ in $\mathcal{C}$, define by:
$$ d_{\mathcal{C}}(p,q) =  \inf \left  \{ \sum_{i=0}^N d(p_i , p_{i+1}) \right \} ,$$
 where the infimum is taken over all sequence of points $p_0=p, \ldots , p_N=q$ where $p_i$ and $p_{i+1}$ lie on the same square in $\mathcal{C}$. 
As any cell of the complex $\mathcal{C}$ has only finitely many isometries, 
we may apply a general result from \cite[Section I.7]{bridson_haefliger} and conclude that   
the function $d_\mathcal{C}$ induces a metric on the complex and turns $(\mathcal{C}, d_\mathcal{C})$ into a complete metric space. 
We will explain in \S \ref{section_hyperbolic} the global properties on the complex induced by this metric.   

\medskip

%

Let us define the action of the tame group $\tame$ on the complex $\mathcal{C}$.
Pick any two automorphisms $f, g \in\tame$. 
We define the action of $g$ on the each vertices of the complex by setting:
 \begin{align*}
 g \cdot [f_1] &:= [f_1 \circ g^{-1}], \\
g \cdot [f_1,f_2] & := [f_1 \circ g^{-1}, f_2 \circ g^{-1}],\\
g\cdot [f] & := [f \circ g^{-1}].
 \end{align*}
 The action on vertices induces a morphism of the square complex which preserves the type of vertices and edges and preserves the distance. 


Recall that the subgroup $\stame$ generated by $\SO_4$ and elementary transformations has index $2$ in $\tame$.
\begin{defi}
  An edge $E$ of the complex is called \textbf{horizontal} (resp. \textbf{vertical}) if there exists an element $f\in \stame$ such that $f \cdot E$ is equal to the edge joining $[x,y]$ with $[x]$ (resp. $[x,z]$ with $[x]$) or to the edge between $[\Id]$ and $[x,z]$ (resp. $[\Id]$ and $[x,y]$).
  \end{defi} 
  
We will see that the set of vertical and horizontal edges form a partition of the set of edges (see (iii) and (iv) of Proposition \ref{prop_action_tame_complex}). 

\subsection{Stabilizer of vertices of type III, II and the properties of the action}
\label{subsection_stab_III_II}

In this section, we shall first review the properties of the stabilizer of type II and III vertices then deduce from these the global properties of the action of the group on this complex. To do so, we shall exploit the relationship between the local geometry near each vertices and their respective stabilizer subgroups. 
The geometry near a given vertex $v$ is encoded in its link   $\mathcal{L}(v)$ which is constructed as follows.
The vertices of $\mathcal{L}(v)$ are in bijection with the vertices $v'$ such that $[v,v']$ is an edge of the complex $\mathcal{C}$. And we draw an edge joining $v'$ and $v''$ in $\mathcal{L}(v)$ if the vertices $v,v',v''$ belong to the same square.

Observe that the action of the tame group on the vertices of type III is transitive. As a result, we shall focus on the stabilizer subgroup of the vertex $[\Id]$, which is by construction $\O4$. Its action on the complex induces an action on the link $\mathcal{L}([\Id])$.
  
\begin{prop} \label{prop_correspondence_vertex_III}
 The link $\mathcal{L}([\Id])$ is a complete bipartite graph
and there exists an $\O4$-equivariant bijection between the set of vertices of the link $\mathcal{L}([\Id])$ to the set of lines at infinity such that 
  the vertices which belong to a vertical (resp. horizontal) edge of type \textsc{III}  are mapped to vertical (resp. horizontal) lines at infinity in $H_\infty$. Moreover, this bijection induces an $\O4$-equivariant bijection from the edges of $\mathcal{L}([\Id])$ to the set of points at infinity $H_\infty$.

\end{prop}

\begin{rem} Observe that Proposition \ref{prop_O4_action_infinity} and Proposition \ref{prop_correspondence_vertex_III} imply that the group $\O4$ acts faithfully and transitively on the link $\mathcal{L}([\Id])$.
\end{rem}  
  
\begin{proof} 
We identify two types of vertices in the link of $[\Id]$, the vertices which belong to a horizontal edge containing $[\Id]$ or those which are contained in a vertical edge containing $[\Id]$.
 
We define a map $\varphi$ from the vertices of the link $\mathcal{L}([\Id])$ to the set of lines in $H_\infty$.
Take a vertex $v$ in the link $\mathcal{L}([\Id])$ and a component $(f_1,f_2)$ such that $[f_1,f_2] = v$. 
By definition, there exists an element $f \in \O4$ such that $f_1 = x\circ f$ and $f_2 = y \circ f$ since the stabilizer of $[\Id]$ is $\O4$. 
The zero locus $V(f_1) \cap V(f_2) \cap H_\infty$ in $\overline{\SL2}$ is the line at infinity corresponding to the preimage of $\{ x = y =0\} \cap \overline{\SL2}$ by $f$. Observe that the line $V(f_1) \cap V(f_2) \cap H_\infty$ does not depend on the choice of representative of the equivalence class $v$ since any two other component in the same class defines the same homogeneous ideal $\langle f_1,f_2, xt - yz - w^2 \rangle$. We thus define $\varphi(v)$ to be the line $V(f_1) \cap V(f_2) \cap H_\infty$. 
Observe that if $v $ is a vertex of type II such that the edge containing $v$ and $[\Id]$ is vertical, then  $f \in \SO_4$. Hence the line at infinity $V(x \circ f ) \cap V(y\circ f) \cap H_\infty$ is vertical. 
Observe also that $\varphi$ is naturally $\O4$-equivariant.
The same argument holds for the vertices of type II which belong to horizontal edges containing $[\Id]$. 
\smallskip

Let us prove that the map $\varphi$ is surjective. Consider a vertical line $L \subset H_\infty$ at infinity, then there exists by Proposition \ref{prop_O4_action_infinity}.$(i)$ an automorphism $f$ in $\SO_4$ such that the image of the vertical line at infinity given by $[0,1] \times \Pg^1$ is $L$. 
Since $\varphi([x,y])$ corresponds to the line $[0,1]\times \Pg^1$, the vertex of type II $[x\circ f, y \circ f]$ defines a component of an automorphism which belongs to the link $\mathcal{L}([\Id])$ such that $\varphi([x\circ f, y\circ f])= L$.
Hence, $\varphi$ is surjective.  

\smallskip

Let us prove that $\varphi$ is injective. Consider two vertices $v_1,v_2$ such that their image by $\varphi$ is equal, we prove that $v_1=v_2$. 
Consider two components $(f_1,f_2), (g_1,g_2)$ such that $[f_1,f_2]=v_1$ and $[g_1,g_2]=v_2$. We must prove that $(f_1,f_2)$ and $(g_1,g_2)$ belong to the same equivalence class.
%
By symmetry, we can suppose that the line $\varphi(v_1)$ is vertical. Hence, there exists $f,g\in \SO_4$ such that $f_1 = x \circ f, g_1 = x \circ g , f_2 = y \circ f$ and $g_2 = y \circ g$. In particular, this implies that $f \circ g^{-1}$ fixes the vertical line at infinity given by $\{[0,1]\} \times \Pg^1$. Using Proposition \ref{prop_O4_action_infinity}.$(iii)$, we conclude that $f \circ g^{-1}$ is of the form
\begin{equation*}
f \circ g^{-1} = \left ( \begin{array}{ll}
a x + b y & c x + d y \\
a' z + b' t & c' z + d' t  
\end{array} \right ),
\end{equation*}
where the matrices $ \left ( \begin{array}{ll}
a & b \\
c & d
\end{array} \right ) $, $\left ( \begin{array}{ll}
d' & -b' \\
-c' & a'
\end{array} \right ) \in \M_2(\C)$ satisfy
\begin{equation*}
\left ( \begin{array}{ll}
a & b \\
c & d
\end{array} \right ) \cdot \left ( \begin{array}{ll}
d' & -b' \\
-c' & a'
\end{array} \right ) = \left ( \begin{array}{ll}
1 & 0 \\
0& 1
\end{array} \right ).
\end{equation*}
In particular, this implies that the components $(f_1,f_2)$ and $(g_1,g_2)$ are equivalent since $f_1 = a g_1 + b g_2, f_2 = c g_1 + d g_2$.

\smallskip
One similarly defines a bijection from the edges of the link $\mathcal{L}([\Id])$ to $H_\infty$. The link is complete since  a horizontal and a vertical line in $H_\infty$ always intersect at a point in $H_\infty$, hence for any vertices $v_1,v_2$ in $\mathcal{L}([\Id])$ which are mapped by $\varphi$ to a vertical and a horizontal line respectively, there exists an edge joining $v_1$ and $v_2$.
\end{proof}

\begin{prop}\label{prop_stab_III} The following properties are satisfied. 
\begin{enumerate}
\item[(i)]   The stabilizer of a vertex of type \textup{III} in $\stame$ is conjugated in $\stame$ to $\SO_4$.
\item[(ii)] The stabilizer of an edge of type \textsc{III} is conjugated in $\tame$ to the subgroup:
\begin{equation*}
\left ( \begin{array}{ll}
ax + b y & c x + d y \\
a' z + b' t & c' z + d' t 
\end{array} \right ) 
\end{equation*}
where the matrices $ \left ( \begin{array}{ll}
a' & b' \\
c' & d'
\end{array} \right ) \in \GL_2 $ and
\begin{equation*}
\left ( \begin{array}{ll}
a & b \\
c & d
\end{array} \right ) = \dfrac{1}{a' d' - b' c'} \left ( \begin{array}{ll}
a' & b' \\
c'& d'
\end{array} \right ).
\end{equation*}

\item[(iii)] The stabilizer of a $1\times 1$ square is conjugated in $\tame$ to:
\begin{equation*}
\left \lbrace \left ( \begin{array}{ll}
a x & b ( y + c x) \\
b^{-1} (z + dx) &  a^{-1}(t + cz + d y + dc x) 
\end{array} \right ) \ | \  (a,b,c,d) \in \C^* \times \C^* \times  \C \times \C  \right \rbrace  \rtimes  \left \lbrace \left ( \begin{array}{ll}
x & z \\
y & t 
\end{array} \right ), \Id \right \rbrace
\end{equation*}

\item[(iv)] The pointwise stabilizer of the union of the four squares containing $[\Id]$ and $[x], [y],[z]$ and $[t]$ respectively is equal to:
\begin{equation*}
\left \lbrace \left ( \begin{array}{ll}
ax & b y \\
b^{-1} z & a^{-1} t 
\end{array} \right ) \ | \ a,b \in \C^* \right \rbrace
\end{equation*}
\end{enumerate}
\end{prop}

\black 

\begin{proof}
Observe that $(i)$ follows directly from the definition of the definition. Moreover, the next  assertions $(ii), (iii)$ and $(iv)$ are exactly the content of \cite[Lemma 2.5 (2), Lemma 2.7 and Lemma 2.11]{bisi_furter_lamy}. 
\end{proof}

We focus on the stabilizer subgroups of vertices of type II.

\begin{prop} \label{prop_stab_II}
The following properties are satisfied.
\begin{enumerate}

\item[(i)] The stabilizer of a vertex of type \textup{II} in $\tame$ is conjugated in $\tame$ to the semi-direct product $\EV \rtimes \GL_2$ where the group $\GL_2$ is identified with the stabilizer of the edge of type \textsc{III} joining $[\Id]$ and $[x,y]$.
\item[(ii)] The stabilizer of a vertical edge of type \textsc{I} is conjugated in $\stame$ to the subgroup:
 $$\EH \rtimes \left \lbrace \left ( \begin{array}{ll}
a x & d^{-1} y \\
d z + cx &  at + c a^{-1} d^{-1} y 
\end{array} \right ) \ | \  (a,c,d) \in \C^* \times \C \times \C^*  \right \rbrace .$$
\item[(iii)] The pointwise stabilizer of the geodesic segment of length $2$ joining the vertices $[f_1]$, $[f_3]$ and $[f_1,f_3]$ where $f = (f_1,f_2,f_3,f_4) \in \stame$ is conjugated in  $\stame$ to:
\begin{equation*}
\EH \rtimes \left  \{ \left ( \begin{array}{ll}
a x & b y \\
b^{-1} z & a^{-1} t 
\end{array} \right ) , a,b \in \C^*\right \}  .
\end{equation*}

\black 

\end{enumerate}
\end{prop}

\begin{proof}
Assertion $(i) $, $(ii)$ and $(iii)$ are given in \cite[Lemma 2.3, Lemma 2.5 (1) and Lemma 2.6 (1)]{bisi_furter_lamy} respectively.
\end{proof}

From the description of the previous stabilizer subgroups, we state the following consequences on the action of the group on this complex. 

\begin{prop} \label{prop_action_tame_complex}	
 The tame group $\tame$ acts by isometry on the complex $\mathcal{C}$ and this action satisfies the following properties. 
\begin{enumerate}
\item[(i)] The action preserves the types of vertices and the types of edges.
\item[(ii)] The action is faithful and transitive on the set of vertices of type \textsc{I} , \textsc{II} and \textsc{III} respectively.
\item[(iii)] The subgroup $\stame$ acts transitively on the set of vertical (resp. horizontal) edges of type \textsc{I} and \textsc{III}.
\item[(iv)] Any automorphism $f \in \tame$ which does not belong to the subgroup $ \stame$ sends a vertical edge to a horizontal edge of the same type.
\item[(v)] The subgroup $\stame$ acts transitively on the set of $1\times 1$ squares.
\item[(vi)] The group $\tame$ acts transitively on the union of $4$ squares which is isometric to $[0,2]\times[0,2]$ and which contains a common vertex of type \textsc{III}.
\end{enumerate}
\end{prop}
\begin{proof}   
The transitive of the action on the set of vertices of type I, II and III and assertions $(i)$ and $(iv)$  follow from \cite[Lemma 2.1 and Lemma 2.4]{bisi_furter_lamy}. 

To prove $(ii)$, we need to explain why the action is also faithful.
Observe that if a tame automorphism fixes every vertices of type III, or type II or type I, then it fixes the whole complex since every vertex of type III (resp. type II or I) is the middle point of a geodesic segment joining type I or type II points. 
Then the faithfulness follows from the faithfulness of the action on the link $\mathcal{L}([\Id])$.

The assertions $(iii), (v)$ and $(vi)$ are exactly the content of \cite[Lemma 2.4, Lemma 2.7, Corollary 2.10]{bisi_furter_lamy} respectively.
\end{proof}

\subsection{Bass-Serre tree associated to plane automorphisms}
\label{subsection_bass_serre}

 We consider the field $\Ks = \C(x) $.
We define the graph $\mathcal{T}_{\C(x)}$ which is a bipartite metric graph.
\begin{enumerate}
\item Vertices of type I are equivalence classes of components $f_1 \in \C(x)[y,z]$ of plane automorphisms where one identifies two components $f_1$ and $g_1$ if there exists $a \in \C(x)^*$ and $b \in \C(x)$ such that $f_1 = a g_1 + b$. An equivalence class induced by a component $f_1$ is denoted $[f_1]$.
\item Vertices of type II are equivalence classes of automorphisms $f $ where one identifies two automorphisms $f$ and $ g$ if there exists an affine automorphism $h$ such that $f = h \circ g$ i.e there exists a matrix $M \in \GL_3(\C(x))$ of the form:
\begin{equation*}
M = \left ( \begin{array}{lll}
a & b & c \\
a' & b' & c'\\
 0& 0 & 1
\end{array} \right )
\end{equation*}
such that $(f_1,f_2) = (a g_1 + bg_2 + c , a' g_1 + b' g_2 + c)$ where $f=(f_1,f_2)$ and $g=(g_1,g_2)$. An equivalence class induced by a plane automorphism $f=(f_1,f_2)$ is denoted $[f_1,f_2]$

\item Edges link a vertex $v_1$ of type I  with a vertex $v_2$ of type II if there exists a polynomial automorphism $f = (f_1,f_2)$ such that  $[f_1]=v_1$ and $[f_1,f_2]= v_2$.
\end{enumerate}  
\smallskip

We endow this graph $\mathcal{T}_{\C(x)}$ with the distance such that each edge is of length $1$. 
This graph $\mathcal{T}_{\C(x)}$ is thus a complete geodesic metric space.

The action of an automorphism $g \in \mathbb{A}^2_{\C(x)}$ on $\mathcal{T}_{\C(x)}$ is defined as follows:
$$g \cdot [f_1] = [f_1 \circ g^{-1}],$$ 
and $$g \cdot [f_1,f_2] = [f_1 \circ g^{-1}, f_2 \circ g^{-1}]$$
for any automorphism $f = (f_1,f_2) \in \Aut(\mathbb{A}^2_{\C(x)})$.
\smallskip

A classical theorem from Jung (\cite{jung}) proves that the graph $\mathcal{T}_{\C(x)}$ is a tree and that the group of plane automorphism acts faithfully, by isometry and transitively on the set of type I and II vertices respectively.
%



\subsection{Link over a vertex of type I}
\label{subsection_link_type_I}
In this subsection, we study the link over the vertex of type I given by $[x]$. 
Observe that the stabilizer subgroup of the vertex $[x]$ acts naturally in the link of the vertex $[x]$.

\begin{lem} \label{lem_stab_x_link}The group $\stab([x])$ acts transitively, faithfully on the set of vertices of type I in the link of $[x]$. 
\end{lem}

\begin{proof}
By Proposition \ref{prop_action_tame_complex}.(v), the group $\stame$ acts transitively on the set of $1\times 1$ squares and since a $1\times 1$ square containing $[x]$ defines an edge in the link $\mathcal{L}([x])$, the induced action of $\stab([x])$ is transitive on the edges of the link $\mathcal{L}([x])$.
Observe also that the involution $\sigma : (x,y,z,t) \mapsto (x,z,y,t)$ induces an action on the link which exchanges the vertices $[x,y]$, $[x,z]$ in the link and fixes the edge between these two vertices. This proves that the action of the stabilizer $\stab([x])$ is transitive on the link of $[x]$.
\smallskip

Let us prove that the action is faithful. Suppose $f\in \stab([x])$ acts by the identity map in the link over $[x]$, then in particular, $f$ must fix pointwise the square containing $[\Id]$ and $[x]$. 
By Proposition \ref{prop_stab_III}.(iii), $f $ is of the form:
\begin{equation*}
f=\left ( \begin{array}{ll}
a x & d^{-1}( y + bx ) \\
d (z + cx) &  a^{-1}(t + c  y + bz + bcx ) 
\end{array} \right ),
\end{equation*}
where $a,d \in \C^*$ and $b,c \in \C$. 
Since $f$ must also fix the vertices of type II $[x, y + xP(x)]$ and $[x, z + xP(x)]$ where $P\in\C[x]$, we have that $a = d = 1$ and $c = b = 0$ as required.
\end{proof}

In the following arguments, we will use the fact that the link $\mathcal{L}([x])$ is connected (\cite[Lemma 3.2]{bisi_furter_lamy}), which is a highly non-trivial argument which relies deeply on the reduction theory inspired by the work of Shestakov-Umirbaev (see \cite[Corollary 1.5]{bisi_furter_lamy}).

Recall that the general fiber of the projection $\pi_x : \SL2 \to \mathbb{A}^1$ defined in \S \ref{section_elementary} is isomorphic to $\mathbb{A}^2$. 
We fix an identification of $\pi_x^{-1}(\mathbb{A}^1 \setminus\{ 0\})$ with $\mathbb{A}^1 \setminus \{ 0\} \times \mathbb{A}^2$ given by:
\begin{equation} \label{eq_identification_fibration}
(x,y,z) \mapsto (x,y,z,(yz+1)/x).
\end{equation} 
The relationship between the stabilizer of the vertex $[x]$ and $\Aut(\mathbb{A}^2_{\C(x)})$ is realized explicitly as follows. 
\medskip

Denote by $\mathcal{L}([x])'$ the first barycentric division of $\mathcal{L}([x])$.
We shall define a simplicial map $\pi : \mathcal{L}([x])' \to \mathcal{T}_{\C(x)}$ as follows.

Let $v$ be a vertex of type II in $\mathcal{C}$ which defines a vertex in the link of $[x]$, then since the action of $\stab([x]) $ on the link $\mathcal{L}([x])$ is transitive by Lemma \ref{lem_stab_x_link}, there exists an element $f \in \stab([x])$ such that $f \cdot [x,y] = v$.
 Since $f$ naturally fixes the fibration $\pi_x$, under the identification $\pi_x^{-1} (\mathbb{A}^1 \setminus \{0 \}) \simeq\mathbb{A}^1 \setminus \{0 \} \times \mathbb{A}^2$ given by \eqref{eq_identification_fibration}, the regular map $f$ is given by:
 \begin{equation*}
 (x,y,z) \mapsto (x\circ f, y\circ f, z\circ f).
\end{equation*}
Under this identification, $(y\circ f , z \circ f)$ induces an element of $\mathbb{A}^2_{\C(x)}$.
We thus define 
$$\pi(v) = [y \circ f] \in \mathcal{T}_{\C(x)}.$$
Observe that $\pi(v)$ does not depend on the choice of $f$. Indeed, if $g \in \stab([x])$ is another automorphism such that $g \circ [x,y] = v$, then by Proposition \ref{prop_stab_III}.$(ii)$, the composition  $g\circ f^{-1}$ satisfies:
\begin{equation*}
g \circ f^{-1}  \in \EH \rtimes \left \lbrace \left ( \begin{array}{ll}
a x & d^{-1} y \\
d z + cx &  at + c a^{-1} d^{-1} y 
\end{array} \right ) \ | \  (a,c,d) \in \C^* \times \C \times \C^*  \right \rbrace ,
\end{equation*}  
hence $[y \circ g]= [y \circ f] \in \mathcal{T}_{\C(x)}$.
\smallskip 

Let $m \in \mathcal{L}([x])'$ be the middle point of an edge $E$ of $\mathcal{L}([x])$ and let $m_0$ be the middle point of the geodesic joining $[x,y]$ and $[x,z]$ in $\mathcal{L}([x])'$. 
Since the action of $\stab([x])$ in the link $\mathcal{L}([x])$ is transitive by Lemma \ref{lem_stab_x_link}, there exists an element $f \in \stab([x])$ such that $f \circ m_0 = m$.
Since $f$ naturally fixes the fibration $\pi_x$, it induces an automorphism of $\pi_x^{-1}(\mathbb{A}^1\setminus \{ 0\})$ and under the identification given by \eqref{eq_identification_fibration}, it is of the form
\begin{equation*}
( x, y ,z) \mapsto (x \circ f , y \circ f, z \circ f).
\end{equation*}
 We thus define:
\begin{equation*}
\pi(m) = [y \circ f , z \circ f].
\end{equation*}
Observe also that $\pi(m)$ does not depend on the choice of $f$. If $g \in \stab([x])$ such that $g \cdot m_0 = m$, then $g$ and $f$ differ by an element which belongs to the subgroup:
\begin{equation*}
\left \lbrace \left ( \begin{array}{ll}
a x  & b ( y + c x) \\
b^{-1}(z + d x) & a^{-1}(t + dy + cz + cd x)
\end{array} \right ) \ | \ a,b \in \C^*, c,d\in \C \right \rbrace \rtimes \left \{ \Id , \left ( \begin{array}{ll}
x & z \\
y & t
\end{array} \right )\right  \},
\end{equation*}
hence $[y \circ g , z \circ g] = [y \circ  f, z \circ  f] \in \mathcal{T}_{\C(x)}$ and $\pi(m)$ is well-defined.
  
If $E$ is an edge of $\mathcal{L}([x])'$ of length $1$, then we define the image of $E$ by $\pi$ as the geodesic joining the image of the endpoints of $E$ by $\pi$. As a result,the map $\pi$ is a simplicial map between $\mathcal{L}([x])'$ and $\mathcal{T}_{\C(x)}$ such that the action of $\stab([x])$ descends into an action on $\mathcal{T}_{\C(x)}$ (one can prove that $\pi : \mathcal{L}([x])' \to  \mathcal{T}_{\C(x)}$ is the unique $\stab([x])$-equivariant map for which $\pi([x,y]) = [y]$ and $\pi([x,z]) = [z]$).

\begin{defi}  The subgroup $A_{[x]}$ is the intersection of $\stame$ with the kernel of the morphism induced by the $\stab([x])$-equivariant simplicial map $\pi: \mathcal{L}([x])' \to \mathcal{T}_{\C(x)}$. 
\end{defi}


\begin{prop} \label{prop_link_bass_serre} Denote by $m \in \mathcal{L}([x])'$ the middle point between the point $[x,y]$ and $[x,z]$. 
The  simplicial map $\pi : \mathcal{L}([x])' \to \mathcal{T}_{\C(x)}$ satisfies the following properties.
\begin{enumerate}
\item[(i)]  The image of the edge between the point $[x,y]$ and $m$ by $\pi$ is a fundamental domain of $\mathcal{T}_{\C(x)}$.
\item[(ii)] The image $\pi(\mathcal{L}([x])')$ is a subtree of $\mathcal{T}_{\C(x)}$.
\item[(iii)] The preimage by $\pi$ of the segment of length $2$ joining $[z]$ and $[y]$ is a bipartite graph.
\item[(iv)] The subgroup $A_{[x]} \subset \stab([x]) \cap \stame$ is generated by elements of the form:
\begin{equation*}
\left ( \begin{array}{ll}
 ax & b (y + x P(x)) \\
b^{-1}(z + x S(x)) & a^{-1}(t + z P(x) + y S(x) + x P(x)S(x)) 
\end{array} \right ),
\end{equation*}
where $P, S \in \C[x]$ and $a ,b \in  \C^*$.
\end{enumerate}
\end{prop}
\begin{proof}
Assertion $(i)$, $(ii)$ and $(iii)$ are the content of \cite[Lemma 3.4 (1), Lemma 3.5 (1) and (2)]{bisi_furter_lamy} respectively.
\medskip

Let us prove statement $(iv)$.
Let us denote by $\phi : \stab([x]) \to \Aut(\mathbb{A}^2_{\C(x)})$ the morphism of groups induced by the simplicial map $\pi : \mathcal{L}([x])' \to \mathcal{T}_{\C(x)} $.
It is clear that any element of the form:
\begin{equation*}
\left ( \begin{array}{ll}
 ax & b (y + x P(x)) \\
b^{-1}(z + x S(x)) & a^{-1}(t + z P(x) + y S(x) + x P(x)S(x)) 
\end{array} \right ),
\end{equation*}
where $P, S \in \C[x]$ and $a ,b \in  \C^*$ induces the identity on $\mathcal{T}_{\C(x)}$.
Conversely, we prove that any element of $ A_{[x]}$ has this form.  Pick $g \in A_{[x]}$, since $\phi(g)$ fixes every vertices of $\mathcal{T}_{\C(x)}$, $\phi(g) $ is an affine automorphism of $\mathbb{A}^2(\C(x))$. As $\phi(g)$ fixes every vertex of type I and since it belongs to the image of $\phi$, the plane automorphism $\phi(g)$ must be of the form: 
\begin{equation*}
\phi(g) = (y,z) \rightarrow ( b (y + xP(x)), c (z + x S(x))), 
\end{equation*}
where $P,S \in \C[x]$ and where $b, c \in \C^*$.  
In particular, as $g \in\tame$, $b = c^{-1}$ and $g$ is of the form:
\begin{equation*}\left (
\begin{array}{ll}
a x & b( y + xP(x)) \\
b^{-1} (z +  x S(x)) & a^{-1}(t + z P(x) + yS(x) + xP(x) S(x))
\end{array}\right ),
\end{equation*}
proving $(iv)$.
\end{proof}

\subsection[Five lemmas]{Five technical consequences on the local geometry at each vertex}
\label{subsection_squares_def}
We say that a subset  $S \subset \mathcal{C}$ is a $2\times 2$ square of $\mathcal{C}$ if $S$ is the union of four distinct $1\times 1$ squares such that $S$ isometric to  $[0,2]\times [0,2]$. 
Moreover, we say that a $2\times 2$ square is centered on a vertex $v$ if 
the vertex $v$ corresponds to the image of the point $(1,1)$ by an isometry from $[0,2] \times [0,2]$ to $S$. 
 
Two $1 \times 1$ (resp. $2 \times 2$) squares $S, S'$ are said to be \textbf{adjacent} if their union $S \cup S'$ is isometric to $[0,2] \times [0,1]$ (resp. $[0,4]  \times [0,2]$). Two $1\times 1$ squares $S$ and $S'$ are adjacent along a vertical (resp. horizontal) edge if they are adjacent and their intersection $S\cap S'$ is a vertical edge (resp. horizontal).

Two $1\times 1$ (resp. $2\times 2$) squares $S_1$ and $S_2$ are said to be \textbf{adherent} if they are not adjacent but their intersection is reduced to a vertex. 
If a vertex $v \in \mathcal{C}$ belongs to the intersection of two adherent squares $S_1 \cap S_2$, then $S_1$ and $S_2$ are said to be adherent along the vertex $v$. 

We say that two  $1 \times 1$ (resp. $2 \times 2$) squares $S , S'$ are \textbf{flat}  if there exists two $1\times 1$ (resp. $2 \times 2$) squares $S_1,S_2$ such that the union $S_1 \cup S_2 \cup S\cup S'$ is isometric to $[0,2] \times [0,2]$ (resp. $[0,4]\times [0,4]$).
Similarly, three $1\times 1$ (resp. $2 \times 2$) squares  are flat if we can find another $1\times 1$ (resp. $2\times 2$) square such that their union is isometric to $[0,2] \times [0,2]$ (resp. $[0,4]\times[0,4]$). 

We will prove that three $1\times 1$ squares  $S_1, S_2,S_3$ such that $S_1$ and $S_2$, $S_2$ and $S_3$ are adjacent and contain a common vertex of type II or III are necessarily flat (see Lemma \ref{lem_complete_square_2} and Lemma \ref{lem_complete_square_3} below). 
However, this property does not necessarily hold when the squares contain a common vertex of type I (see Lemma \ref{lem_complete_square_1} below), we prove that they are either flat or contained in a spiral staircase. We explain this terminology below.

%
\medskip

A collection $(S,S')$ of $1\times 1$ or $2\times 2$ squares is contained in a \textbf{vertical spiral staircase} (see \ref{ex_spiral} for an example) if they contain a common vertex $v$ of type I and such that  any minimal the sequence $S_1 = S , \ldots , S_k=S'$ of squares connecting $S$ to $S'$ satifies the following conditions:
\begin{enumerate}
\item for all integer $i \leqslant k-1$,  the squares $S_i$ and $S_{i+1}$ are alternatively adjacent along a vertical or horizontal edge containing $v$;
\item  any three consecutive squares $(S_i,S_{i+1} , S_{i+2})$ for $i\leqslant k-2$ is not flat. 
\item the first two squares $S_1$ and $S_2$ are adjacent along a horizontal edge containing $v$;
\end{enumerate}
Similarly, one defines a \textbf{horizontal spiral staircase} requiring that the first two squares in a minimal sequence are adjacent along a  vertical edge.
When two squares $S,S'$ are flat, then the collection $(S,S')$ is not contained in a spiral staircase.

\begin{ex}\label{ex_spiral} Consider $P_1 , P_2, P_3 \in \C[x,y] \setminus \C[x]$, denote by $S$ the square containing $[x]$ and $[\Id]$ and $S'$ the square containing $[x]$ and $[f]$ where $f \in \tame$ is given by:
\begin{equation*}
f = \left ( \begin{array}{ll}
x & y+ x P_1(x,y) + xP_3(x, z + xP_2(x,y+ xP_1(x,y))) \\
z + xP_2(x, y + xP_1(x,y)) & f_4
\end{array} \right ),
\end{equation*} 
where $f_4 = t +  y (P_1(x,y) + P_3(x, z + xP_2(x,y+ xP_1(x,y)))) +  y P_2(x, y + xP_1(x,y)) +  x(P_1(x,y) + P_3(x, z + xP_2(x,y+ xP_1(x,y))))P_2(x, y + xP_1(x,y)))$.
Then the pair $(S, S')$ is contained in a horizontal spiral staircase and one has the following figure:
\begin{center}
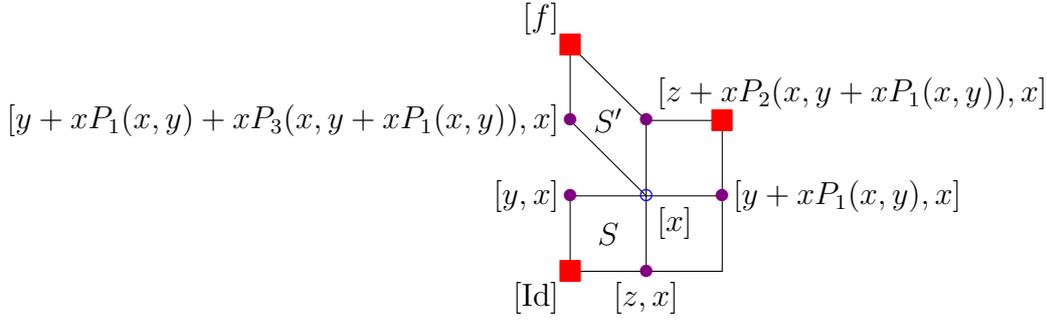
\begin{figure}[!h]
\begin{tikzpicture}
\draw (0,0)--(1,0) --(1,1) -- (0,1) --cycle;
\draw (1,0) -- (2,0) -- (2,1) --(1,1);
\draw (2,1) -- (2,2) --(1,2)--(1,1);
\draw (1,2)--(0,3)--(0,2)--(1,1);

\draw (0, 1 ) node[left] {$   [y,x] $} node[violet] {$\bullet$};
\draw (1, 0 ) node[below] {$   [z,x] $} node[violet] {$\bullet$};
\draw (2, 1 ) node[right] {$   [y+ xP_1(x,y),x] $} node[violet] {$\bullet$};
\draw (1, 2 ) node[above right] {$   [z+ xP_2(x, y+ xP_1(x,y)),x] $} node[violet] {$\bullet$};
\draw (0, 2 ) node[left] {$   [y + xP_1(x,y) + xP_3(x, y+ xP_1(x,y)),x] $} node[violet] {$\bullet$};

\draw (1, 1 ) node[below right] {$   [x] $} node[blue] {$\circ$};
\draw (0, 0 ) node[below left] {$   [\Id] $} node[red] {$\blacksquare$};
\draw (0, 3 ) node[above left] {$   [f] $} node[red] {$\blacksquare$};
\draw (2, 2 ) node[red] {$\blacksquare$};

\draw (0.5,2) node {$S'$};
\draw (0.5,0.5) node {$S$};
\end{tikzpicture}
\caption{Example of horizontal spiral staircase.}
\end{figure}
\end{center}
\end{ex}
 
The next lemmas describe when three squares containing a common vertex are flat.
\begin{lem} \label{lem_complete_square_3} 
Let $v$ be a vertex of type \textsc{III} and let $S_1,S_2,S_3$ be three distinct $1\times 1 $ squares such that $S_1$ is adjacent to $S_2$ along an edge containing $v$, and $S_2$ is adjacent to $S_3$ along an edge containing $v$. Then the three squares can be completed into a $2\times 2$ square centered along $v$.
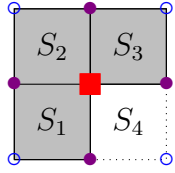
\begin{figure}[h!]
\centering
\begin{tikzpicture}
\draw[fill=lightgray] (0,0) -- (1,0) -- (1,2) -- (0,2) -- (0,0);
\draw[fill=lightgray] (1,1) --(2,1) -- (2,2) -- (1,2) -- (1,1); 
\draw (0,1) -- (0,0) -- (1,0) -- (1,1) -- (0,1) -- (0,2) -- (1,2) --(1,1)--(2,1) -- (2,2) -- (1,2);
\draw[dotted] (1,0) -- (2,0) -- (2,1);
\draw  (0,0 ) node[left]    {} node[blue]{$\circ$};
\draw  (2,2 ) node[right]    {} node[blue]{$\circ$};
\draw  (0,2 ) node[right]    {} node[blue]{$\circ$};
\draw  (2,0 ) node[right]    {} node[blue]{$\circ$};

\draw  (0,1 ) node[left]    {} node[violet]{$\bullet$};
\draw  (1,0 ) node[left]    {} node[violet]{$\bullet$};
\draw  (1,2 ) node[left]    {} node[violet]{$\bullet$};
\draw  (2,1 ) node[left]    {} node[violet]{$\bullet$};

\draw  (1,1 ) node[below right ]    {} node[red]{$\blacksquare$};

\draw  (0.5,0.5 ) node[]    {$ S_1$} ;
\draw  (0.5,1.5 ) node[]    {$ S_2$} ;
\draw  (1.5,1.5 ) node[]    {$ S_3$} ;
\draw  (1.5,0.5 ) node[]    {$ S_4$} ;

\end{tikzpicture}
\end{figure}
\end{lem}

\begin{proof}

Since the group acts transitively on the vertices of type III by Proposition \ref{prop_action_tame_complex}, we can reduce by conjugating by a tame element to the case where the vertex $[\Id]$ is a common point of the three squares.
By Proposition \ref{prop_correspondence_vertex_III}.$(i)$ and $(ii)$, the three squares determine $3$ distinct points $p_1,p_2,p_3$ at infinity such that $p_1$ and $p_2$ are on a same line at infinity $L_{12}$, and $p_2$ , $p_3$ lie on another line $L_{23}$ which is transverse to $L_{12}$. 
Denote by $p_4$ the intersection of the line passing through $p_1$ transverse to $L_{12}$ with the line passing through $p_3$ transverse to $L_{23}$. 
This point determines a unique square $S_4$ containing $[\Id]$ by Proposition \ref{prop_correspondence_vertex_III}.$(ii)$ and the union $S_1\cup S_2 \cup S_3 \cup S_4$ is isometric to $[0,2] \times [0,2]$ since $p_1, p_2 , p_3 $ and $p_4$ lie on a cycle of four lines at infinity. 
\end{proof}


\begin{lem} \label{lem_complete_square_2}
Let $v$ be a vertex of type \textsc{II} and let $S_1,S_2,S_3$ be three distinct $1\times 1 $ squares such that $S_1$ is adjacent to $S_2$ along an edge containing $v$, and $S_2$ is adjacent to $S_3$ along an edge containing $v$. Then the three squares can be completed into a $2\times 2$ square centered along $v$.
\end{lem}

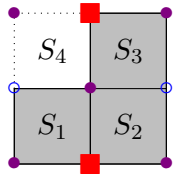
\begin{equation*}
\begin{tikzpicture}
\draw[fill=lightgray] (1,0) -- (2,0) -- (2,2) -- (1,2) --(1,0);
\draw[fill=lightgray] (0,0) -- (1,0) --(1,1) -- (0,1) --(0,0);
\draw (1,0) -- (0,0) -- (0,1) -- (1,1) -- (1,0) -- (2,0) -- (2,1) --(1,1)--(1,2) -- (2,2) -- (2,1);
\draw[dotted] (0,1) -- (0,2) --(1,2);
\draw  (0,2 ) node[left]    {} node[violet]{$\bullet$};

\draw  (0,0 ) node[left]    {} node[violet]{$\bullet$};
\draw  (2,0 ) node[left]    {} node[violet]{$\bullet$};

\draw  (0,1 ) node[left]    {} node[blue]{$\circ$};

\draw  (2,1 ) node[right]    {} node[blue]{$\circ$};
\draw  (1,0 ) node[below right]    {} node[red]{$\blacksquare$};
\draw  (1,2 ) node[below right]    {} node[red]{$\blacksquare$};

\draw  (1,1 ) node[below right ]    {} node[violet]{$\bullet$};
\draw  (2,2 ) node[below right ]    {} node[violet]{$\bullet$};
\draw  (0.5,0.5 ) node[]    {$ S_1$} ;
\draw  (0.5,1.5 ) node[]    {$ S_4$} ;
\draw  (1.5,1.5 ) node[]    {$ S_3$} ;
\draw  (1.5,0.5 ) node[]    {$ S_2$} ;
\end{tikzpicture}
\end{equation*}

\begin{proof} 
Since the tame group and $\PGL_2$ act transitively on the vertices of type III and on the pairs of points on $\mathbb{P}^1$ respectively, we are reduced by conjugating with an appropriate tame automorphism to the situation where the squares $S_1$ and $S_2$ contain $[\Id]$ and the points $[x]$ and $[y]$ respectively.
Take $f$ a  tame automorphism such that the vertex $f \cdot S_2 = S_3$.
By Proposition \ref{prop_stab_II}.$(ii)$, $f $ belongs to:
\begin{equation*}
 \EV  \rtimes \left \lbrace \left ( \begin{array}{ll}
a x & d^{-1} y \\
d z + cx &  at + c a^{-1} d^{-1} y 
\end{array} \right ) \ | \  (a,c,d) \in \C^* \times \C \times \C^*  \right \rbrace .
\end{equation*}
Since the subgroup $\left \lbrace \left ( \begin{array}{ll}
a x & d^{-1} y \\
d z + cx &  at + c a^{-1} d^{-1} y 
\end{array} \right ) \ | \  (a,c,d) \in \C^* \times \C \times \C^*  \right \rbrace$ is a normal subgroup in the product, we conclude that there exists a square $S_4$ such that the union $S_1 \cup S_2 \cup S_3 \cup S_4$ is isometric to $[0,2] \times [0,2]$, as required.
\end{proof}

\begin{lem}
\label{lem_complete_square_1}
Let $v$ be a vertex of type \textsc{I} and let $S,S_1,S_2$ be three distinct $1\times 1 $ squares such that $S$ is adjacent to $S_1$ along an edge containing $v$, and $S$ is adjacent to $S_2$ along an edge containing $v$. 
Let $g_1$ and $g_2 \in \stame$ such that $g_1 S = S_1$ and $g_2 S = S_2$.  
 Then the three squares can be completed into a $2\times 2$ square centered along $v$ if and only if $g_1$ or $g_2$ belongs to $A_{v}$.

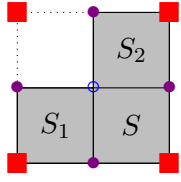
\begin{equation*}
\begin{tikzpicture}
\draw[dotted] (0,1) --(0,2) --(1,2);
\draw[fill=lightgray] (1,0) -- (2,0) -- (2,2) -- (1,2) --(1,0);
\draw[fill=lightgray] (0,0) -- (1,0) --(1,1) -- (0,1) --(0,0);
\draw (1,0) -- (0,0) -- (0,1) -- (1,1) -- (1,0) -- (2,0) -- (2,1) --(1,1)--(1,2) -- (2,2) -- (2,1);
\draw  (0,2 ) node[left]    {} node[red]{$\blacksquare$};

\draw  (0,0 ) node[left]    {} node[red]{$\blacksquare$};
\draw  (2,0 ) node[left]    {} node[red]{$\blacksquare$};
\draw  (2,2 ) node[left]    {} node[red]{$\blacksquare$};

\draw  (0,1 ) node[left]    {} node[violet]{$\bullet$};

\draw  (2,1 ) node[right]    {} node[violet]{$\bullet$};
\draw  (1,0 ) node[below right]    {} node[violet]{$\bullet$};
\draw  (1,2 ) node[below right]    {} node[violet]{$\bullet$};

\draw  (1,1 ) node[below right ]    {} node[blue]{$\circ$};
\draw  (0.5,0.5 ) node[]    {$ S_1$} ;
\draw  (1.5,1.5 ) node[]    {$ S_2$} ;
\draw  (1.5,0.5 ) node[]    {$ S$} ;
\end{tikzpicture}
\end{equation*}

\end{lem}
\black

 \begin{proof} Since the group $\stame$ is transitive on the set of $1\times 1$ squares, we can suppose that the common vertex $v$ is $[x]$ and that $S_2 $ contains the vertex $[\Id]$.
 We are thus in the following situation:
 
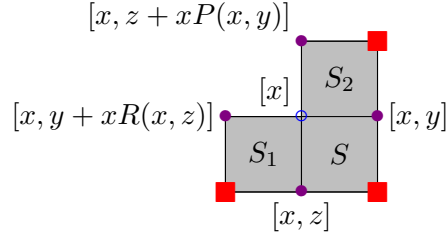
\begin{equation*}
\begin{tikzpicture}
\draw[fill=lightgray] (1,0) -- (2,0) -- (2,2) -- (1,2) --(1,0);
\draw[fill=lightgray] (0,0) -- (1,0) --(1,1) -- (0,1) --(0,0);
\draw (1,0) -- (0,0) -- (0,1) -- (1,1) -- (1,0) -- (2,0) -- (2,1) --(1,1)--(1,2) -- (2,2) -- (2,1);
\draw (1,1) node {} node[above left] {$ [x ]$};
\draw (1,0) node {} node[ below] {$ [x,z ]$};
\draw (2,1) node {} node[ right] {$ [x ,y]$};
\draw (1,2) node {} node[above left] {$ [x ,z + xP(x,y)]$};
\draw (0,1) node {} node[ left] {$ [x ,y + xR(x,z)]$};

\draw  (0,0 ) node[left]    {} node[red]{$\blacksquare$};
\draw  (2,0 ) node[left]    {} node[red]{$\blacksquare$};
\draw  (2,2 ) node[left]    {} node[red]{$\blacksquare$};

\draw  (0,1 ) node[left]    {} node[violet]{$\bullet$};

\draw  (2,1 ) node[right]    {} node[violet]{$\bullet$};
\draw  (1,0 ) node[below right]    {} node[violet]{$\bullet$};
\draw  (1,2 ) node[below right]    {} node[violet]{$\bullet$};

\draw  (1,1 ) node[below right ]    {} node[blue]{$\circ$};
\draw  (0.5,0.5 ) node[]    {$ S_1$} ;
\draw  (1.5,1.5 ) node[]    {$ S_2$} ;
\draw  (1.5,0.5 ) node[]    {$ S$} ;
\end{tikzpicture}
\end{equation*}
where $P,R \in \C[x,y]$.
\smallskip
The reverse implication then follows directly from the fact that the subgroup $A_{[x]}$ is a normal subgroup of $\stab([x])$.

Let us prove the first implication $(\Rightarrow)$. Suppose that the squares $S_1,S_2,S_3$ are flat. Then there exists a component  $f_4 \in \CSL2$ such that the element $f$ given by:
 \begin{equation*}
 f = \left ( \begin{array}{ll}
 x & y + xR(x,z) \\
z + xP(x,y) & f_4
\end{array}  \right )
\end{equation*}   
belongs to $\tame$. In particular, it must fix the volume form $\Omega$, this implies that:
\begin{equation*}
\partial_y P(x,y) \partial_z R(x,z) = 0 \in \CSL2.
\end{equation*}
This implies that $\partial_y P(x,y) =0$ or $\partial_z R(x,z) = 0$ hence $g_1$ or $g_2$ belongs to $A_{[x]}$ as required.
 \end{proof}

\begin{lem} \label{lem_reduction_stab_I_geometric} Take $S$ and $S'$ two $2\times 2$ squares centered at a vertex of type \textsc{III} which are adherent along a vertex of type \textsc{I}. 
Then $S$ and $S'$ satisfy one of the following properties. 
\begin{enumerate}
\item[(i)] Either the pair $(S,S')$ is flat. 
\item[(ii)] Either the pair of squares $(S,S')$ is contained in a horizontal or vertical spiral staircase.
\end{enumerate}
\end{lem}

\begin{proof}
 Consider two squares $S,S'$ such that the pair of square $(S,S')$ is not flat. 
 Up to a conjugation by an element of $\stame$, we can suppose that $S$ and $S'$ are adherent along the vertex $[x]$.  
Since the group $\tame$ acts transitively the set of $2\times 2$ squares centered on type III vertices by Proposition \ref{prop_action_tame_complex}.$(vi)$, there exists an element  $g \in \stame$ such that $g \cdot S = S'$. \black 
Choose a minimal sequence $S_i$ of adjacent $2\times 2$ squares centered along a vertex of type III containing $[x]$ such that $S_1 = S, \ldots , S_k =S'$. 
Since the sequence of square is minimal, the square $S_i $ and $S_{i+2}$ are adherent along the vertex $[x]$ but not flat.
Moreover, the squares $S_i $ and $S_{i+1}$ are alternatively adjacent along vertical and horizontal edges. Hence the pair $(S,S')$ is contained in a horizontal or vertical spiral staircase, as required.      
\end{proof}

In practice, we will use the following explicit characterization to determine whether two squares adherent along a vertex of type I are flat.
\begin{lem} \label{lem_technical_flat}
Consider two $2\times 2$ adjacent squares $S_1, S_2$ along a horizontal edge containing $[x_1],[y_1]$ and a polynomial $P \in \C[x,y] \setminus \C$. 
Denote by $[z_1], [t_1]$ the other vertices of $S_1$ such that  $[x_1], [z_1]$ belong to a  vertical edge of $S_1$ and by $[z_1 + x_1P(x_1,y_1)], [t_1+ y_1P(x_1,y_1)]$ the two other vertices of $S_2$.
Let $g$ be the tame automorphism defined by 
\[
g =
\left(\begin{array}{cc}
x &  y \\
z + x P(x,y) & t + y P(x,y)
\end{array}\right)\]
so that $g \cdot S_1 = S_2$.  

The following assertions hold.
\begin{enumerate}
\item[(i)] We have $g \in A_{[x_1]}$ if and only if $P\in \C[x] \setminus \C$.
\item[(ii)] For any square $S'$ adjacent to $S_1$ along the vertical edge containing $[x_1],[z_1]$, the squares $S_1,S',S_2$ are flat if and only if $P \in \C[x]\setminus \C$.

\end{enumerate} 

\end{lem}
The following figure summarizes the initial situation in the previous lemma:
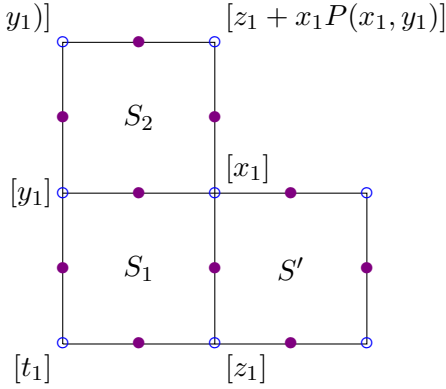
\begin{center}
\begin{tikzpicture}
 \draw (0,0) --(2,0) --(2,4) --(0,4) --(0,0);
 \draw (0,2) --(2,2);
 \draw (2,2)--(4,2)--(4,0)--(2,0);
 \draw(1,1) node {$S_1$};
 \draw (1,3) node {$S_2$};
 \draw (0,0) node[blue] {$\circ$} node[below left] {$[t_1]$};
 \draw (2,0) node[blue] {$\circ$} node[below right] {$[z_1]$};
 \draw (2,2) node[blue] {$\circ$} node[ above right] {$[x_1]$};
 \draw (0,2) node[blue] {$\circ$} node[left] {$[y_1]$};
 \draw (0,4)  node[blue] {$\circ$} node[above left] {$[t_1+ y_1P(x_1,y_1)]$};
 \draw (2,4) node[blue] {$\circ$} node[above right] {$[z_1+ x_1P(x_1,y_1)]$};
\draw (0,1) node[violet] {$\bullet$};
\draw (4,1) node[violet] {$\bullet$};
\draw (3,0) node[violet] {$\bullet$};
\draw (3,2) node[violet] {$\bullet$};
\draw (4,2) node[blue] {$\circ$};
\draw (4,0) node[blue] {$\circ$};
 \draw (3,1) node {$S'$};

\draw (2,1) node[violet] {$\bullet$};
\draw (0,3) node[violet] {$\bullet$};
\draw (2,3) node[violet] {$\bullet$};
\draw (1,0) node[violet] {$\bullet$};
\draw (1,2) node[violet] {$\bullet$};
\draw (1,4) node[violet] {$\bullet$};
\end{tikzpicture}
\end{center}

\begin{proof} By conjugation, we can suppose that $x_1 =x, y_1= y , z_1= z $ and $t_1=t$.
Assertion $(i)$ follows directly from the definition of $A_{[x]}$. 

Let us prove assertion $(ii)$.
 Choose a square $S'$ such that $g' S_1 = S'$ where $g' \notin A_{[x]}$. Lemma \ref{lem_complete_square_1} implies that the squares $S_1,S_2,S'$ are flat if and only if $g \in A_{[x]}$. 
 And $g \in A_{[x]}$ is equivalent to the fact that $P \in \C[x] \setminus \C$ by assertion $(i)$. 
%
\end{proof}

\section{Global geometry of the complex} \label{section_hyperbolic}

In this section, we first review the results due to Bisi-Furter-Lamy regarding the global geometric properties of the metric square complex $(\mathcal{C}, d_\mathcal{C})$ introduced in \S \ref{section_square}.
We then describe the degree of iterates of a tame automorphism fixing a vertex of the complex. 

\subsection{Gromov curvature and Gromov-hyperbolicity} \label{section_cat0_hyperbolic} 

Recall that a map $\gamma : [0 , l] \to (\mathcal{C}, d_\mathcal{C})$ defines a geodesic segment of length $l$ if $\gamma$ induces an isometry from $[0,l]$ to $\gamma([0,l])$. 
A map $\gamma : \mathbb{R} \to \mathcal{C}$ which is an isometry onto its image is called a geodesic line and a map $\gamma: \mathbb{R}^+ \to \mathcal{C}$ which is an isometry onto its image is called a geodesic half-line.
Recall also that $\gamma : [0,l] \to \mathcal{C}$ is a quasi-geodesic if there exists $\lambda >0,M >0$ such that for any $s,s' \in [0,l]$, the following inequality is satisfied:
\begin{equation*}
\dfrac{1}{\lambda} |s - s'| - M \leqslant d_\mathcal{C}( \gamma(s) , \gamma(s') ) \leqslant \lambda |s-s'| + M.
\end{equation*}
  As a result, a geodesic line is also a quasi-geodesic. 
When any two points on a metric space can be joined by a geodesic segment, we say that the space is a geodesic metric space.


\medskip

A geodesic space $(X,d)$ is $\CAT(0)$ (see \cite[Section II.1]{bridson_haefliger}) if its triangles are thinner than euclidian triangles. In other words, $(X,d)$ satisfies the following condition. 
For any three points $p,q,r$ in $X$, take a triangle in the euclidean plane $(\mathbb{R}^2, || \cdot ||)$ with vertices $\bar p, \bar q , \bar r \in \mathbb{R}^2$ such that $d( p , q) = ||\bar p - \bar q||$, $d( q , r ) = || \bar q - \bar r||$ and $d( r , p )=|| \bar  r - \bar p || $. 
Then for any point $m_1 \in X$ and $m_2 \in  X$ in the geodesic segment $[p,q]$ and $[q,r]$ respectively, one has:
\begin{equation*}
d(m_1,m_2) \leqslant || \bar m_1 - \bar m_2||,
\end{equation*}  
where $\bar m_1$ and $\bar m_2$ are the unique points on the segments $[\bar p,\bar q]$ and $[\bar q ,\bar r]$ respectively such that $d(m_1,p) = || \bar p - \bar m_1||$ and $d(r,m_2) = || \bar r - \bar m_2||$. 
\medskip

Let us recall the notion of Gromov-hyperbolic metric space. Let $\delta>0$ be a positive real number. A metric space $(X,d)$ is $\delta$-hyperbolic if for any geodesic triangle $T = [p, q] \cup [q,r] \cup [r , p]$ in $X$ and for any point $m \in [p,q]$, we have:
\begin{equation*}
d( m , [q,r] \cup [r,p]) \leqslant \delta
\end{equation*}

\begin{thm} (\cite[Theorem A]{bisi_furter_lamy}) \label{thm_struct_complex} The square complex $\mathcal{C}$, endowed with the distance $d_\mathcal{C}$, is a  geodesic metric space which is simply connected, $\CAT(0)$ and Gromov-hyperbolic. 
\end{thm}

 The previous result has important consequences on the behavior of the isometries of the complex, i.e distance preserving maps.
Recall that the translation length, denoted $l(f)$, of an isometry $f: \mathcal{C} \to \mathcal{C}$ is defined by:
\begin{equation*}
l(f) = \inf_{v\in \mathcal{C}} d_\mathcal{C}(v , f(v)).
\end{equation*} 
Observe that for any isometry $f$, the points in the complex where the infimum is reached is invariant by $f$. 
We denote by $\Min(f)$ the subset of $\mathcal{C}$ on which the infimum is reached.  
\begin{thm} \label{thm_general_action} Let $f : \mathcal{C} \to \mathcal{C}$ be an isometry of $\mathcal{C}$ which is also a morphism of complex. Then either $l(f)=0$ and $f$ fixes a vertex in the complex, either $l(f)>0$ and one can find $f$-invariant geodesic line on which $f$ acts by translation by $l(f)$. 
\end{thm}

In other words, a tame automorphism $f$ is either \textbf{elliptic} (when $l(f)=0$) or \textbf{hyperbolic}.
\begin{proof}
Take $f$ an isometry of the complex $\mathcal{C}$. 
Then $\Min(f)$ is non-empty by \cite[II.6.6.(2)]{bridson_haefliger}. 
Suppose that $l(f) >0$, then $f$ satisfies the hypothesis of \cite[II.Theorem 6.8]{bridson_haefliger}.
More precisely, \cite[II.Theorem 6.8.(1)]{bridson_haefliger} asserts that  an isometry $f$ of a $\CAT(0)$ space satisfies $l(f) >0$ if and only if $f$ translates by $l(f)$ on an invariant geodesic line, as required. 
Otherwise $l(f)=0$, we prove that there exists a vertex which is fixed by $f$. Take a sequence of points $v_p$ in $\mathcal{C}$ such that the distance $d_\mathcal{C}(v_p , f^2\cdot v_p) $ tends to $0$. If these points belong to the interior of a square, their image will also be in the interior of a square. Since the distance between $v_p$ and $f^2 \cdot v_p$ is arbitrarily small, they should belong to two squares $S_p,S_p'$ which intersect, the only solution is that the intersection is fixed by $f^2$, hence $f^2$ fixes a vertex or an edge or a square. Since each edge is joined by two vertices of different type and since $f^2$ preserves the type of vertices, we conclude that $f^2$ fixes a vertex in the complex. Similarly, if $f^2$ fixes a square, then it also fixes the unique vertex of type III on the given square. 
A similar argument also holds if the sequence $v_p$ are contained in the edges of $\mathcal{C}$. 
In any of these cases, we conclude that $f$ must also preserve a vertex in the complex $\mathcal{C}$, as required.    
\end{proof}
\black 


\subsection{Degree growths of elliptic automorphisms}
\label{subsection_degree_elliptic}
%
%

In this section, we apply the results of the previous section to study the degree growth of particular tame automorphisms.
Recall from the previous section  that a tame automorphism is \textbf{elliptic} or \textbf{hyperbolic} if its action on the complex fixes a vertex  or preserves a geodesic line of the complex on which it acts by translation respectively.
%
%
%

The following result classifies the  degree growth of any  elliptic tame automorphisms.
\begin{thm} \label{thm_growth_elliptic} Let $f \in \tame$ be any tame automorphism of $\SL2$ fixing a vertex in the square complex.
 Then we are in one of the following situations:
\begin{enumerate}
\item[(i)] The sequence $(\deg(f^n), \deg(f^{-n}))$ is bounded and $f$ is linear or $f^2$ is conjugated in $\Bir(\Pg^3)$ to an automorphism of the form $(x,y,z) \mapsto (a x, b y + xR(x) , b^{-1}z + xP(x,y)) $ with $a,b \in \C^*$, $P\in\C[x,y]$ and $R \in \C[x]$.
\item[(ii)] There exists a constant $C>0$ such that:
 $$ \dfrac{1}{C} n \leqslant \deg(f^{\epsilon n}) \leqslant Cn ,$$
where $\epsilon \in \{ +1, -1 \}$ and $f$ is conjugated in $\Bir(\Pg^3)$ to an automorphism of the form:
\begin{equation*}
(x,y,z) \mapsto (ax ,b^{-1} (z + xR(x)) , b(y + xP(x) z )),
\end{equation*}
with $a,b \in \C^*$, $R \in \C[x]$ and $P\in \C[x]\setminus \C$.
\item[(iii)] There exists a constant $C>0$ and an integer $d$ such that:
 $$ \dfrac{1}{C} d^n \leqslant \deg(f^{\epsilon n}) \leqslant C d^n ,$$
 where $\epsilon \in \{ +1, -1 \}$ and $f$ is conjugated in $\Bir(\Pg^3)$ to a composition of elements of the form:
 \begin{equation*}
(x,y,z) \mapsto (ax, b (z + xP(x,y)) , b^{-1} (y+ xR(x))),
\end{equation*}
where $a,b \in \C^*$, $R \in \C[x]$ and $P \in \C[x,y]$ such that $\deg_y(P) \geqslant 2$. 
\end{enumerate}
\end{thm}

\begin{rem} In case $(iii)$ of the previous Theorem, suppose $f$ is a normal form, then $\deg(f^p) = C d^p + C_0$ where $C>0 $ and $C_0 \in \mathbb{Z}$. 
\end{rem}

\begin{rem} We summarize the growth of the degree of elliptic automorphisms.
%
$$ \begin{array}{|c|c|c|c|c|}
\hline
\text{Fixed vertex} &  \text{Action on the link} &  \text{Fibration} & \text{Behavior on the fiber} & \deg(f^n) \\
\hline
\text{Type III} &  &  & & \text{bounded}   \\
 \hline
 \text{Type II} &  & \text{over }\Pg^2 &\text{Flow of a vector field } &  \text{bounded} \\
\hline
\text{Type I} & \text{trivial on the Bass-Serre tree} & \text{over} \Pg^1 &\text{Flow of a vector field } & \text{bounded} \\
\hline 
\text{Type I} & \text{involution on the Bass-Serre tree} & \text{over }\Pg^1 & \text{Affine} &  \text{linear} \\
\hline 
\text{Type I} & \text{hyperbolic on the Bass-Serre tree} & \text{over }\Pg^1 & \text{Composition of Henon} & \text{exponential} \\
\hline
\end{array} $$
\end{rem}
\black 
\begin{proof} Take $f \in \tame$ an elliptic automorphism. Since $f$ fixes a vertex on the complex, we will distinguish three cases depending on the type of vertices $f$ fixes.  Moreover, recall that the degree growth is an invariant of conjugation and that by Proposition \ref{prop_action_tame_complex}, the tame group acts transitively on the set of vertices of type I, II and III respectively. We are thus reduced to compute the degree growth for $f $ in the subgroups $\stab([\Id]), \stab([x,z])$ and $\stab([x])$ respectively.

\textbf{First case}: If $f \in \stab([\Id])= \O4$, the sequence $(\deg(f^n), \deg(f^{-n}))$ is bounded.
\bigskip

\textbf{Second case}: Suppose that $f \in \stab([x,z])$. 
By Proposition \ref{prop_stab_II}, one has:
\begin{equation*}
\stab([x,z]) = \EH \rtimes \left \{ \left ( \begin{array}{ll}
a x + b z & a' y + b' t \\
cx + d z & c' y + d't
\end{array} \right )\  | \  \left ( \begin{array}{ll}
a & b \\
c & d
\end{array} \right ) \left (\begin{array}{ll}
d' & -b' \\
-c' & a'
\end{array} \right ) = I_2 \in M_2(\C)\right \}.
\end{equation*}
Denote by $\pi_{xz} : \SL2 \to \mathbb{A}^2 \setminus \{(0,0) \}$ the map induced by the projection 
\begin{equation*}
(x,y,z,t) \rightarrow (x,z).
\end{equation*}
 Recall that $\pi_{xz}^{-1}(\mathbb{A}^2 \setminus (\{ 0 \}\times \mathbb{A}^1))$ is isomorphic to $\mathbb{A}^2 \setminus (\{ 0 \}\times \mathbb{A}^1) \times \mathbb{A}^1$. We fix an isomorphism, since $f$ fixes the fibration $\pi$, it induces a regular automorphism on $\mathbb{A}^2 \setminus (\{ 0 \}\times \mathbb{A}^1) \times \mathbb{A}^1$ of the form:
\begin{equation*}
f : (x,z , y)  \mapsto \left ( ax, b^{-1} z , b(y+ xP(x,z)) \right ).
\end{equation*}
In particular, the sequence $(\deg(f^n), \deg(f^{-n}))$ is bounded and $f$ satisfies assertion $(i)$.
\bigskip

\textbf{Third case}: Consider  $f \in  \stab([x])$ such that $f \notin \stab([x,y])\cup \stab([x,z]) $. 
Since $f$ preserves the fibration $\pi_x : \SL2 \to \mathbb{A}^1$ and since $\pi_x^{-1}(\mathbb{A}^1\setminus \{ 0\})$ is isomorphic to $\mathbb{A}^1 \setminus \{ 0\} \times \mathbb{A}^2$, the automorphism $f$ is of the form:
\begin{equation*}
f :(x,y,z) \rightarrow (x, f_1 , f_2),
\end{equation*}
where $(f_1, f_2)  $ defines an element of $\Aut(\mathbb{A}^2_{\C[x]})$. 

By Proposition \ref{prop_link_bass_serre}, $f$ induces an action on the subtree of the Bass-Serre tree associated to $\Aut(\mathbb{A}^2_{\C(x)})$. If $f$ induces an action on this subtree which fixes every point of the tree, then $f$ belongs to $A_{[x]}$. 
By Proposition \ref{prop_link_bass_serre}.$(iv)$, $f$ is then of the form:
\begin{equation*}
\left ( \begin{array}{ll}
 ax & b (y + x P(x)) \\
b^{-1}(z + x S(x)) & a^{-1}(t + z P(x) + y S(x) + x P(x)S(x)) 
\end{array}
\right )
\end{equation*}
where $P,S \in \C[x] \setminus \C$. 
In particular, the sequences $(\deg(f^n))$ and $ \deg(f^{-n}))$ are bounded and $f$ satisfies assertion $(i)$ since in the fixed trivialization, $f$ is of the form $(x,y,z) \mapsto (x, by + xP(x) , z + x S(x) )$. 
\medskip

Recall that the vertices of type II in the Bass-Serre tree $\mathcal{T}_{\C(x)}$ were equivalence classes of components $(f_1,f_2)$ of automorphisms in $\Aut(\mathbb{A}^2_{\C(x)})$ where two components $(f_1,f_2) \simeq (g_1,g_2)$ if and only if there exists $\left ( \begin{array}{ll}
a & b \\
c & d
\end{array}  \right ) \in \GL_2(\C(x))$ such that $(g_1,g_2) = (af_1 + bf_2, c f_1 + d f_2)$.

Suppose that $f,f^2 \notin A_{[x]}$ and the action of $f$ on the subtree of $\mathcal{T}_{\C(x)}$ fixes a vertex. If the fixed vertex in the tree $\mathcal{T}_{\C(x)}$ is of type II, then we can suppose that $f$ fixes the vertex given by $[y,z]$.
In particular, this implies that $f$ is conjugated to 
\begin{equation*}
\left ( \begin{array}{ll}
ax & b (y + xP(x)z )\\
 b^{-1} (z + xR(x)) & a^{-1} (t + z^2P(x) + yR(x))
\end{array} \right ) \text{or} \left ( \begin{array}{ll}
ax & b^{-1} (z + xR(x)) \\
b (y + xP(x)z )& a^{-1} (t + z^2P(x) + yR(x))
\end{array} \right )
\end{equation*}
with $P \in \C[x] \setminus \C$ and $R \in \C[x]$.
 In particular, the sequences $\deg(f^n)$ and $\deg(f^{-n})$ are bounded in the first case and grow linearly in the second. In the first case, $f$ satisfies assertion $(i)$ and $f$ satisfies assertion $(ii)$ in the second.
\medskip

If $f, f^2 \notin A_{[x]}$ and the action $f$ on $\mathcal{T}_{C(x)}$ fixes a vertex of type I but no vertices of type II, then $f$ is conjugated to an element which fixes the vertex $[z]$ in the Bass-Serre tree, in particular it is conjugated to
\begin{equation*}
\left ( \begin{array}{ll}
ax & b (y + xP(x,z) )\\
 b^{-1} (z + xR(x)) & a^{-1} (t + z^2P(x) + yR(x))
\end{array} \right )
\end{equation*} 
with $P \in \C[x,y],R \in \C[x] \setminus \C$.
In this case, the degrees are both bounded and $f$ satisfies assertion $(i)$. 

The remaining case is when the action on the tree $\mathcal{T}_{\C(x)}$ is hyperbolic and  using the amalgamated product structure, we deduce that $f$ is conjugated to a composition of elements of the form:
\begin{equation*}
\left ( \begin{array}{ll}
ax & b(z + xP(x,y))\\
b^{-1}(y + xR(x)) & a^{-1}(t + z R(x) + yP(x,y) + xP(x,y)R(x)) 
\end{array} \right ),
\end{equation*}
where $R \in \C[x]$ and $P \in \C[x,y]$ such that $\deg_y(P) \geqslant 2$. 
In this case, the degree sequences $(\deg(f^n)), (\deg(f^{-n})$ are both equivalent to $d^n$ and $f$ satisfies assertion $(iii)$. 
\end{proof}


\section{Valuative estimates}
\label{section_estimates}

This section is devoted to the generalization of the so-called parachute inequalities (see \cite[Minoration A.2]{bisi_furter_lamy}). Our proof extends the method of \cite{lamy_venereau} to more general valuations. 
The plan of the section is as follows. First we recall some general facts on valuations (\S \ref{subsection_valuation_general}), then we consider a particular class of valuations in  \S \ref{subsection_monomial_valuations}. 
For these particular valuations, we introduce the parachute associated to a pair of regular functions on the quadric allowing us to estimate the degree of a derivative on a given direction (\S \ref{subsection_parachute}). Using this and some elementary facts on key  polynomials (\S \ref{subsection_key_polynomials}), we finally deduce our key estimates in \S \ref{subsection_parachute_ineq}.

\subsection{Valuations on affine and projective varieties}
\label{subsection_valuation_general}
Let  $X $ be an affine variety of dimension $n$ over $\C$. 
By convention for us, a valuation on $X$ is a map $\nu : \C[X] \to \mathbb{R} \cup \{ +\infty\}$ which satisfies the following properties.
\begin{enumerate}
\item We have $\nu^{-1}( \{ + \infty \}) = \{0 \}$.
\item The function $\nu$ is not constant on $\C[X] \setminus \{0 \}$.
\item For any $a\in \C^*$, one has  $\nu(a) = 0$.
\item For any $f_1,f_2 \in \C[X]$, one has $\nu(f_1f_2) = \nu(f_1) + \nu(f_2)$.
\item For any $f_1,f_2 \in \C[X]$, one has $\nu(f_1+f_2) \geqslant \min(\nu(f_1), \nu(f_2))$.
\end{enumerate}
When the subset $\nu^{-1}(\{+\infty \})$ is not reduced to $\{ 0\}$, we say that $\nu$ is a semi-valuation.
We endow the space of valuations with the coarsest topology for which all evaluation maps $\nu \mapsto \nu(f)$ are continuous where $f \in \C[X]$.
\medskip

The group $\mathbb{R}_+^{*} $ naturally acts on the set of valuations by multiplication.

\medskip

The main examples of valuations are monomial valuations. We recall their definition below.
Fix a point $p$ on $X$, an algebraic system of coordinates $u = (u_0,\ldots , u_{n-1})$  at this point and some weights $\alpha = ( \alpha_1, \ldots , \alpha_n) \in \mathbb{R}^n$.
We shall denote by $u^I = \prod_{j=0}^n u_{j}^{i_j}$ when $I= (i_0, \ldots , i_{n-1}) \in \mathbb{N}^n$ and by $\langle I, \alpha \rangle = \alpha_0 i_0 + \ldots + \alpha_{n-1} i_{n-1}$ the usual scalar product.
 The monomial valuation $\nu$ with weight $\alpha$ with respect to the system of coordinates $u$ is defined by:
\begin{equation*}
\nu\left ( \sum_{I \in  \mathbb{N}^n} a_{I} u^I\right ) = \min \left \{ \langle I, \alpha \rangle \ | \ a_{I} \neq 0 \right \}, 
\end{equation*}
where $a_I \in \C$. 

When $f \in \mathcal{O}_{p,X}$ is a regular function at the point $p$, then one defines $\nu(f)$ as:
\begin{equation*}
\nu(f) = \nu( \sum a_I(f) u^I ) ,
\end{equation*}
where $\sum a_I(f)u^I$ is a formal expansion of $f$ near $p$. The fact that $\nu(f)$ does not depend on the choice of the formal expansion of $f$ near $p$ is proved in \cite[Proposition 3.1]{mustata_jonsson}.

Observe that when $\alpha = (1,0,\ldots , 0)$, then the associated valuation coincides with the order of vanishing along $\{u_0=0\}$.
Furthermore, when $X = \Spec( \C[x,y,z,t])$, the valuation $- \deg$ coincides with the monomial valuation with weight $(-1,-1,-1,-1)$ with respect to $(x,y,z,t)$.

\medskip

Consider a regular morphism $f: X\to Y$ where $Y$ is an affine variety and a valuation $\nu$ on $X$.
 The pushforward of the valuation $\nu$ on $X$ by $f$ is denoted $f_* \nu$ is given by the formula:
\begin{equation*}
f_* \nu = \nu \circ f^{\sharp},
\end{equation*}
where $f^\sharp$ denotes the morphism of $\C$-algebra corresponding to $f$.

\smallskip

We also recall the notion of center of a valuation $\nu$. 

When $\nu_{| \C[X]} \geqslant 0$, then the center of $\nu$ in $X$, denoted $Z(\nu)$, is the scheme theoretic point corresponding to the prime ideal $\{ f_1 \in \C[X] \ | \ \nu(f_1) >0 \}$. When this condition does not hold, there exists a regular function $f_1$ such that $\nu(f_1) < 0$ and we say that $\nu$ is centered at infinity.
In the latter case, for any projective variety $\bar X$ containing $X$ as a Zariski open subset, the center of $\nu$ in $\bar X$ is a non-empty Zariski closed irreducible subset which is contained in $\bar X \setminus X$.  
Denote by $R_\nu$ the valuation ring and by $\mathcal{M}_\nu$ its maximal ideal, then the center $Z(\nu)$ of $\nu$ in $\bar X$ can be defined as follows:
\begin{equation*}
Z(\nu) = \{ p\in \bar X \ | \ \mathcal{O}_{p,\bar X} \subset R_\nu , \mathcal{M}_{p,\bar X} = \mathcal{M}_\nu \cap \mathcal{O}_{p,\bar X}  \},
\end{equation*}
where $\mathcal{O}_{p,\bar X}$ denotes the local ring of regular functions at the point $p$ and where $\mathcal{M}_{p, \bar X}$ is its maximal ideal.
The fact that $Z(\nu)$ is non-empty follows from the valuative criterion of properness and we shall refer to \cite{vaquie} for the general properties of this set. 
%
\subsection{Valuations $\mathcal{V}_0$ on the quadric}
\label{subsection_monomial_valuations}
We denote by $q \in \C[x,y,z,t]$ the polynomial $q= xt - yz$ and by $\pi : \C[x,y,z,t] \to \CSL2$ the canonical projection.
Our objective is to define a subset  of the set of all valuations on the quadric $\SL2$, with different weights on some coordinate axis.

Take a point $p=(x_0,y_0,z_0,t_0) \in \mathbb{A}^4$ and a weight $\alpha = (\alpha_0 , \alpha_1, \alpha_2,\alpha_3) \in (\mathbb{R}^-)^4$, we write by $\nu_{p}^\alpha$ the monomial valuation on $\C[x,y,z,t]$ with weight $\alpha$ with respect to the system of coordinates $(x-x_0,y-y_0, z- z_0, t-t_0)$.

\begin{prop} \label{prop_minimal_valuation}
For any point $p \in \mathbb{A}^4$ and any weight $\alpha = (\alpha_0, \alpha_1,\alpha_2,\alpha_3) \in (\mathbb{R}^- \setminus \{0 \})^4$ such that $\alpha_0 + \alpha_3 = \alpha_2 + \alpha_1$, the map $\nu : \CSL2 \to \mathbb{R}^- \cup \{ +\infty \}$ given by: 

\begin{equation*}
  \nu( f) :=  \sup \left \{ \nu_p^{\alpha} (R) \ | \ R \in \C[x,y,z,t] , \pi(R) = f \right  \},
\end{equation*}
for any $f \in \CSL2$ is a valuation on the quadric which is centered at infinity. 

Moreover, suppose $p = (x_0,y_0,z_0,t_0)\in \C^4$ and $\nu' : \CSL2 \to \mathbb{R}^{-} \cup \{ +\infty \}$ is a valuation such that  $\nu(\pi(x-x_0)) = \nu'(\pi(x-x_0))$, $\nu(\pi(y-y_0))=\nu'(\pi(y-y_0))$, $\nu(\pi(z-z_0)) = \nu'(\pi(z-z_0))$ and $\nu(\pi(t-t_0)) = \nu'(\pi(t-t_0))$, then
\begin{equation*}
\nu'(f) \geqslant \nu(f),
\end{equation*}
for any regular function $f \in \CSL2$.
\end{prop}

\begin{defi} The set $\mathcal{V}_0$ is set of all valuations $\nu: \CSL2 \to \mathbb{R}^- \cup \{ +\infty
 \}$ defined by
 \begin{equation*}
 \nu(f) := \sup \{\nu_p^\alpha(R)\ | \ \pi(R) = f \},
 \end{equation*}
for any $f \in \CSL2$ where $p \in \C^4$ and where $\alpha = (\alpha_0,\alpha_1,\alpha_2, \alpha_3)\in (\mathbb{R}^- \setminus \{ 0\})^4$ is a multi-index for which $\alpha_0+\alpha_3= \alpha_1 +\alpha_2$.
%
\end{defi}  
The group $\mathbb{R}^{+,*}$ acts naturally by multiplication on the set of valuations on the quadric and this action descends on an action on $\mathcal{V}_0$.

\begin{rem} Observe that for $x_0= y_0= z_0= t_0 = 0$ and $\alpha_1= \alpha_2= \alpha_3 = \alpha_4=-1$, the corresponding valuation on the quadric is the order of vanishing along the hyperplane at infinity.
\end{rem}

\begin{ex} Consider $p = (0,0,0,0)$ and $\alpha = (- 1/2, -3/5 , -9/10,-1)$, then the associated valuation $\nu$ is the monomial valuation at the point $[0,0,0,1,0] \in \overline{\SL2}$ with weight $(2/5, 1/10, 1)$ with respect to the coordinate chart $(u,v,w) \mapsto [ w^2 + uv , u ,v ,1, w] \in \overline{\SL2}$. 
In particular, its center  is the point $[0,0,0,1,0] \in \overline{\SL2}$.
\end{ex}

\begin{ex} Consider $p = (1,2,3,4)$ and $\alpha = (- 1/2, -3/5 , -9/10,-1)$, then the associated valuation $\nu$ is the monomial valuation at the point $[6,2,3,1,0] \in \overline{\SL2}$ with weight $(2/5, 1/10, 1)$ with respect to the coordinate chart $(u,v,w) \mapsto [ w^2 + (2 +u)(3+v) , 2 +u ,3 +v ,1, w] \in \overline{\SL2}$. 
In particular, its center  is the point $[6,2,3,1,0] \in \overline{\SL2}$.
\end{ex}

To prove the proposition, we shall need the following technical lemma.

\begin{lem} \label{lem_maximality_valuation} Let $\nu': \C[x,y,z,t] \to \mathbb{R}^- \cup \{+\infty\}$ be a valuation such that $\nu'_{| \C[x,y,z,t]\setminus \C} < 0$.
For any polynomial $R \in \C[x,y,z,t]$ given by
\begin{equation*}
R = \sum_{ijkl} a_{ijmn} x^iy^jz^mt^n,
\end{equation*}
with $a_{ijmn}\in \C$, the following assertions are equivalent:
\begin{enumerate}
\item[(i)] There exists a polynomial $R_1 \in \C[x,y,z,t]$ such that $\pi(R_1) = \pi(R) \in \CSL2$ and such that $\nu'(R_1)> \nu'(R)$.
\item[(ii)] The polynomial $q$ divides $R^w$ where $R^w$ is the homogeneous polynomial given by:
\begin{equation*}
R^w = \sum_{i\nu'(x) + j\nu'(y) + m \nu'(z) + n\nu'(t) = \nu'(R)} a_{ijmn} x^i y^jz^mt^n.
\end{equation*} 
\end{enumerate}
\end{lem}

\begin{proof} The implication $(ii) \Rightarrow (i)$ is straightforward. If $q | R^w$ then we can decompose $R$ as:
\begin{equation*}
R = q R_1 + S,
\end{equation*}
where $R_1,S \in \C[x,y,z,t]$ such that $\nu'(S)> \nu'(q R_1)$. 
Hence $\pi(R_1 + S) = \pi(R)$ and $\nu'(R_1 + S) \geqslant \min( \nu'(R_1) , \nu'(S)) > \nu'(R)$ as required.

Let us prove the implication $(i) \Rightarrow (ii)$. 
Take a polynomial $R_1$ which satisfies $(i)$. Then we can write:
\begin{equation*}
R_1= R + (q-1)S,
\end{equation*}
where $S \in \C[x,y,z,t]$.
Let us prove that $R^w + q S^w = 0$. 
As $\nu'(R_1) >\nu'(R)$, the above equality implies that $\nu'(qS)= \nu'(R)$.
Let us suppose by contradiction that $R^w + q S^w \neq 0$. 
This implies that $\nu'(R_1^w) = \nu'(R^w + q S^w) = \nu'(R^w)$ which also contradicts our assumption.     
Hence $R^w + q S^w = 0 $ and $q | R^w$ as required.
\end{proof}

The above lemma proves that the supremum $\nu(f)$ in Proposition \ref{prop_minimal_valuation} is a maximum which is reached on a value $R \in \C[x,y,z,t]$ such that $\pi(R) = f$ and such that $q$ does not divide $R^w$.

\begin{proof}[Proof of Proposition \ref{prop_minimal_valuation}]

Fix $p \in \C^4$ and $\alpha \in (\mathbb{R}^- \setminus \{0 \})^4$.
Observe that for any $f_1 \in \CSL2$, the value $\nu(f_1)$ is smaller or equal than $0$. 
If $a \in \C^*$, then by definition $\nu(a) = \nu'(a) = 0$. 

\medskip

Fix $f_1,f_2 \in \CSL2$ and let us prove that $\nu(f_1 + f_2) \geqslant \min(\nu(f_1), \nu(f_2))$.
Take $R_1,R_2 \in \C[x,y,z,t]$ such that $\nu'(R_1) = \nu(\pi(R_1)) $ and $\nu'(R_2) = \nu (\pi(R_2))$. 

As $\nu_p^\alpha$ is a valuation on $\C[x,y,z,t]$, we have by definition:
\begin{equation*}
\nu'(R_1 + R_2) \geqslant \min (\nu'(R_1), \nu'(R_2)) = \min ( \nu(\pi(R_1)) , \nu(\pi(R_2))). 
\end{equation*}

In particular, the maximal value in the right hand side yields: 
\begin{equation*}
\nu(f_1 + f_2) \geqslant \min(\nu(f_1), \nu(f_2)).
\end{equation*}

We prove that $\nu( \pi(f_1f_2)) = \nu( \pi(f_1)) + \nu( \pi(f_2))$.
Take two polynomials $R_1$ and $R_2 \in \C[x,y,z,t]$ such that $\pi(R_1) = f_1$, $\pi(R_2) = f_2$ and $\nu(f_1) = \nu_p^\alpha(R_1)$, $\nu(f_2) = \nu_p^\alpha(R_2)$. 
Observe that $(R_1 R_2)^w= R_1^w R_2^w$. 
As the polynomial $q$ does not divide either $R_1^w$ or $R_2^w$, it does not divide $(R_1R_2)^w$ since the ideal generated by $q$ is a prime ideal. Hence by Lemma \ref{lem_maximality_valuation}, one has $\nu(f_1f_2) = \nu_p^\alpha( R_1 R_2) =\nu_p^\alpha(R_1^w) + \nu_p^\alpha(R_2^w) =  \nu(f_1) + \nu(f_2)$ as required.

By construction, the valuation $\nu$ is centered at infinity since $\nu$ takes negative values on the regular functions on the quadric.

Let us prove that the valuation $\nu$ is minimal, take another valuation $\nu' : \CSL2 \to \mathbb{R}^- \cup \{ +\infty \}$ such that $\nu'(\pi(x-x_0))= \nu(x-x_0)$, $\nu'(\pi(y-y_0)) = \nu(\pi(y-y_0))$, $\nu'(\pi(z-z_0))= \nu(\pi(z-z_0))$ and $\nu'(\pi(t-t_0)) = \nu(\pi(t-t_0))$.
Then the map $\hat\nu':  R \in \C[x,y,z,t] \rightarrow \nu'(\pi(R))$ defines a semi-valuation on $\C[x,y,z,t]$. 
Remark that the monomial valuation $\nu_p^\alpha$ is minimal in $\C[x,y,z,t]$, in the sense that for any $R\in \C[x,y,z,t]$:
\begin{equation*}
\hat\nu' (R) \geqslant \nu_p^\alpha(R).
\end{equation*} 
Take $f \in \CSL2$ and choose a polynomial $R\in \C[x,y,z,t]$ such that $\nu_p^\alpha(R)= \nu(f)$,  the above inequality implies:
\begin{equation*}
\nu'(f) \geqslant \nu(f),
\end{equation*}
hence, $\nu$ is also minimal.
\end{proof}



\subsection{Parachute} \label{subsection_parachute}

In this subsection, we define the parachute associated to a component of a tame automorphism. 
For any $4$-tuple $(R_1,R_2,R_3,R_4) \in \C[x,y,z,t]$ of polynomials, we write:
\begin{equation*}
dR_1 \wedge dR_2 \wedge dR_3 \wedge dR_4 = \Jac(R_1,R_2,R_3,R_4) dx \wedge dy \wedge dz\wedge dt,
\end{equation*}
with $\Jac(R_1,R_2,R_3,R_4) \in \C[x,y,z,t]$.
\begin{defi} The pseudo-jacobian of a triple $(f_1,f_2,f_3)$ of regular functions on $\SL2$ is defined by
%
\begin{equation*}
j (f_1,f_2,f_3) := \Jac(q , R_1,R_2,R_3))_{|\SL2},
\end{equation*}
where $R_i\in \C[x,y,z,t]$ are polynomials such that $\pi(R_i) = f_i$ for $i=1,2,3$.
\end{defi}

Observe that the pseudo-jacobian $j(f_1,f_2,f_3)$ is well-defined since any two representatives $R_1,R_2\in \C[x,y,z,t]$ of the same equivalence class in $\CSL2$ are equal modulo $(q-1)$.  

\begin{rem} Geometrically, the Poincare residue of the map induced by the rational map $(x,y,z,t) \mapsto (f_1,f_2,f_3,f_4)$ for $f_i \in \C[Q]$ is given by $1/j(f_1,f_2,f_3) df_1\wedge df_2\wedge df_3 $. In other words, $j(f_1,f_2,f_3)$ controls how the volume form $\Omega$ on the quadric is changed by the induced rational map.
\end{rem}

\begin{lem} \label{lem_elementary_jacobian_estimate} Let $\nu \in \mathcal{V}_0$ be a valuation. For any $f_1,f_2,f_3 \in \CSL2$, we have:
\begin{equation*}
\nu(j(f_1,f_2,f_3)) \geqslant \nu(f_1) + \nu(f_2) + \nu(f_3) - \nu(xt)
\end{equation*}
\end{lem}

\begin{proof}
Fix $f_1,f_2,f_3 \in \CSL2$ and a  valuation $\nu \in \mathcal{V}_0$. 
By definition, there exists a valuation $\nu' : \C[x,y,z,t] \to \mathbb{R}^- \cup \{ +\infty\}$ such that $\nu (P) = \sup \{  \nu'(R) | \pi(R) = P\}$ for any $P\in \CSL2$ where $\pi : \C[x,y,z,t] \to \CSL2$ is the canonical projection.
Take $R_1, R_2,R_3 , R_4 \in \C[x,y,z,t]$. We first claim that:
\begin{equation*}
\nu'(\Jac(R_1,R_2,R_3,R_4)) \geqslant \nu'(R_1) + \nu'(R_2) + \nu'(R_3) + \nu'(R_4) - \nu'(xyzt).
\end{equation*}
Let $a_{I}^{(k)} \in \C$ be the coefficients of $R_k$ for $k = 1,2,3,4$ so that:
\begin{equation*}
R_k = \sum_{I=(i_1,i_2,i_3,i_4)} a_I^{(k)} x^{i_1} y^{i_2} z^{i_3} t^{i_4}.
\end{equation*}
One obtains by linearity that $\Jac(R_1,R_2,R_3,R_4)$ is a sum of monomials where the valuation of each term is greater or equal to:
\begin{equation*}
\nu'(R_1) + \nu'(R_2) + \nu'(R_3) + \nu'(R_4) - \nu'(xyzt).
\end{equation*} 
Hence:
\begin{equation*}
\nu'(\Jac(R_1,R_2,R_3,R_4)) \geqslant \nu'(R_1) + \nu'(R_2) + \nu'(R_3) + \nu'(R_4) - \nu'(xyzt).
\end{equation*} 
In particular, we apply to $R_4 = q$ and obtain:
\begin{equation*}
\nu'(\Jac(R_1,R_2,R_3, q) ) \geqslant \nu'(R_1) + \nu'(R_2) + \nu'(R_3) - \nu'(xt),
\end{equation*}
since $\nu'(q) = \nu'(xt)= \nu'(yz)$.
Take $f_1,f_2,f_3 \in \CSL2$, by Lemma \ref{lem_maximality_valuation}, there exists $R_1,R_2,R_3 \in \C[x,y,z,t]$ such that $\pi(R_i) = f_i\in \CSL2$ and $\nu(f_i) = \nu'(R_i)$ for all $i = 1,2,3$, the above inequality implies:
\begin{equation*}
\nu(j(f_1,f_2,f_3)) \geqslant \nu' (\Jac(q, R_1,R_2,R_3))  \geqslant \nu'(R_1) + \nu'(R_2) + \nu'(R_3) - \nu'(xt),
\end{equation*}
where the first inequality follows from the definition of $\nu$. 
Observe that $\nu'(xt) = \nu(xt)$ by Lemma \ref{lem_maximality_valuation}, hence we have proven that:
\begin{equation*}
\nu(j(f_1,f_2,f_3)) \geqslant \nu(f_1) + \nu(f_2) +\nu(f_3)  - \nu(xt),
\end{equation*}
as required.
\end{proof}

The regular function $j(f_1, f_2,f_3)$ may vanish so that $\nu(j(f_1,f_2,f_3))$ may be equal to $+\infty$, even if $\nu \in \mathcal{V}_0$.
\begin{lem} \label{lem_parachute_well_def}
For  any algebraically independent functions $f_1,f_2 \in \CSL2$, one of the four regular functions $j(x,f_1,f_2),j(y,f_1,f_2), j(z,f_1,f_2), j(t,f_1,f_2)$ is not identically zero. 
In particular,
$$\min( \nu ( j(x, f_1,f_2)) , \nu(j(y,f_1,f_2)), \nu( j(z,f_1,f_2)) ,\nu(j(t,f_1,f_2))) < + \infty,$$
for any valuation $\nu  \in \mathcal{V}_0$.
\end{lem}

\begin{proof} Consider two algebraically independent regular functions $f_1,f_2 \in \CSL2$ and 
 suppose by contradiction that  $j(x, f_1,f_2) = j(y, f_1,f_2) = j(z,f_1,f_2) = j(t,f_1,f_2) =0$.
  If $\Ks\subset L$ are two fields of characteristic zero, then  \cite[Section VIII.5, Proposition 5.5]{lang} states that
 \begin{equation} \label{eq_dim_deriv}
 \degtr_{\Ks} L = \dim_L \Der_K(L),
 \end{equation}
 where $\Der_{\Ks}(L)$ denotes the vector space of $\Ks$ derivations of $L$. 
 When $K = \C(f_1,f_2)$ and $L = \C(\SL2)$, the above equality implies that any two $\C(f_1,f_2)$-derivations are proportional. 
 Since $j(x,f_1,\cdot), j(y,f_1, \cdot ), j(z,f_1,\cdot) $ and $j(t, f_1,\cdot)$ are $\C(f_1,f_2)$-derivations,
 this translates as:
\begin{equation*}
j(x,f_1, x) j(y,f_1,y) - j(x,f_1,y) j(y,f_1,x) = 0 \in \CSL2,
\end{equation*}
Hence,
\begin{equation*}
j(x,f_1,y) = 0 \in \C(\SL2).
\end{equation*}
The same argument also yields:
\begin{equation*}
j(f_1,x,y) = j(f_1,x,z) = j(f_1,x,t) = j(f_1,y,z) = j(f_1,y,t) = j(f_1,z,t)= 0.
\end{equation*}
Hence the maps $j(x,y, \cdot), j(x,z, \cdot), j(y,z, \cdot)$ are also $\C(f_1)$-derivations.
By \eqref{eq_dim_deriv} applied to $K= \C(f_1)$ and to $L = \C(\SL2)$, the space of $\C(f_1)$ derivations is $2$-dimensional and there exists $a,b,c \in \C(\SL2)$ such that:
\begin{equation*}
a j(x, y , \cdot ) + b j(x,z,\cdot) + c j(y,z,\cdot) = 0,
\end{equation*}
where $a,b$ and $c$ are not all equal to zero. 
Suppose that $a \neq 0$, we then conclude that:
\begin{equation*}
a j(x,y,z) = 0 \in \C(\SL2),
\end{equation*}
which in turn implies that $j(x,y,z) = x = 0 \in \CSL2$ and this is impossible. 
%
%
\end{proof}

\begin{defi} For any monomial valuation $\nu \in \mathcal{V}_0$ and for any algebraically independent regular functions $f_1,f_2 \in \CSL2$, the parachute $\nabla(f_1,f_2)$ with respect to the valuation $\nu$ is defined by the following formula:
\begin{equation*}
\nabla(f_1,f_2) = \min( \nu ( j(x, f_1,f_2)) , \nu(j(y,f_1,f_2)), \nu( j(z,f_1,f_2)) ,\nu(j(t,f_1,f_2))) - \nu(f_1) - \nu(f_2).
\end{equation*}
\end{defi}

Observe that Lemma \ref{lem_parachute_well_def} and Lemma \ref{lem_elementary_jacobian_estimate} imply that $\nabla(f_1,f_2)$ is finite and is strictly greater than zero.

For any polynomial $R\in \C[x,y]$, we write by $\partial_2 R \in \C[x,y]$ the partial derivative with respect to $y$. 
The next identity is similar to \cite[Lemma 5]{lamy_venereau} and is one of the main ingredient to find an upper bound on the value of a valuation.
\begin{lem} \label{lem_proof_para} Let $\nu \in \mathcal{V}_0$, let $R \in \C[x,y]$ and let $f_1,f_2 \in \CSL2$ be two algebraically independent elements. 
Suppose that there exists an integer $n$ such that $\nu(\partial_2^n R(f_1,f_2))$ is equal to the value on $\partial_2^n R$ of the monomial valuation in two variables having weight $\nu(f_1)$ and $ \nu(f_2)$ on $x$ and $y$ respectively.
Then 
\begin{equation*}
\nu (R(f_1,f_2)) < \deg_y(R) \nu(f_2) + n \nabla(f_1,f_2).
\end{equation*} 
\end{lem}

\begin{proof} Lemma \ref{lem_elementary_jacobian_estimate} proves that $j(x, f_1, f_2) \geqslant \nu(f_1) +\nu(f_2)-\nu(xt) $ for any $f_1,f_2 \in \CSL2$. 
Using this and the fact that $j(x,f_1, \cdot)$ is a derivation, we obtain:
\begin{equation*}
 \nu(\partial_2 R(f_1,f_2) j (x, f_1,f_2)) = \nu(j(x,f_1 , R(f_1,f_2))) \geqslant \nu(f_1) + \nu(R(f_1,f_2)) + \nu(x) - \nu(xt). 
\end{equation*}
In particular since $\nu(x) - \nu(xt) = -\nu(t)> 0$, this yields:
\begin{equation*}
\nu(\partial_2 R(f_1,f_2) > - (\nu( j(x, f_1,f_2) ) - \nu(f_1) - \nu(f_2) ) + \nu(R(f_1,f_2)) - \nu(f_2) . 
\end{equation*}
A similar argument with $y,z$ and $t$ also gives:
\begin{equation} \label{eq_init}
\nu(\partial_2 R(f_1,f_2)) > - \nabla(f_1,f_2) + \nu(R(f_1,f_2)) - \nu(f_2).
\end{equation}
We apply \eqref{eq_init} inductively and obtain the following inequalities:
\begin{align*}
\nu(\partial_2^2 R(f_1,f_2)) &> - \nabla(f_1,f_2) + \nu(\partial_2 R(f_1,f_2)) - \nu(f_2),\\
\ldots& \\
\nu(\partial_2^n R(f_1,f_2)) &> - \nabla(f_1,f_2) + \nu(\partial_2^{n-1} R(f_1,f_2)) - \nu(f_2).
\end{align*}
This implies that:
\begin{equation*}
\nu( \partial_2^n R(f_1,f_2)) > - n \nabla(f_1,f_2) - n \nu(f_2) + \nu(R(f_1,f_2)).
\end{equation*}
As $\nu(\partial_2^nR(f_1,f_2))$ is equal to the value of the monomial valuation with weight $(\nu(f_1),\nu(f_2))$ applied to $\partial_2^n(R)$, the last inequality rewrites as:
\begin{equation*}
(\deg_y R - n ) \nu(f_2) > \nu(\partial_2^n R(f_1,f_2)) \geqslant -  n \nabla(f_1,f_2) - n \nu(f_2) + \nu(R(f_1,f_2)).
\end{equation*}
Hence,
\begin{equation*}
\nu(R(f_1,f_2)) < \deg_y(R) \nu(f_2) + n \nabla(f_1,f_2),
\end{equation*}
as required.
\end{proof}

\subsection{Key polynomials} \label{subsection_key_polynomials}
Let us explain how one can find a polynomial which satisfies the hypothesis of Lemma \ref{lem_proof_para}. 

Consider $\mu : \C[x,y] \to \mathbb{R}^- \cup \{ +\infty \} $ any valuation and $\mu_0 : \C[x,y] \to \mathbb{R}^- \cup \{ +\infty \}$ the monomial valuation having weight $\mu(x)$ and $\mu(y)$ on $x$ and $y$ respectively.
For any polynomial $R \in \C[x,y]$, we write by $\overline{R}  \in \C[x,y]$ the homogeneous polynomial given by:
\begin{equation*}
\overline{R} = \sum_{ i \mu(x) + j \mu(y) = \mu_0(R) } a_{ij}x^i y^j,
\end{equation*}
with $a_{ij} \in \C$ such that $R = \sum_{ij} a_{ij} x^i y^j$.

\begin{prop} \label{prop_key_polynomial} Consider $\mu : \C[x,y] \to \mathbb{R}^- \cup \{ +\infty \}$ any valuation and $\mu_0$ the monomial valuation having weights $\mu(x)$ and $\mu(y)$ on $x$ and $y$ respectively. 
The following properties are satisfied.
\begin{enumerate}
\item[(i)] For any $R \in \C[x,y]$, one has $\mu(R) \geqslant \mu_0(R)$.
\item[(ii)] If $\mu \neq \mu_0$, then there exists two coprime integers $s_1,s_2$ satisfying $s_1 \mu(x) = s_2\mu(y)$ and a unique constant $\lambda \in \C$ for which the polynomial
$
H = x^{s_1} - \lambda y^{s_2}$
 satisfies $\mu(H) > \mu_0(H)$.
\item[(iii)] For any $R\in \C[x,y]$, one has $\mu(R) > \mu_0(R)$ if and only if $H | \overline{R}$.
\end{enumerate}
\end{prop}

The polynomial $H$ associated to $\mu$ is called a key polynomial associated to $\mu$.

\begin{proof}

Let us prove assertion $(i)$. Write $R \in \C[x,y]$ as $R = \sum a_{ij}x^i y^j$ where $a_{ij} \in \C$.
Recall that the fact that $\mu_0$ is monomial implies that:
\begin{equation*}
\mu_0( R) = \min \{ i \mu_0(x) + j \mu_0(y) \ | \ a_{ij} \neq 0 \}. 
\end{equation*}
Also, $\mu$ is a valuation, hence:
\begin{equation*}
\mu(R) \geqslant \min \{ i \mu_0(x) + j \mu_0(y) \ | \ a_{ij} \neq 0 \} = \mu_0(R). 
\end{equation*}
We have thus proved that $\mu(R) \geqslant \mu_0(R)$, as required.
\smallskip

\textbf{Step 1}: Fix $s_1,s_2$ two coprime integers and $\lambda \in \C$.
Suppose that  $s_1 \mu(x) = s_2 \mu(y) $ and that the polynomial $ H = x^{s_1} - \lambda y^{s_2}$ satisfies $\mu(H) >\mu_0(H)$, we prove that $\lambda$ is unique.
Take $\lambda' \neq \lambda \in \C$, then 
\begin{equation*}
\mu( x^{s_1} - \lambda' y^{s_2}) = \mu( H + (\lambda-\lambda')y^{s_2}) = s_2\mu(y),
\end{equation*} 
since $\mu(H) > \mu ((\lambda - \lambda')y^{s_2})$. 
Hence $\mu(x^{s_1} - \lambda' y^{s_2}) = \mu_0(x^{s_1} - \lambda' y^{s_2})$ for any $\lambda' \neq \lambda$.

\smallskip

\textbf{Step 2}: Choose two integers $s_1,s_2$ such that $s_1 \mu(x) = s_2 \mu(y)$.  We prove that there exists $\lambda\in \C^*$ such that $\mu( x^{s_1} - \lambda y^{s_2}) > s_1 \mu(x) = s_2 \mu(y)$. Suppose by contradiction that for any $\lambda \in \C$, one has $\mu(x^{s_1} - \lambda y^{s_2}) = s_1\mu(x)$. We claim that $\mu(R) = \mu_0(R)$ for any polynomial $R \in \C[x,y]$. 
Fix $R \in \C[x,y]$. Observe that if $R$ is a homogeneous polynomial with respect to the weight $(\mu(x) , \mu(y))$, then $R$ is of the form:
\begin{equation*}
 R = \alpha x^{k_0} \prod_{i } (x^{s_1} - \lambda_i y^{s_2})
\end{equation*}
where $\alpha,\lambda_i\in\C^*$ and $k_0 \in \mathbb{N}$. 
Our assumption implies that $\mu(R) = \mu_0(R)$ for any homogeneous polynomial $R$.

If $R$ is a general polynomial, then $R$ can be decomposed into $R = \sum_i R_i$ where each polynomial $R_i$ is homogeneous. Since $\mu(R_i) = \mu_0(R_i)$ for each $i$, this proves that $\mu(R) = \mu_0(R)$ for any $R \in \C[x,y]$, which contradicts our assumption.
We have thus proven assertion $(ii)$.
\smallskip

\textbf{Step 3}: We prove assertion $(iii)$. Suppose that $\mu(R) = \mu(\overline{R}) $, we claim that $H$ does not divide $\overline{R}$. 
Observe that $\mu_0(\overline{R}) = \mu_0 (R)$, hence $\mu(R) = \mu(\overline{R}) = \mu_0(\overline{R})$. The equality $\mu(\overline{R}) = \mu_0(\overline{R})$ implies that $H$ does not divide $\overline{R}$ by the previous argument.

Conversely, suppose that $H$ does not divide $\overline{R}$, we prove that $\mu(R) = \mu_0(R)$.
Since $H$ does not divide $\overline{R}$, we have that $\mu( \overline{R})= \mu_0(\overline{R})$. 
Decompose $R$ into $R = \overline{R} + S$ where $S\in \C[x,y]$ such that $\mu_0(S) > \mu_0(R)$. We have that:
\begin{equation*}
\mu(R) \geqslant \min( \mu(\overline{R}), \mu(S)).
\end{equation*} 
Since $\mu(S) \geqslant \mu_0(S) > \mu_0(\overline{R}) = \mu(\overline{R})$, we have thus:
\begin{equation*}
\mu(R) = \mu(\overline{R}) = \mu_0(\overline{R}),
\end{equation*}
 as required. We have proven that $\mu(R) = \mu_0(R)$ if and only if $H $ does not divide $\overline{R}$ which is equivalent to assertion $(iii)$.
\end{proof}

\subsection{Parachute inequalities}
\label{subsection_parachute_ineq}
We introduce various notions of resonances of components of a tame automorphism. These notions will play an important role in the theorem below. Consider a valuation $\nu \in \mathcal{V}_0$ and a component $(f_1,f_2)$ of a tame automorphism. 
We are interested in the value of $\nu$ on $R(f_1,f_2)$ where $R \in \C[x,y]$. The estimates of the value $\nu(R(f_1,f_2))$ will depend on the possible values of the pair $(\nu(f_1), \nu(f_2))$. We shall distinguish the following three cases: 
\begin{enumerate}
\item The family $(\nu(f_1),\nu(f_2))$ is $\mathbb{Q}$-independent and we say that the component $(f_1,f_2)$ is \textbf{non resonant} with respect to $\nu$.
\item  There exists two coprime integers $s_1,s_2$ such that $s_1>s_2 \geqslant 2$ or $s_2 > s_1 \geqslant 2$ such that $s_1 \nu(f_1) = s_2 \nu(f_2)$ and we say in this case that the component $(f_1,f_2)$ is \textbf{properly resonant} with respect to $\nu$.
\item  Either $\nu(f_1)$ is a multiple of $\nu(f_2)$ or $\nu(f_2)$ is a multiple of $\nu(f_1)$ and there exists a polynomial $H \in \C[x,y]$ of the form $x - \lambda y^k$ where $k\in \mathbb{N}^*, \lambda \in \C^*$ such that $\nu(H(f_1,f_2)) > \nu(f_1) = k \nu(f_2)$. In this case, the component $(f_1,f_2)$ is called \textbf{critically resonant} with respect to $\nu$.
\end{enumerate}

\begin{ex} When $\nu = - \deg : \C[Q] \to \mathbb{R}^-\cup \{+ \infty \}$, the family $(x,y)$ is not critically resonant, but it is neither properly resonant nor non resonant (in particular there is no alternative). 
However, $(x,y)$ is non resonant for the monomial valuation with weight $(-\sqrt{2}, -\sqrt{3},-\sqrt{2},-\sqrt{3})$ on $(x,y,z,t)$.
\end{ex}

\begin{ex} Take $f_1 = x , f_2 = y + x^2 \in \C[Q]$, then $(f_1,f_2)$ is critically resonant with respect to the valuation $ \ord_{H_\infty} = -\deg$.
\end{ex}

\begin{ex} Take $f_1 = z + x^2 , f_2 = y + x^3 \in \C[Q]$, then $(f_1,f_2)$ is properly resonant with respect the valuation $\ord_{H_\infty} = -\deg$.
\end{ex}

For $\nu \in \mathcal{V}_0$ and $(f_1,f_2)$ a component of a tame automorphism, the following theorem allows us to estimate the value of $\nu$ on $R(f_1,f_2)$ only when $(f_1,f_2)$ is not critically resonant.

%
%

%
%
%
%
\begin{thm} \label{thm_precise_parachute} Let $\nu \in \mathcal{V}_0$ be a valuation and let $\nu_0$ be the monomial valuation on $\C[x,y]$ with weight $(\nu(f_1), \nu(f_2))$ with respect to $(x,y)$. The following assertions hold.
\begin{enumerate}
\item[(i)] For any polynomial $R\in \C[x,y]$, one has the lower bound $\nu(R(f_1,f_2)) \geqslant \nu_0(R(x,y))$.
\item[(ii)] If the component $(f_1,f_2)$ is non resonant with respect to $\nu$, then for any polynomial $R\in \C[x,y]$, one has $\nu(R(f_1,f_2)) = \nu_0(R(x,y))$. 
\item[(iii)] Suppose that the component $(f_1,f_2)$ is properly resonant with respect to $\nu$ and let $s_1,s_2$ be two coprime integers such that $s_1 \nu(f_1) = s_2 \nu(f_2)$, then for any polynomial $R \in \C[x,y]$, either $\nu(R(f_1,f_2)) = \nu_0(R(x,y))$ or $\nu(R(f_1,f_2)) > \nu_0(R(x,y))$ and we have:
\begin{equation*}
\nu(R(f_1,f_2)) < \min \left \{ \left (s_1 - 1 - \dfrac{s_1}{s_2} \right ) \nu(f_1) , \left (s_2 - 1 - \dfrac{s_2}{s_1} \right )\nu(f_2) \right  \}.
\end{equation*}
\end{enumerate}
\end{thm}

\begin{rem} Observe that in assertion $(iii)$, only one inequality is relevant. Suppose for example that $\nu(f_1) < \nu(f_2)$, then $s_1 < s_2$ and the value $(s_2 -1 - s_2/s_1) \nu(f_2)$ is greater or equal to $0$ whereas $\nu(R(f_1,f_2))< 0$.
\end{rem}

Remark that the inequalities in Theorem \ref{thm_precise_parachute} are strict and this fact is crucial in our proof.
Before giving the proof of Theorem \ref{thm_precise_parachute}, we state two consequences of this theorem below. 
\begin{cor} \label{cor_parachute}  Let $\nu \in \mathcal{V}_0$ be a monomial valuation and let $f= (f_1, f_2, f_3, f_4)$ be an element of $\tame$. We suppose that $\nu(f_1) < \nu(f_2)$ and that $(f_1,f_2)$ is not critically resonant with respect to $\nu$. Then for any polynomial $R \in \C[x,y] \setminus \C[y]$, we have:
\begin{equation*}
\nu( f_2 R(f_1 , f_2))< \nu (f_1).
\end{equation*}
\end{cor}

\begin{proof} Two cases appear. 
Either $\nu(R(f_1,f_2))= \nu_0( R(x,y))$ where $\nu_0 $ is the monomial valuation with weight $(\nu(f_1),\nu(f_2))$ with respect to $(x,y)$, and we are finished since $R \in \C[x,y] \setminus \C[y]$.
Or $\nu(R(f_1,f_2)) > \nu_0(R(x,y))$ and there exists some integers $s_1,s_2$ such that $s_1 \nu(f_1) = s_2  \nu(f_2)$ where $s_2 > s_1 \geqslant 2$. 
Using Theorem \ref{thm_precise_parachute}.$(iii)$ and the fact that $s_1\geqslant 2$, we have thus:
\begin{equation*}
\nu(f_2 R(f_1,f_2)) < (s_1 - 1 ) \nu(f_1) < \nu(f_1),
\end{equation*}
as required.
\end{proof}

We state the second  corollary for which the constant $4/3$ appears naturally.
\begin{cor} 
 \label{cor_precise_parachute} Let $\nu \in \mathcal{V}_0$ be a valuation and let $(f_1, f_2)$ a properly resonant component with respect to $\nu$ such that $\nu(f_1) < \nu(f_2)$. 
 Then for any polynomial $R \in \C[x,y] \setminus \C[y]$, one has:
\begin{equation*}
  \nu(f_1 R(f_1,f_2)) < \dfrac{4}{3} \nu(f_1).
\end{equation*}
\end{cor}
\begin{proof} Denote by $\nu_0 : \C[x,y] \to \mathbb{R}^- \cup \{ + \infty \}$ the monomial valuation with weight $(\nu(f_1) , \nu(f_2))$ with respect to $(x,y)$. 
Two cases appear, either $\nu(R(f_1,f_2)) = \nu_0(R(x,y))$ and we are done since $\nu(R(f_1,f_2)) \leqslant 2 \nu(f_1)$ as $R\in \C[x,y] \setminus \C[y]$ or $\nu(R(f_1,f_2)) > \nu_0(R(x,y))$.
In the latter case, consider two coprime integers $s_1,s_2$ such that $s_1 \nu(f_1)= s_2 \nu(f_2) $.
Since $\nu(f_1) < \nu(f_2)$ and the component $(f_1,f_2)$ is properly resonant,   the inequality $s_2 > s_1 \geqslant 2$ holds. 
 Using Theorem \ref{thm_precise_parachute}.$(iii)$,  we obtain:
\begin{equation*}
\nu(f_1R(f_1,f_2)) < \left (s_1 - \dfrac{s_1}{s_2} \right ) \nu(f_1).
\end{equation*}
Suppose $s_1 \geqslant 3$, then $s_1 - s_1/s_2 \geqslant 2 $ as $s_1/s_2 \leqslant 1$. In particular, the above inequality gives: 
\begin{equation*}
\nu(f_1R(f_1,f_2) \leqslant 2 \nu(f_1) < \dfrac{4}{3} \nu(f_1).
\end{equation*}
The only remaining case is when $s_1 = 2$ and $s_2 > s_1 = 2$. Then $s_1/s_2 \leqslant 2/3$ and we obtain:
\begin{equation*}
\nu(f_1R(f_1,f_2)) < \left (2 - \dfrac{2}{3} \right ) \nu(f_1) = \dfrac{4}{3} \nu(f_1).
\end{equation*}
\end{proof}

\begin{proof}[Proof of Theorem \ref{thm_precise_parachute}]

Let us denote by $R = \sum a_{ij} x^i y^j$. Consider the projection $\pi_{xy} : \SL2 \to \mathbb{A}^2$ induced by the embedding of $\SL2$ into $\mathbb{A}^4$ composed with the projection onto $\mathbb{A}^2$ of the form:
\begin{equation*}
\pi_{xy} : (x,y,z,t) \in \SL2(\C)  \mapsto (x,y). 
\end{equation*}
Choose an automorphism $f$ such that $f = (f_1,f_2,f_3,f_4)$ where $f_3,f_4 \in \CSL2$.
We denote by $\mu $ the valuation on $\C[x,y]$ given by $\mu = {\pi_{xy}}_* f_* \nu$.

Observe that for any polynomial $R\in \C[x,y]$, we have $\nu(R(f_1,f_2)) = \mu(R(x,y))$ and assertion $(i)$ follows directly from Proposition \ref{prop_key_polynomial}.$(i)$. Observe also that assertion $(ii)$ follows immediately from the fact that $\nu(f_1)$ and $\nu(f_2)$ are $\mathbb{Q}$-independent.

\medskip

Let us prove assertion $(iii)$. We can suppose by symmetry that $\nu(f_1) < \nu(f_2)$.
Since the component $(f_1,f_2)$ is properly-resonant, there exists two coprime integers $s_1,s_2$ such that $s_1 \nu(f_1) = s_2 \nu(f_2)$ and such that $s_2 >s_1 \geqslant 2$. 


%
%
By Proposition \ref{prop_key_polynomial} applied to $\mu$, there exists $\lambda \in \C^*$ such that the polynomial  $H = x^{s_1} - \lambda y^{s_2}$ satisfies
\begin{equation*}
\mu(H(x,y)) = \nu (H(f_1,f_2)) > \nu_0 (H) = s_1 \nu(f_1).
\end{equation*}
For any polynomial $R\in \C[x,y]$, denote by $\overline{ R} $ be the polynomial given by:
\begin{equation*}
\overline{ R} = \sum_{i \mu(x) + j\mu(y) = \nu_0(R(x,y))} a_{ij} x^i y^j.
\end{equation*}
%
By construction, we have that there exists an integer $n\geqslant 1$ such that $\overline{R} \in (H^n) \setminus (H^{n+1})$. 

We shall use the following lemma (proved at the end of this section):
\begin{lem}\label{lem_graded_technical} Let $R \in \C[x,y]$ such that $H | \overline{R}$ . Consider the integer $n = \max \{ k \ | \  H^k \ \text{divides} \  \overline{R} \} \geqslant 1$. Then the following properties are satisfied.
\begin{enumerate}
\item[(i)]For any integer $k \leqslant n$, we have $\overline{\partial_2^k(R)} = \partial_2 \overline{R}$.
\item[(ii)]  For any integer $k \leqslant n$, we have $H^{n-k} | \partial_2^k\overline{R}$ but $H^{n-k+1} \nmid \partial_2^{k} \overline{R}$.
\end{enumerate}
\end{lem}
The above lemma implies that $\overline{ \partial_2^k R} = \partial_2^k \overline{R}$ and that $  H^{n-k})| \partial_2^k \overline{R} $ but $   H^{n-k+1} \nmid \partial_2^k \overline{R}  $ for any $k$. 
In particular, $H$ does not divide $\partial_2^n \overline{R} $ and  Proposition \ref{prop_key_polynomial}.$(iii)$ implies that:
\begin{equation*}
\mu( \partial_2^n R(x,y) ) = \nu_0 (\partial_2^n R) = \nu_0 (\partial_2^n \bar{R}).  
\end{equation*}
The previous equation translates as:
\begin{equation*}
\nu ((\partial_2^n R)(f_1,f_2)) = \nu_0 (\partial_2^n R)
\end{equation*}
and  $R$ satisfies the conditions of Lemma \ref{lem_proof_para} (for the same integer $n$), which in turn asserts that:
\begin{equation*}
\nu(R(f_1,f_2)) < \deg_y(R) \nu(f_2) + n \nabla(f_1,f_2),
\end{equation*}
Since $H^n | \bar R$, one has $\deg_y(R)\geqslant \deg_y( \overline{R}) = s_2 n $, we get:
\begin{equation*}
\nu(R(f_1,f_2)) < n \left (s_2 \nu(f_2) +  \nabla(f_1,f_2)\right ).
\end{equation*}
As $n\geqslant 1$ and $\nabla(f_1,f_2) \leqslant -\nu(f_1) - \nu(f_2)$, the above implies that 
\begin{equation*}
\nu(R(f_1,f_2)) < s_2\nu(f_2) - \nu(f_1) - \nu(f_2).
\end{equation*}
Since $s_1 \nu(f_1) = s_2\nu(f_2)$, we finally prove that:
\begin{equation*}
\nu(R(f_1,f_2)) <  \nu(f_1) \left ( s_1 - 1 - \dfrac{s_1}{s_2}\right ) ,
\end{equation*}
as required.
\end{proof}

\begin{proof}[Proof of Lemma \ref{lem_graded_technical}]
Consider a monomial valuation $\nu_0 : \C[x,y] \to \mathbb{R}^- \cup \{ +\infty\}$ with weight $(\alpha,\beta) \in (\mathbb{R}^{-,*})^2$ with respect to $(x,y)$ and $H = x^{s_1}- \lambda y^{s_2}$ where $s_1,s_2$ are coprime integers such that $s_1 \alpha = s_2 \beta$.

Let us prove assertion $(i)$ for $k = 1$. 
Fix $R \in \C[x,y]$ and write $R$ as:
\begin{equation*}
R = \sum_{ij} a_{ij} x^i y^j,
\end{equation*}
where $a_{ij} \in \C$.
The partial derivative is given explicitly by:
\begin{equation*}
\partial_2 R = \sum_{i\geq 0,j \geq 1} ja_{ij}  x^i y^{j-1}.
\end{equation*}
Since $H | \overline{R}$, one has $\overline{R} \in \C[x,y] \setminus \C[x]$ and $\nu_0( R) = \nu_0 (\overline{R})$. 
Take $(i,j)$ such that $a_{ij} \neq 0$ and $i \alpha + (j -1) \beta = \nu_0(\partial_2 R).$
Then $i \alpha + j \beta = \nu_0(\partial_2 R) + \nu_0(y) \leq \nu_0(R)  $. 
Conversely, since $H | \overline{R}$, there exists $(i,j)$ such that $i \alpha + j \beta = \nu_0(R)$ where $j\geq 1$, hence we have that $i\alpha + (j-1) \beta \geq \nu_0(\partial_2 R)$. 
Hence, $\nu_0(\partial_2 R) = \nu_0(R) - \beta $ and  $\overline{ \partial_2 R} = \partial_2 \overline{R} $. 
\smallskip 

Let us prove assertion $(ii)$ for $k =1$. 
We have that $H^n | \overline{R}$ but $H^{n+1} \nmid \overline{R}$, then we have:
\begin{equation*}
\overline{R} = H^n S,
\end{equation*}
where $S\in \C[x,y]$ is a homogeneous polynomial such that $H\nmid S$. 
By definition, 
\begin{equation*}
\partial_2 \overline{R} = n s_2 H^{n-1} y^{s_2-1} S + H^{n} \partial_2
S. \end{equation*}
Hence $H^{n-1} | \partial_2 \overline{R}$. 
Suppose by contradiction that $H^{n}| \partial_2 \overline{R}$, then this implies that $H | y^{s_2 - 1} S$ which is impossible since $H$ does not divide $S$.
We have thus proven that $H^{n-1} | \partial_2 R$ but $H^n \nmid \partial_2  R$, as required.

An immediate induction on $k \leqslant n$ proves assertion $(i)$ and $(ii)$.  
\end{proof}

\section{Proof of Theorem \ref{thm_int_degree_growth} and Theorem \ref{thm_int_degree_versus_distance}}
\label{section_comparison}

This section is devoted to the proof of Theorem \ref{thm_int_degree_growth} and Theorem \ref{thm_int_degree_versus_distance}. 
The proof of these two results  are very similar and rely on a lower bound of the degree of an automorphism $f$ by $(4/3)^p$ where $p$ is an integer that we determine.

Let us explain our general strategy. Take an automorphism $f \notin \O4$.

\textbf{Step 1}: We choose an appropriate valuation $\nu$. 

We consider a geodesic line $\gamma$ in the complex joining $[\Id]$ and $[f]$. Recall from Proposition \ref{prop_correspondence_vertex_III} that the set of $1\times 1$ squares containing $[\Id]$ is in bijection with the points on the hyperplane at infinity $H_\infty \subset \SL2 \subset \Pg^4$.
Depending on which $1 \times 1$ square the geodesic $\gamma$ near the vertex $[\Id]$ is contained, we choose accordingly a valuation $\nu$ in $\mathcal{V}_0$ centered on the corresponding point at infinity in $\bar{\SL2}$.
 \smallskip

 \textbf{Step 2} (see \S\ref{subsection_potential}): We define an integer $p$ according to the geometry of some geodesics in the complex and according to the choice of the valuation $\nu$.
 
 Recall that a path in the $1$-skeleton of $\mathcal{C}$ induces a sequence of numbers obtained by evaluating the valuation $\nu$ on the consecutive vertices.
 The integer $p$ is defined as the distance in a graph denoted $\mathcal{C}_\nu$ and encodes  the shortest path in the $1$-skeleton with minimal degree sequence. 
 
 \smallskip
 
\textbf{Step 3}: We prove that $\deg(f) \geqslant (4/3)^p$.
 
 Consider the graph $\mathcal{C}_\nu$ associated to $\nu$ and denote by $d_\nu$ the distance in this graph.
 This step is the content of the following theorem. Recall that the standard $2\times 2$ square $S_0$  is the square whose vertices are $[x],[y],[z]$ and $[t]$. 
\begin{thm} \label{thm_degree_versus_distance}  
Pick any valuation $\nu \in \mathcal{V}_0$ satisfying:
\begin{equation}
\label{eq_spe_valuation}
\max(\nu(y) + \nu(t), \nu(z) + \nu(t)) <\nu(x) < \min(\nu(y), \nu(z), \nu(t)).
\end{equation}
Consider any geodesic segment of $\mathcal{C}$ joining $[\Id]$ to a vertex $v$ of type \textsc{I} which intersects an edge of the square $S_0$, then the following assertions hold.  
\begin{enumerate}
\item We have:
\begin{equation*}
\nu(v) \leqslant \left (\dfrac{4}{3} \right )^{ d_\nu([t],v) - 1}  \max(\nu(x), \nu(y) , \nu(z), \nu(t)).
\end{equation*}
\item For any valuation  $\nu' \in \mathcal{V}_0$ satisfying \eqref{eq_spe_valuation}, we have:
\begin{equation*}
d_\nu( [t] , v) = d_{\nu'} ( [t] , v).
\end{equation*}
\end{enumerate}

\end{thm}

\bigskip

The proof of Theorem \ref{thm_degree_versus_distance} basically proceeds by induction on the distance between $[t]$ and $v$ in the graph $\mathcal{C}_\nu$. 
The essential ingredient to bound below the degree inductively are the parachute inequalities stated in Theorem \ref{thm_precise_parachute}. 
We explain in \S \ref{subsection_R} how to arrive to the situation where these inequalities can be applied using the local geometry near the vertices of type I (i.e the geometry of its link).
We then use these arguments to compute the degree or estimate the valuation $\nu$ when one  passes from one square to another in each possible situation, this is done successively in \S\ref{degree_estimates_adjacent}, 
\S\ref{subsection_bound_stab_rev}, \S \ref{subsection_degree_non_extremal} and \S\ref{subsection_degree_adherent_minimal}.

Once we conjugate appropriately to arrive to the situation of Theorem \ref{thm_degree_versus_distance}, we then  deduce directly both Theorem \ref{thm_int_degree_growth} and Theorem \ref{thm_int_degree_versus_distance}.




\subsection{The graph $\mathcal{C}_\nu$ associated to a valuation and the orientation of certain edges of the complex}
\label{subsection_potential}
Fix a valuation $\nu \in  \mathcal{V}_0$. Given any automorphism $f = (f_1,f_2,f_3,f_4) \in \tame$, 
we remark that $\nu(f_1)$ does not depend on the choice of representative of the class $[f_1]$ so that $\nu$ induces a function on the vertices of type I of $\mathcal{C}$. 

\smallskip

We say that a vertex $v \in \mathcal{C}$ of type I is \textbf{$\nu$-minimal} (resp. \textbf{$\nu$-maximal}) in a $2\times 2$ square $S$ if $\nu(v)$ is strictly smaller (resp. greater)  than the value of the valuation $\nu$ on every other vertices of type I of $S$.
Observe that for some valuations, two vertices of type I can have the same value on $\nu$, hence there can be no $\nu$-minimal or $\nu$-maximal vertices.

We now define a graph $\mathcal{C}_\nu$ associated to a valuation $\nu \in \mathcal{V}_0$ as follows: 
\begin{enumerate}
\item the vertices are the vertices of $\mathcal{C}$ type \textsc{I};
\item one draws an edge between two vertices $v_1$ and $v_2$ of $\mathcal{C}'$ if there exists a $2\times 2$ square $S$ centered at a vertex of type III in $\mathcal{C}$ containing $v_1,v_2$ such that the vertices $v_1$, $v_2$ belong to an edge of $S$ or $v_1$ and $v_2$ are the $\nu$-minimal and $\nu$-maximal vertices of $S$ respectively.
\end{enumerate}

Observe that whenever there is no $\nu$-maximal or minimal vertex in a $2 \times 2$ square $S$ centered at a point of type III, then we only draw the four edges of the square $S$.

The graph $\mathcal{C}_\nu$ is endowed with the distance $d_\nu$ such that its the edges have length $1$. 

\begin{lem}
The graph $\mathcal{C}_\nu$ is a connected metric graph.
\end{lem}
\begin{proof}
This follows from the fact that the $1$-skeleton of $\mathcal{C}$ is connected. 
\end{proof}
\medskip

Since we will exploit the properties of this function on the vertices of type I, we introduce the following convention on the figures.
Take an edge of length $2$ between two type I vertices $v_1,v_2$, then we put an arrow pointing to $v_2$ if $\nu(v_2) < \nu(v_1)$ as in the following figure.
\begin{center}
\begin{tikzpicture} 
\draw (0, 0) -- (2,0);
\draw (0,0) node[blue] {$\circ$} node[left] {$v_1$};
\draw (2,0) node[blue] {$\circ$} node[right] {$v_2$};
\fill (1.0,-0.2)--(1.0,0.2) -- (1.2,0.0) --cycle;
\end{tikzpicture}
\end{center}


%
%
\smallskip

\begin{lem} \label{lem_arrow_square} Let $\nu : \CSL2   \to \mathbb{R} \cup \{ + \infty \}$ be a valuation which is trivial over $\C^*$ and such that $\nu(x) ,\nu(y), \nu(z), \nu(t) < 0$. 
Let $S$ be a $2\times 2$ square of the complex $\mathcal{C}$ centered at a type \textsc{III} vertex.
Suppose $S$ has a unique $\nu$-maximal vertex (resp. $\nu$-minimal), then there exists a unique $\nu$-minimal (resp. $\nu$-maximal) vertex and the $\nu$-minimal and $\nu$-maximal vertices are at distance $2\sqrt{2}$ in  $\mathcal{C}$. 
\end{lem}

Let $S$ be a $2\times 2$ square centered at a vertex of type III which satisfies the conditions of Lemma \ref{lem_arrow_square}.
 and let $\phi$ be the associated isometry. 
 Denote by $[x_1],[y_1],[z_1]$ and $[t_1]$ the vertices of type I of the square $S$ where $x_1,y_1,z_1,t_1 \in \CSL2$ such that the vertex $[x_1]$ is $\nu$-minimal and $[t_1]$ is $\nu$-maximal in $S$. Then there exists a unique isometry $\phi : S \to [0,1]^2$ such that:
\begin{equation*}
\phi([x_1]) = (2,2),
\end{equation*}
and 
\begin{equation*}
\phi([t_1]) =(0,0),
\end{equation*}
and such that the horizontal edges of $S$ are given the geodesic segments between $[x_1]$ and $[y_1]$, and between $[z_1]$ and $[t_1]$.
\smallskip

Using this convention, Lemma \ref{lem_arrow_square} implies that we are in the following situation:
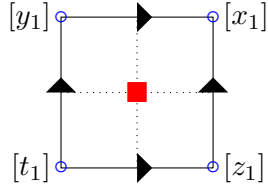
\begin{center}
\begin{tikzpicture}
\draw (2,0) -- (0,0) --(0, 2) --(2,2)-- (2,0);
\fill (1.0,-0.2)--(1.0,0.2) -- (1.2,0.0) --cycle;

\fill (2.2,1.0)--(1.8,1.0) -- (2.0,1.2) --cycle;
\fill (-0.2,1.0)--(0.2,1.0) -- (0.0,1.2) --cycle;
\fill (1.0,2.2)--(1.0,1.8) -- (1.2,2.0) --cycle;

\draw (2,2) node[right] {$[x_1]$};
\draw (2,0) node[right] {$[z_1]$};
\draw (0,0) node[left] {$[t_1]$};
\draw (0,2) node[left] {$[y_1]$};
\draw (0,0) node[blue] {$\circ$};
\draw (2,0) node[blue] {$\circ$};
\draw (2,2) node[blue] {$\circ$};
\draw (0,2) node[blue] {$\circ$};
\draw[dotted] (1,0)--(1,2);
\draw[dotted] (0,1)--(2,1);
\draw (1, 1)  node[red] {$\blacksquare$};

\end{tikzpicture}
\end{center}
In particular, the subgraph of $\mathcal{C}'$ containing the vertices of $S$ looks as follows:
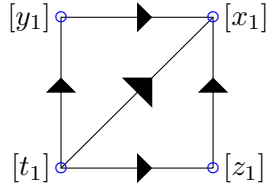
\begin{center}
\begin{tikzpicture}
\draw (2,0) -- (0,0) --(0, 2) --(2,2)-- (2,0);
\fill (1.0,-0.2)--(1.0,0.2) -- (1.2,0.0) --cycle;

\fill (2.2,1.0)--(1.8,1.0) -- (2.0,1.2) --cycle;
\fill (-0.2,1.0)--(0.2,1.0) -- (0.0,1.2) --cycle;
\fill (1.0,2.2)--(1.0,1.8) -- (1.2,2.0) --cycle;

\draw(0,0) -- (2,2);
\fill (1.2,0.8)--(0.8,1.2) -- (1.2,1.2) --cycle;

\draw (2,2) node[blue] {$\circ$} node[right] {$[x_1]$};
\draw (2,0) node[blue] {$\circ$} node[right] {$[z_1]$};
\draw (0,0) node[blue] {$\circ$} node[left] {$[t_1]$};
\draw (0,2) node[blue] {$\circ$} node[left] {$[y_1]$};
\end{tikzpicture}
\end{center}
 
\begin{proof}[Proof of Lemma \ref{lem_arrow_square}]
Let $S$ be a $2 \times 2$ square satisfying the hypothesis of the Lemma. 
Denote  $[t_1]$ the  $\nu$-maximal vertex of $S$. Denote also by $[z_1], [y_1], [x_1]$ the type I vertices of $S$ such that the edges between $[t_1]$ and $[z_1]$, between $[t_1]$ and $[y_1]$ are horizontal and vertical respectively.  

Observe that $\nu(x_1), \nu(y_1) , \nu(z_1), \nu(t_1) < 0$ and that:
\begin{equation*}
\nu(x_1t_1 - y_1 z_1 ) = \nu( 1) = 0.
\end{equation*}  
This implies that:
\begin{equation*}
\nu(x_1) + \nu(t_1) = \nu(y_1) + \nu(z_1).
\end{equation*}
In particular, $\nu(t_1)>\nu(y_1)$ implies that:
\begin{equation*}
\nu(x_1) < \nu(z_1).
\end{equation*}
By symmetry, we also prove that $\nu(x_1) < \nu(y_1)$ and this implies that $[x_1]$ is the unique $\nu$-minimal vertex of $S$, as required.  
\end{proof} 
 

%
%
%
%

Observe that for two distinct valuations $\nu_1, \nu_2 \in \mathcal{V}_0$, the graphs $\mathcal{C}_{\nu_1}$ and $\mathcal{C}_{\nu_2}$ are not in general equal. 

\begin{lem} \label{lem_easy}
Fix any valuation $\nu \in  \mathcal{V}_0 $, and 
any two adjacent $2\times 2$ squares $S,S'$ centered at a vertex of type III.  
Suppose that $v$ is  a vertex in $S \cap S'$ which is $\nu$-minimal in $S$. 

Then the unique vertex $v' \in S' \setminus  S$ which belongs to an edge containing $v$ is also $\nu$-minimal in $S'$.
\end{lem}

One has the following figure:
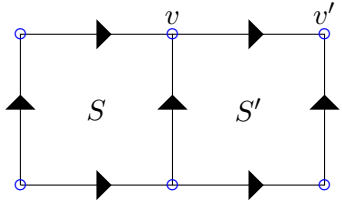
\begin{center}
\begin{tikzpicture} 
\draw (0,0) -- (4,0)--(4,2) -- (2,2) -- (2,0);
\draw (0,0) -- (0,2) -- (2,2);
\fill (1.0,-0.2)--(1.0,0.2) -- (1.2,0.0) --cycle;
\fill (3.0,-0.2)--(3.0,0.2) -- (3.2,0.0) --cycle;
\fill (3.0,2.2)--(3.0,1.8) -- (3.2,2.0) --cycle;
\fill (1.0,2.2)--(1.0,1.8) -- (1.2,2.0) --cycle;

\fill (-0.2,1.0)--(0.2,1.0) -- (0.0,1.2) --cycle;
\fill (2.2,1.0)--(1.8,1.0) -- (2.0,1.2) --cycle;
\fill (4.2,1.0)--(3.8,1.0) -- (4.0,1.2) --cycle;

\draw (1,1) node {$S$};
\draw (2,2) node {} node[above] {$v$};
\draw (3,1) node {$S'$};
\draw (4,2) node {} node[above] {$v'$};

\draw (0,0) node[blue ]{$\circ$};
\draw (2,0) node[blue ]{$\circ$};
\draw (4,0) node[blue ]{$\circ$};
\draw (0,2) node[blue ]{$\circ$};
\draw (2,2) node[blue ]{$\circ$};
\draw (4,2) node[blue ]{$\circ$};
\end{tikzpicture}
\end{center}

\begin{proof}[Proof of Lemma \ref{lem_easy}]
Take $x_1,y_1,z_1,t_1 \in \CSL2$ such that $v=[x_1]$, $[z_1] \in S\cap S'$ and $[y_1],[t_1] \in S$ are the four distinct vertices of $S$. 
We claim that we are in the following situation:
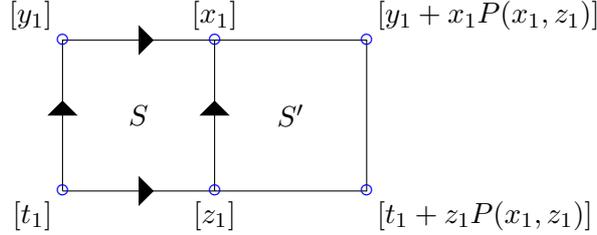
\begin{center}
\begin{tikzpicture} 
\draw (0,0) -- (4,0)--(4,2) -- (2,2) -- (2,0);
\draw (0,0) -- (0,2) -- (2,2);
\fill (1.0,-0.2)--(1.0,0.2) -- (1.2,0.0) --cycle;
\fill (1.0,2.2)--(1.0,1.8) -- (1.2,2.0) --cycle;

\fill (-0.2,1.0)--(0.2,1.0) -- (0.0,1.2) --cycle;
\fill (2.2,1.0)--(1.8,1.0) -- (2.0,1.2) --cycle;

\draw (1,1) node {$S$};
\draw (2,2) node[blue] {$\circ$} node[above] {$[x_1]$};
\draw (2,0) node[blue] {$\circ$} node[below] {$[z_1]$};

\draw (3,1) node {$S'$};
\draw (0,0)node[blue] {$\circ$} node[below left] {$ [t_1]$};
\draw (0,2) node[blue] {$\circ$} node[above left] {$ [y_1]$};

\draw (4,0) node[blue] {$\circ$} node[below right] {$ [t_1 + z_1P(x_1,z_1) ]$};
\draw (4,2) node[blue] {$\circ$} node[above right] {$ [y_1 + x_1P(x_1,z_1) ]$};

\end{tikzpicture}
\end{center}
where $P\in \C[x,y] \setminus \C$.
Indeed, recall that the tame group acts as $g \cdot [f] = [f \circ g^{-1}]$. In particular, if $S_0$ is the standard $2\times 2$ square containing $[x],[y],[z],[t]$ and $[\Id]$ and if $f =(x_1,y_1,z_1,t_1)$, then $S = f^{-1} \cdot S_0$.
Since $S$ and $S'$ are adjacent along an two edges of type I, there exists an element $e \in \EH$ such that $S' =  (f^{-1} \circ e \circ f) \cdot  S$. 
This proves that $S' = (f^{-1} \circ e) \cdot S_0$, and the vertex $v'$ is given by:
\begin{equation*}
v' = [y \circ e^{-1} \circ f],
\end{equation*}
as required.
\smallskip 

Since $\nu(x_1) < \nu(y_1)$ and since $\nu(P(x_1,z_1))< 0$, this implies that:
\begin{equation*}
\nu(y_1 + x_1P(x_1,z_1)) = \nu(x_1P(x_1,z_1)) < \nu(x_1).
\end{equation*}
Similarly, one has:
\begin{equation*}
\nu(t_1 + z_1P(x_1,z_1)) =\nu(z_1P(x_1,z_1)) < \nu(z_1). 
\end{equation*}
Hence since the vertex $[z_1]$ is $\nu$-maximal, we have that $v' = [y_1 + x_1P(x_1,z_1)]$ is the $\nu$-minimal vertex in $S'$ by Lemma \ref{lem_arrow_square}, as required.
\end{proof}

The following proposition compares the distance $d_\nu$ with the distance $d_\mathcal{C}$.

\begin{prop} \label{prop_length_vs_distance} The distance $d_\nu $ and the distance $d_\mathcal{C}$ are equivalent, i.e there exists a constant $C>0$ such that for any vertices $v_1,v_2 \in \mathcal{C}$ of type \textsc{I}, one has:
\begin{equation*}
\dfrac{1}{2\sqrt{2}}d_{\mathcal{C}}(v_1,v_2) \leqslant d_\nu( v_1,v_2) \leqslant 3 d_\mathcal{C} (v_1,v_2). 
\end{equation*}
\end{prop}

\begin{proof} 
For each $2\times 2$ square $S$ centered at a vertex of type III in $\mathcal{C}$, the restriction to $S \cap \mathcal{C}_\nu$ of the distance in $\mathcal{C}_\nu$ and the distance $d_\mathcal{C}$ are bi-lipschitz equivalent. 
More precisely, for any $v_1,v_2 \in S\cap \mathcal{C}_\nu$, the following inequality holds:
\begin{equation*}
\dfrac{d_\mathcal{C} (v_1,v_2)}{2\sqrt{2}} \leqslant d_\nu(v_1,v_2) \leqslant 3 d_\mathcal{C}(v_1,v_2).
\end{equation*}
 Hence, if we apply the previous inequality to a chain of points which belong successively to the same square, we obtain the distance in $\mathcal{C}$ is equivalent to the distance $d_\nu$ and for any vertices $v_1,v_2$ of type I in $\mathcal{C}$, we have:
 \begin{equation*}
 \dfrac{d_\mathcal{C} (v_1,v_2)}{2\sqrt{2}} \leqslant d_\nu(v_1,v_2) \leqslant 3 d_\mathcal{C}(v_1,v_2),
\end{equation*}  
as required.
%
\black  
\end{proof}
%
%
\bigskip

\subsection{Avoiding critical resonances}
\label{subsection_R}

\bigskip

Fix a valuation $\nu \in \mathcal{V}_0$ and fix a $2\times 2$ square  $S$.
Consider a vertex $[x_1]$ of type I in $S$ which is $\nu$-minimal in $S$ where $x_1 \in \CSL2$ and denote by $[z_1]$ another vertex of type I in $S$ such that $[x_1]$ and $[z_1]$ belong to a vertical edge of the square $S$. 
For any square $S'$ which is adjacent to $S$ along the edge containing $[x_1]$ and $[z_1]$, Lemma \ref{lem_easy} implies that the function induced by $\nu$ on the vertices is as follows,
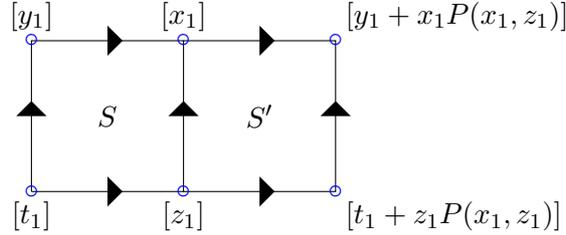
\begin{center}
\begin{tikzpicture} 
\draw (0,0) -- (4,0)--(4,2) -- (2,2) -- (2,0);
\draw (0,0) -- (0,2) -- (2,2);
\fill (1.0,-0.2)--(1.0,0.2) -- (1.2,0.0) --cycle;
\fill (3.0,-0.2)--(3.0,0.2) -- (3.2,0.0) --cycle;
\fill (3.0,2.2)--(3.0,1.8) -- (3.2,2.0) --cycle;
\fill (1.0,2.2)--(1.0,1.8) -- (1.2,2.0) --cycle;

\fill (-0.2,1.0)--(0.2,1.0) -- (0.0,1.2) --cycle;
\fill (2.2,1.0)--(1.8,1.0) -- (2.0,1.2) --cycle;
\fill (4.2,1.0)--(3.8,1.0) -- (4.0,1.2) --cycle;

\draw (1,1) node {$S$};
\draw (2,2) node[blue] {$\circ$} node[above] {$[x_1]$};
\draw (0,2) node[blue] {$\circ$} node[above] {$[y_1]$};

\draw (0,0) node[blue] {$\circ$} node[below] {$[t_1]$};

\draw (2,0) node[blue] {$\circ$} node[below] {$[z_1]$};

\draw (4,2) node[blue] {$\circ$} node[above right] {$[y_1 + x_1P(x_1,z_1)]$};

\draw (4,0) node[blue] {$\circ$} node[ below right] {$[t_1 + z_1P(x_1,z_1)]$};
\draw (3,1) node {$S'$};

\end{tikzpicture}
\end{center}
where $y_1, t_1, \in \CSL2$ and  $P \in \C[x,y] \setminus \C$.
Observe that if the component $(x_1,z_1)$ is not critically resonant with respect to $\nu$, then by Corollary \ref{cor_precise_parachute}, one has:
\begin{equation*}
\max ( \nu(y_1 + x_1P(x_1,z_1) , \nu(t_1 + z_1P(x_1,z_1)) ) < \dfrac{4}{3} \nu(z_1).
\end{equation*}
Moreover, Corollary \ref{cor_parachute} implies also:
\begin{equation*}
\max ( \nu(y_1 + x_1P(x_1,z_1) , \nu(t_1 + z_1P(x_1,z_1)) ) < \nu(x_1)
\end{equation*}

When the component $(x_1,z_1)$ is critically resonant, then the previous inequality does not necessarily hold since we cannot apply Corollary \ref{cor_precise_parachute}. 

Our key observation is that the previous inequality remains valid whenever there exists a square $S_1$ adjacent to $S$ along the edge  containing $[t_1], [z_1]$  and such that its other edge containing $[z_1]$ is not critically resonant. 
If we choose $S_1$ so that the squares $S_1,S,S'$ are flat, we arrive at the following situation where a \textbf{blue edge} means that the corresponding component is not critically resonant and a \textbf{red edge} that the component is critically resonant:
%
%
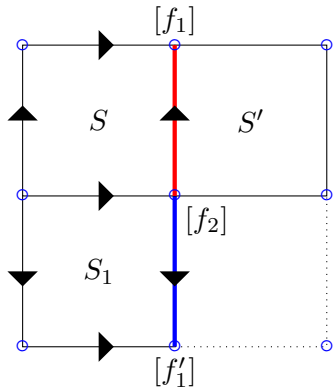
\begin{center}
\begin{tikzpicture}
\draw (0,2) -- (0,0) --(2,0) ;
\draw[ultra thick,blue] (2,0) --(2,2);
\draw[dotted] (2,0)--(4,0) --(4,2);
\draw (2,4) --(4,4) --(4,2)--(2,2);
\draw (0,2) -- (0,4) --(2,4) -- (2,2)-- (0,2);
\draw[ultra thick, red] (2,2)--(2,4) ;
\draw (2,4) node[blue] {$\circ$} node[above] {$[f_1]$};
\draw (2,2) node[blue] {$\circ$} node[below right] {$[f_2]$};
\draw (0,4) node[left] {};
\draw (1,1) node {$S_1$};
\draw (1,3) node {$S$};
\draw (3,3) node {$S'$};
\draw (0,2) node[left] {};
\draw (2,0) node[blue] {$\circ$} node[below] {$[f_1']$};
\fill (1.0,-0.2)--(1.0,0.2) -- (1.2,0.0) --cycle;
\fill (1.0,2.2)--(1.0,1.8) -- (1.2,2.0) --cycle;
\fill (2.2,1.0)--(1.8,1.0) -- (2.0,0.8) --cycle;
\fill (-0.2,1.0)--(0.2,1.0) -- (0.0,0.8) --cycle;
\fill (1.0,4.2)--(1.0,3.8) -- (1.2,4.0) --cycle;
\fill (2.2,3.0)--(1.8,3.0) -- (2.0,3.2) --cycle;
\fill (-0.2,3.0)--(0.2,3.0) -- (0.0,3.2) --cycle;

\draw (0,0) node[blue] {$\circ$};
\draw (0,2) node[blue] {$\circ$};
\draw (0,4) node[blue] {$\circ$};
\draw (4,0) node[blue] {$\circ$};
\draw (4,2) node[blue] {$\circ$};
\draw (4,4) node[blue] {$\circ$};

\end{tikzpicture}
\end{center}
In the rest of section, we keep the same convention on the colors of the edges.

We now illustrate our argument in the following lemma.

\begin{lem} \label{lem_apply_R}
Fix $\nu \in \mathcal{V}_0$ and $S,S'$ two adjacent  $2\times 2$ squares.
Consider $v_1, v_2$ two vertices of the common edge of these squares and suppose that $v_1$ is $\nu$-minimal in $S$.
Suppose that the edge joining $v_1$ and $v_2$ corresponds to a component $(f_1,f_2)$ which is not critically resonant. Then for any vertex $v' \in S'$ distinct from $v_1,v_2$, we have:
\begin{equation*}
\nu(v') < \min\left (\dfrac{4}{3} \nu(v_2), \nu(v_1)\right ).
\end{equation*}
\end{lem}

\begin{proof}
 Observe that this lemma  follows immediately from Corollary \ref{cor_precise_parachute} and Corollary \ref{cor_parachute}. 
\end{proof}
\black 

The Proposition below is the key ingredient in our proof and explains how one can find a square which has an edge which is not critically resonant. 

\begin{prop} \label{prop_prop_R_practice}
Fix a valuation $\nu \in \mathcal{V}_0$. 
Let $S$ be any $2\times 2$ square having a unique $\nu$-minimal vertex, and let $[f_1],[f_2]$ be any horizontal (resp. vertical) edge of $S$.
Suppose that $\nu(f_1) < \nu(f_2)$, that $(f_1,f_2)$ is critically resonant and that for any polynomial $R \in \C[x] \setminus \C$, one has:
\begin{equation*}
\nu(f_1 - f_2R(f_2)) < \nu(f_2).
\end{equation*}
Then there exists a square $S_1$ adjacent to $S$ along the vertical (resp. horizontal) edge containing $[f_2]$ which satisfies the following properties.
\begin{enumerate}
\item[(i)] For any square $S_2$ adjacent to $S$ along the edge containing $[f_1],[f_2]$, the squares $S_1,S,S_2$ are flat.
\item[(ii)] The horizontal (resp. vertical) edge in $S_1$ containing $[f_2]$ is not critically resonant.
\item[(iii)] There exists an element $g \in A_{[f_2]} $ such that $g \cdot  S = S_1$.
\end{enumerate}
\end{prop}


\begin{proof} 
Statement (i) and $(iii)$ follow from Lemma \ref{lem_technical_flat}.$(ii)$ and Lemma \ref{lem_technical_flat}.$(i)$ respectively. Indeed pick
any polynomial $R \in\C[x] \setminus \C$, and let $S_R$ be the square containing $[f_2], [f_1- f_2R(f_2)]$ which is adjacent to $S$ along the vertical edge containing $[f_2]$.
Since $R$ depends on a single variable, it follows that for any square $S_2$ adjacent to $S$ along the edge containing $[f_1],[f_2]$, the squares $S_R,S_2,S$ are flat.

We now prove $(ii)$, and produce a polynomial $R \in \C[x] \setminus \C$ such that the component $(f_2,f_1-f_2R(f_2))$ is not critically resonant. 
Since the component $(f_1,f_2)$ is critically resonant, there exists a constant $\lambda \in \C^*$ and an integer $n \geqslant 1$ such that 
$$\nu(f_1 - \lambda f_2^{n})) > \nu(f_1) = n \nu(f_2).$$
Since $\nu(f_1) < \nu(f_2)$, we get $n \geqslant 2$ so that  
$R_1 := \lambda x^{n-1} \in \C[x] \setminus \C$. 

If the component $(f_2,f_1-f_2R(f_2))$ is not critically resonant, then the square $S_1$ containing $[f_2], [f_1 - f_2R(f_2)]$ which is adjacent to $S$ along the vertical edge containing $[f_2]$ satisfies assertion $(ii)$ and we are done. 
Otherwise, $(f_2,f_1-f_2R(f_2))$ is critically resonant. 
Observe that
by assumption, we have
\begin{equation*}
\nu(f_1 - f_2 R_1(f_2)) < \nu(f_2)~,
\end{equation*}
so that $\nu(f_1 - f_2 R_1(f_2)) = n_2 \nu(f_2)$ for some $n_2 \geqslant 1$, and 
$\nu(f_1 - f_2 R_2(f_2)) > n_2 \nu(f_2)$ for some polynomial $R_2\in k[x]\setminus k$ of the form 
$R_2(x) = R_1(x) + \lambda' x^{n_2-1}$.
Repeating this argument we get a sequence of polynomials $R_i \in \C[x] \setminus \C$, and either
$(f_2,f_1-f_2R_i(f_2))$ is not critically resonant for some index $i$; or 
 $(f_2,f_1-f_2R_i(f_2))$ is critically resonant for all $i$. However in the latter case, 
the sequence $(\nu(f_1 - f_2R_i(f_2))$ is strictly increasing and  $(\nu(f_1 - f_2R_i(f_2))$
are all multiples of $\nu(f_2)$ which yields a contradiction. 
The proof is complete. 

%
%
%
%
%
\end{proof}

\subsection{Degree estimates between adjacent squares}
\label{degree_estimates_adjacent}

\begin{thm}   \label{thm_vert_band}
Take a valuation $\nu \in \mathcal{V}_0$. Let $v$ be any $\nu$-maximal vertex of a 
$2\times 2$ square $S$ and let $S'$ be an adjacent square which does not contain $v$. Suppose that the vertex $v$ is also $\nu$-maximal in any square $\tilde{S}$ adjacent to $S$. 
 Then $S'$ admits a unique $\nu$-minimal vertex and for any vertex $v' \in S' \setminus S$, one has:
\begin{equation*}
\nu(v') < \dfrac{4}{3} \nu(v_2),
\end{equation*}
 where $v_2$ is the vertex in $S \cap S'$ which is not $\nu$-minimal in $S$.
\end{thm}

\begin{proof}
Observe that Lemma \ref{lem_easy} implies that $S'$ has a unique $\nu$-minimal vertex.
If the edge $S \cap S'$ is not critically resonant, then Lemma \ref{lem_apply_R} implies the conclusion of the theorem. 
Otherwise, the edge $S\cap S'$ is critically resonant and we check that  the squares $S$ and $S'$ satisfy the hypothesis of Proposition \ref{prop_prop_R_practice}. 
\smallskip

Denote by $[f_1]$ the $\nu$-minimal vertex in $S$ and by $[f_2]=v_2$. 
For any polynomial $R \in \C[x] \setminus \C$, take $S_R$ to be the square  containing $[f_1 - f_2R(f_2)], [f_2]$ and $v$. By construction, the square $S_R$ is adjacent to $S$ along an edge containing $v$, hence the vertex $v$ is $\nu$-maximal in $S_R$ and this implies that the vertex $[f_1 - f_2R(f_2)]$ is $\nu$-minimal in $S_R$. 
In particular, we have proved that $\nu(f_1- f_2R(f_2)) < \nu(f_2)$. 
By Proposition \ref{prop_prop_R_practice}, there exists two squares $S_1',S_2'$ such that the union $S \cap S_1' \cup S' \cup S_2' $ forms a $4\times 4$ square centered along the vertex $[v_2]$ and such that the edge $S_1'\cap S_2'$ is not critically resonant. 
We thus arrive at the following situation (with the same convention on colors as in the previous section).
  
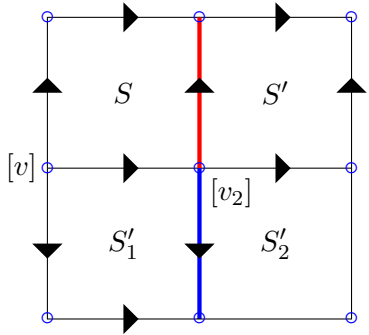
\begin{center}
\begin{tikzpicture} 
\draw (0,0) -- (4,0)--(4,2) -- (2,2) -- (2,0);
\draw (0,0) -- (0,2) -- (2,2);
\draw [red , ultra thick] (2,0) --(2,2);
\draw [blue , ultra thick] (2,0) --(2,-2);
\draw (0,0)--(0,-2) --(4,-2)--(4,0);
\fill (1.0,-0.2)--(1.0,0.2) -- (1.2,0.0) --cycle;
\fill (3.0,-0.2)--(3.0,0.2) -- (3.2,0.0) --cycle;
\fill (3.0,2.2)--(3.0,1.8) -- (3.2,2.0) --cycle;
\fill (1.0,2.2)--(1.0,1.8) -- (1.2,2.0) --cycle;

\fill (-0.2,1.0)--(0.2,1.0) -- (0.0,1.2) --cycle;
\fill (2.2,1.0)--(1.8,1.0) -- (2.0,1.2) --cycle;
\fill (4.2,1.0)--(3.8,1.0) -- (4.0,1.2) --cycle;

\draw (1,1) node {$S$};


\draw (3,1) node {$S'$};

\draw (2,0) node {} node[below right] {$[v_2]$};

\draw (0,0) node {} node[ left] {$[v]$};
\fill (-0.2,-1.0)--(0.2,-1.0) -- (0.0,-1.2) --cycle;
\fill (2.2,-1.0)--(1.8,-1.0) -- (2.0,-1.2) --cycle;
\fill (1.0,-2.2)--(1.0,-1.8) -- (1.2,-2.0) --cycle;
\draw (1,-1) node {$S_1'$};
\draw (3,-1) node {$S_2'$};
\dtypeI{(0,0)}
\dtypeI{(2,0)}
\dtypeI{(4,0)}
\dtypeI{(0,-2)}
\dtypeI{(2,-2)}
\dtypeI{(4,-2)}
\dtypeI{(0,2)}
\dtypeI{(2,2)}
\dtypeI{(4,2)}

\end{tikzpicture}
\end{center}
In particular, by applying Lemma \ref{lem_easy} to the square $S_1'$ and $S_2'$, we find  that 
\begin{equation*}
\nu(v') < \dfrac{4}{3} \nu(v_2),
\end{equation*}        
for any vertex $v' \in S' \setminus S$ as required.
\end{proof}

\subsection{Degree estimates at a $\nu$-maximal vertex}
\label{subsection_bound_stab_rev}

In this section, we analyze the situation of two $2\times 2$ squares adherent at a vertex of type $I$. 

Recall from Section \ref{subsection_squares_def} that a pair of adherent squares $(S, S')$ is contained in a  spiral staircase if there exists a sequence of squares $S_0 = S, \ldots , S_p = S'$ connecting $S$ and $S'$ which are  adjacent alternatively along vertical and horizontal edges and such that any three consecutive squares $S_i,S_{i+1}, S_{i+2}$ are not flat for $i \leqslant p-2$. When the intersection between $S_0$ and $S_1$ is a horizontal (resp. vertical) edge, we say that the staircase is vertical (resp. horizontal).  

%

\begin{thm} \label{thm_technical_degree_stab_I}
Fix a valuation $\nu \in \mathcal{V}_0$.

Consider three $2\times 2$ squares $S,S_1$ and $S'$ having a vertex $[x_1]$ of type \textsc{I}  in common.  We assume that  $S$ and $S_1$ 
have a common horizontal edge $[x_1], [y_1]$, and that the pair $(S,S')$ is contained in a vertical spiral staircase containing $S_1$.
Denote by $[z_1]$ the vertex in $S_1$ which forms a vertical edge with $[x_1]$. 

Assume that $[x_1]$ is $\nu$-maximal in $S_1$, that the component $(x_1,z_1)$ is not critically resonant, that $\nu(z_1) < \nu(y_1)$ and $\nu(z_1) < (4/3) \nu(x_1)$.
Then for any vertex $v \in S'$ distinct from $[x_1]$, one has:
\begin{equation*}
\nu(v) < \dfrac{4}{3} \nu(x_1).
\end{equation*}

\end{thm}

The following figure summarizes the situation of the Theorem (with the convention of section \ref{subsection_R} on the color of the edges).
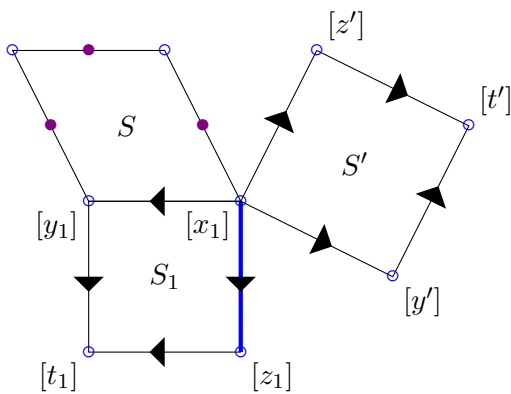
\begin{center}
\begin{tikzpicture} 

\draw    (2, 0) --(0,0)-- (0,2)-- (2,2) ;

\draw[ultra thick, blue] (2,0)--(2,2); 
\fill (1.0,2.2)--(1.0,1.8) -- (0.8,2.0) --cycle;
\fill (1.0,-0.2)--(1.0,0.2) -- (0.8,0.0) --cycle;

\fill (-0.2,1.0)--(0.2,1.0) -- (0.0,0.8) --cycle;
\fill (2.2,1.0)--(1.8,1.0) -- (2.0,0.8) --cycle;

\draw (0,2) --(-1,4) --(1,4) --(2,2)--cycle;
\draw (-1,4) node[blue] {$\circ$};
\draw (1,4) node[blue] {$\circ$};

\dtypeII{(-0.5,3)}
\dtypeII{(1.5,3)}
\dtypeII{(0,4)}

\draw (0.5,3) node {$S$};
\draw  (3, 4)-- (5,3) -- (4,1) --(2,2);
\draw  (2,2) -- (3,4); 
\draw (1,1) node {$ S_1 $};
\draw (3.5,2.5) node {$S'$};

\draw (0, 0) node[blue] {$\circ$} node[below left ] {$ [t_1]$};
\draw (2, 0) node[blue] {$\circ$} node[below right ] {$ [z_1]$};
\draw (2, 2) node[blue] {$\circ$} node[below left ] {$ [x_1]$};
\draw (0, 2) node[blue] {$\circ$} node[below left ] {$ [y_1]$};
\draw (3, 4) node[blue] {$\circ$} node[above  right] {$ [z']$};
\draw (5, 3) node[blue] {$\circ$} node[above  right] {$ [t']$};
\draw (4, 1) node[blue] {$\circ$} node[below  right] {$ [y']$};

\fill (2.3,3.1)--(2.7,2.9) -- (2.6,3.2) --cycle;
\fill (2.9,1.3)--(3.1,1.7) -- (3.2,1.4) --cycle;
\fill (4.1,3.7)--(3.9,3.3) -- (4.2,3.4) --cycle;
\fill (4.7,1.9)--(4.3,2.1) -- (4.6,2.2) --cycle;

\end{tikzpicture}
\end{center}


%

We shall use repeatedly the following lemma, whose proof
is given at the end of this section. 

Recall from Section \ref{subsection_link_type_I} 
the definition of the subgroup $A_v$ of the stabilizer  of a vertex $v$ of type I. 
\begin{lem}\label{lem_flat_replace} Take three $2 \times 2$ squares $S_1,S_2,S_3$ containing $[x_1]$ and which are  adjacent alternatively along vertical and horizontal edges. Suppose that $S_1,S_2$ and $S_3$ are not flat. Then the following assertions hold.
\begin{enumerate}
\item[(i)] Suppose that $S'_1$ is a  $2\times 2$ square which is adjacent to $S_2$ along $S_1\cap S_2$ such that there exists an element $g \in A_{[x_1]}$ for which $g \cdot S_1 = S_1'$. Then the squares $S_1',S_2,S_3$ are not flat.
\item[(ii)] For any $2\times 2$ squares $S_1',S_2'$ such that $S_1,S_2,S_1',S_2'$ are flat, 
the squares $S_1',S_2',S_3$ are not flat. Moreover, given any $g_1, g_2 \in \stab([x_1]) \cap \stame$ such that $g_1 S_1 = S_1'$, and  $g_2 S_2 = S_2'$, we have $g_1, g_2 \in A_{[x_1]}$.
\end{enumerate}
\end{lem}

This lemma will allow us to consider alternative spiral staircase around the vertex $[x_1]$. We thus have the following figures in each situation.
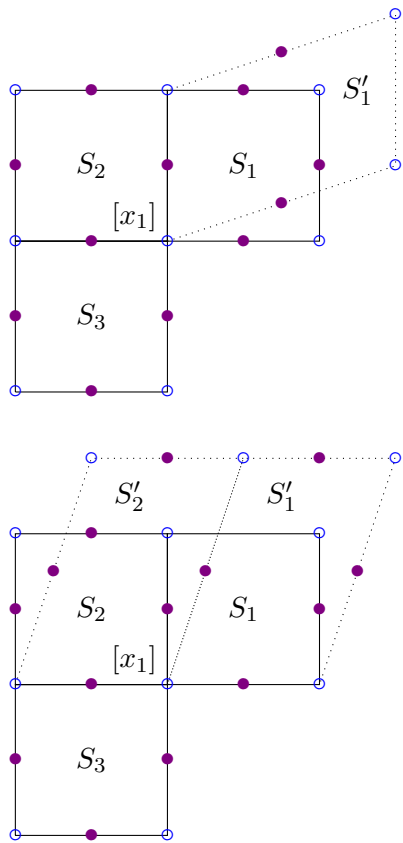
\begin{center}
\begin{tikzpicture} 
\draw  (0,0)-- (0,2) -- (2,2) -- (2, 0) --(0,0); 
\draw (0,2)--(2,2);
\draw[dotted] (2,2) --(5,3)--(5,5)--(2,4);
\draw (4.5,4) node {${S}_1'$};
\draw  (0,2)-- (0,4) -- (2,4) -- (2, 2) --(0,2); 
\draw  (2,2)-- (2,4) -- (4,4) -- (4, 2) --(2,2); 
\draw (2,2)--(2,4);
\draw (1,1) node {$S_{3}$};
\draw (3,3) node {$S_1$};
\draw (1,3) node {$S_2$};
\draw (2, 4) node[blue] {$\circ$};
\draw (0, 0) node[blue] {$\circ$};
\draw (2, 0) node[blue] {$\circ$};
\draw (2, 2) node[blue] {$\circ$} node[above left ] {$ [x_1]$};
\draw (0, 2) node[blue] {$\circ$};
\draw (4, 4) node[blue] {$\circ$};
\draw (0, 4) node[blue] {$\circ$};
\draw (4, 2) node[blue] {$\circ$};
\draw (5, 3) node[blue] {$\circ$};
\draw (5, 5) node[blue] {$\circ$};

\dtypeII{(1,0)}
\dtypeII{(1,2)}
\dtypeII{(1,4)}
\dtypeII{(0,1)}
\dtypeII{(0,3)}
\dtypeII{(2,1)}
\dtypeII{(2,3)}
\dtypeII{(4,3)}

\dtypeII{(3,2)}

\dtypeII{(3,4)}
\dtypeII{(3.5,2.5)}

\dtypeII{(3.5,4.5)}

\end{tikzpicture}
\end{center}

\begin{center}
\begin{tikzpicture} 
\draw  (0,0)-- (0,2) -- (2,2) -- (2, 0) --(0,0); 
\draw (0,2)--(2,2);
\draw (1.5,4.5) node {${S}_2'$};
\draw (3.5,4.5) node {${S}_1'$};
\draw  (0,2)-- (0,4) -- (2,4) -- (2, 2) --(0,2); 
\draw  (2,2)-- (2,4) -- (4,4) -- (4, 2) --(2,2); 
\draw[dotted] (0,2) -- (1,5) --(3, 5) --(2,2);
\draw[dotted] (2,2) -- (3,5) --(5, 5) --(4,2);
\draw (2,2)--(2,4);
\draw (1,1) node {$S_{3}$};
\draw (3,3) node {$S_1$};
\draw (1,3) node {$S_2$};
\draw (2, 2) node[above left ] {$ [x_1]$};

\dtypeII{(1,0)}
\dtypeII{(1,2)}
\dtypeII{(1,4)}
\dtypeII{(0,1)}
\dtypeII{(0,3)}
\dtypeII{(2,1)}
\dtypeII{(2,3)}
\dtypeII{(4,3)}

\dtypeI{(0,0)}
\dtypeI{(0,2)}
\dtypeI{(0,4)}
\dtypeI{(2,0)}
\dtypeI{(2,2)}
\dtypeI{(2,4)}
\dtypeI{(4,2)}
\dtypeI{(4,4)}
\dtypeI{(1,5)}
\dtypeI{(3,5)}
\dtypeI{(5,5)}

\dtypeII{(0.5,3.5)}
\dtypeII{(2.5,3.5)}
\dtypeII{(4.5,3.5)}
\dtypeII{(2,5)}
\dtypeII{(4,5)}
\dtypeII{(3,2)}
\end{tikzpicture}
\end{center}

\begin{proof}[Proof of Theorem \ref{thm_technical_degree_stab_I}]

Take a valuation $\nu \in \mathcal{V}_0$ and three squares $S,S_1,S'$ satisfying the conditions of the theorem. 
By assumption, there exists an integer $p \geqslant 2$ and a sequence of adjacent squares 
$S_2, \ldots, S_{p-1}$ such that  $S_0 = S, S_1,S_2, \ldots , S_p  = S'$ forms a vertical staircase.

We denote by $[y_1],[z_1],[t_1],[x_1]$ and $[z'],[y'],[t']$ the vertices of $S_1$ and $S'$ respectively so that the edges $[x_1],[y_1]$ and $[x_1],[y']$ are horizontal and the edges $[x_1],[z_1]$ and $[x_1],[z']$ are vertical. We are thus in the following situation.
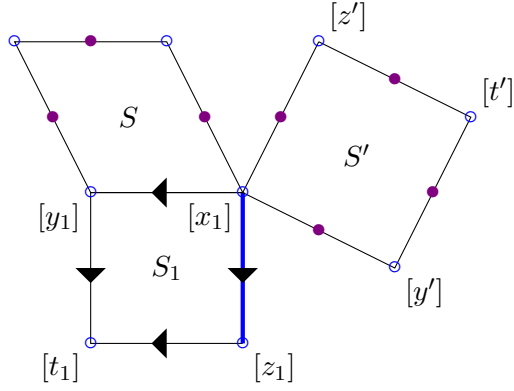
\begin{center}
\begin{tikzpicture} 

\draw    (2, 0) --(0,0)-- (0,2)-- (2,2) ;

\draw[ultra thick, blue] (2,0)--(2,2); 
\fill (1.0,2.2)--(1.0,1.8) -- (0.8,2.0) --cycle;
\fill (1.0,-0.2)--(1.0,0.2) -- (0.8,0.0) --cycle;

\fill (-0.2,1.0)--(0.2,1.0) -- (0.0,0.8) --cycle;
\fill (2.2,1.0)--(1.8,1.0) -- (2.0,0.8) --cycle;

\draw (0,2) --(-1,4) --(1,4) --(2,2)--cycle;
\draw (0.5,3) node {$S$};
\draw  (3, 4)-- (5,3) -- (4,1) --(2,2);
\draw  (2,2) -- (3,4); 
\draw (1,1) node {$ S_1 $};
\draw (3.5,2.5) node {$S'$};
\draw (0, 0) node[blue] {$\circ$} node[below left ] {$ [t_1]$};
\draw (2, 0) node[blue] {$\circ$} node[below right ] {$ [z_1]$};
\draw (2, 2) node[blue] {$\circ$} node[below left ] {$ [x_1]$};
\draw (0, 2) node[blue] {$\circ$} node[below left ] {$ [y_1]$};
\draw (3, 4) node[blue] {$\circ$} node[above  right] {$ [z']$};
\draw (5, 3)node[blue] {$\circ$} node[above  right] {$ [t']$};
\draw (4, 1) node[blue] {$\circ$} node[below  right] {$ [y']$};

\dtypeI{(-1,4)}
\dtypeI{(1,4)}
\dtypeII{(0,4)}
\dtypeII{(-0.5,3)}
\dtypeII{(1.5,3)}

\dtypeII{(2.5,3)}
\dtypeII{(4.5,2)}
\dtypeII{(3,1.5)}
\dtypeII{(4,3.5)}
\end{tikzpicture}
\end{center}

 Recall that $S$ and $S'$ are connected by a vertical staircase $S = S_0, S_1, \ldots, S_{p-1} , S_p = S'$.
\begin{lem}\label{lem:step1}
The theorem holds whenever the edges $S_i \cap S_{i+1}$ are not critically resonant for all $i\ge 1$.  
\end{lem}
\begin{lem}\label{lem:step2}
For any vertex $v$ such that $[x_1], v$ is an edge of $S'$, 
there exists a vertical staircase  $S = S_0, \tilde{S}_1, \tilde{S}_2, \ldots, \tilde{S}_{q-1} , \tilde{S}_q$
such that
\begin{itemize}
\item
$\tilde{S}_1 = S_1$; 
\item 
$\tilde{S}_q$ and $S'$ are adjacent along  the edge $[x_1], v$; 
\item
 the edges $\tilde{S}_i \cap \tilde{S}_{i+1}$ are not critically resonant for all $i\ge 1$.  
 \end{itemize}
\end{lem}
Take  any vertex $v$ of $S'$ such that $[x_1], v$ is an edge of $S'$. By Lemma~\ref{lem:step2}
we get a sequence of squares $\tilde{S}_i$ connecting $S$ to $\tilde{S}_q$ and satisfying the assumptions of Lemma~\ref{lem:step1}.
This proves $\nu(v) < \dfrac{4}{3} \nu(x_1)$ as required. 
\end{proof}

%

\begin{proof}[Proof of Lemma~\ref{lem:step1}]
We prove  by induction on $i$ the following two properties:
\begin{enumerate}
\item[$ ( \mathcal{P}_1) $] For any vertex $v \neq [x_1]$ in $S_i\setminus S_{0}$, one has:
\begin{equation*}
\nu(v) < \dfrac{4}{3} \nu(x_1).
\end{equation*} 
\item[$(\mathcal{P}_2)$] Let $v_1 \neq [x_1]$ be the unique vertex which is contained in the edge $S_i \cap S_{i-1}$ and let $v_2$ be the other vertex in $S_i$ which belongs to an edge containing $[x_1]$. Then one has:
\begin{equation*}
\nu(v_2) < \nu(v_1).
\end{equation*}
\end{enumerate}  

Observe that $(\mathcal{P}_1)$ and $(\mathcal{P}_2)$ are satisfied when $i=1$ by our standing assumption on $S_1$. 

Let us prove the induction step. For all $i$, denote by $t_i$ the unique vertex of $S_i$ which does not lie in $S_{i-1} \cup S_{i+1}$; by $y_i$ the 
vertex in $S_i \cap S_{i-1}$ distinct from $x_1$. We also write $z_i$ for the vertex in $S_i \cap S_{i+1}$ distinct from $x_1$ (so that $y_{i+1} = z_i$). 
We thus have  the following picture:
\bigskip

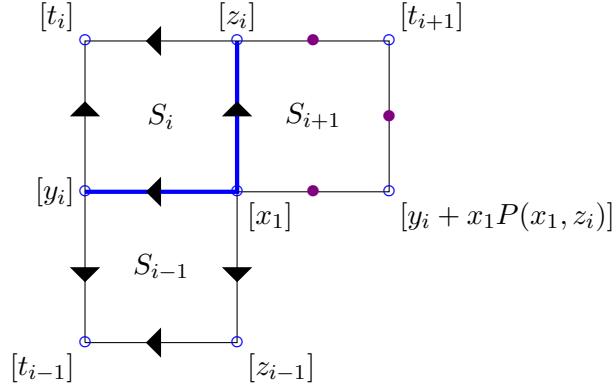
\begin{center}
\begin{tikzpicture}
\draw  (0,-2)-- (0,0) -- (2,0) -- (2, -2) --(0,-2);
\draw    (2, 0) --(0,0);
\draw (0,0)--(0,2)--(2,2);

\draw[ultra thick, blue] (0,0)--(2,0)--(2,2);
\draw   (2,2) -- (4,2) -- (4, 0) --(2,0);

\draw (1,1) node[] {$S_i$};
\draw (1,-1) node[] {$S_{i-1}$};

\fill (1.0,-0.2)--(1.0,0.2) -- (0.8,0.0) --cycle;
\fill (1.0,2.2)--(1.0,1.8) -- (0.8,2.0) --cycle;
\fill (2.2,1.0)--(1.8,1.0) -- (2.0,1.2) --cycle;
\fill (-0.2,1.0)--(0.2,1.0) -- (0.0,1.2) --cycle;

\fill (-0.2,-1.0)--(0.2,-1.0) -- (0.0,-1.2) --cycle;
\fill (2.2,-1.0)--(1.8,-1.0) -- (2.0,-1.2) --cycle;
\fill (1.0,-2.2)--(1.0,-1.8) -- (0.8,-2.0) --cycle;

\draw (0,-2) node [below left] {$[t_{i-1}]$};
\draw (2,-2) node [below right] {$[z_{i-1}]$};
\draw (3,1) node[] {$S_{i+1}$};
\draw (2,0) node[below right] {$[x_1]$};
\draw (4,0) node[below right] {$[y_i + x_1 P(x_1,z_i)]$};
\draw (0,0) node[ left] {$ [y_i]$};
\draw (0,2) node [above left] {$[t_i]$};
\draw (2,2) node[above] {$[z_i]$};
\draw (4,2) node[above right] {$[t_{i+1}]$};

\dtypeI{(0,0)}
\dtypeI{(0,2)}
\dtypeI{(0,-2)}
\dtypeI{(2,0)}
\dtypeI{(2,2)}
\dtypeI{(2,-2)}
\dtypeI{(4,2)}
\dtypeI{(4,0)}
\dtypeII{(3,2)}
\dtypeII{(3,0)}
\dtypeII{(4,1)}

\end{tikzpicture}
\end{center}
%
%
%
%
 By our induction hypothesis, we have:
\begin{equation*}
\nu(z_i) < \nu(y_i) < \nu(x_1)~.
\end{equation*}
Observe that 
 $y_{i+1}$ is given by:
 \begin{equation*}
 y_{i+1} = y_i + x_1 P(x_1,z_i). 
 \end{equation*}
for some polynomial $P \in \C[x,y]$. Since the squares $(S_{i-1},S_i,S_{i+1})$ is not flat, 
Lemma \ref{lem_technical_flat}.$(i)$ and Lemma \ref{lem_complete_square_1}   imply that 
 that $P \notin \C[x]$.

 Since the component $(x_1,z_i)$ is not critically resonant, Corollary \ref{cor_parachute} and Corollary \ref{cor_precise_parachute} applied to $f_1 = z_i$ and $f_2 = x_1$ imply:
\begin{equation*}
\nu(x_1P(x_1,z_i)) < \min \left ( \dfrac{4}{3} \nu(x_1) , \nu(z_i) \right )~,
\end{equation*}
hence:
 \begin{equation*}
 \nu(y_{i+1}) = \nu(y_i + x_1P(x_1,z_i)) = \nu(x_1P(x_1,z_i)) < \min \left ( \dfrac{4}{3} \nu(x_1) , \nu(z_i) \right ). 
 \end{equation*}
This proves that $[x_1]$ is $\nu$-maximal in $S_{i+1}$, hence $[t_{i+1}]$ is $\nu$-minimal in $S_{i+1}$ by Lemma \ref{lem_arrow_square} and assertion $(\mathcal{P}_1)$ and $(\mathcal{P}_2)$ hold for $i+1$, as required.
\end{proof}

\begin{proof}[Proof of Lemma~\ref{lem:step2}]
We prove by induction on the length of the vertical staircase, i.e. on $p$ the following stronger version of the lemma. 
For any vertex $v$ such that $[x_1], v$ is an edge of $S'$, 
there exists a vertical staircase  $S = S_0, \tilde{S}_1, \tilde{S}_2, \ldots, \tilde{S}_{q-1} , \tilde{S}_q$
such that
\begin{itemize}
\item
$\tilde{S}_1 = S_1$; 
\item
$\tilde{S}_q$ and $S'$ are adjacent along  the edge $[x_1], v$, and there exists an element $g \in A_{[x_1]}$ for which $g \cdot \tilde{S}_q = S'$.
\item
 the edges $\tilde{S}_i \cap \tilde{S}_{i+1}$ are not critically resonant for all $i\ge 1$.  
 \end{itemize}
 For $p=2$, we may choose $\tilde{S}_1 = S_1$, $\tilde{S}_2 = S'$, and there is nothing to prove since $[x_1], [z_1]$ is not critically resonant by our
standing assumption.  
 
%
%
%
 
 Let us prove the induction step. Suppose that the claim is true for any staircase of length $p$, and pick a staircase $(S=S_0,S_1, \ldots , S_{p+1}=S')$ joining $S$ to $S'$.
 By the induction step applied to the vertex $v_p \in S_p \cap S_{p+1}$ distinct from $x_1$, we may find another vertical spiral staircase
 $(S=S_0,S_1 = \tilde{S}_1, \tilde{S}_2, \ldots , \tilde{S}_p)$ such that the edges $\tilde{S}_i \cap \tilde{S}_{i+1}$ are not critically resonant for all $1 \le i \le p$, and 
  there exists an element $g \in A_{[x_1]}$ for which $g \cdot \tilde{S}_{p} = S_{p}$.

  Observe that the $\tilde{S}_p$ and $S_{p+1} $ are adjacent along the edge containing $[x_1], v_p$. 
  Since $S_{p-1},S_p, S_{p+1}$ are not flat, Lemma \ref{lem_flat_replace}.$(ii)$ implies that the squares $\tilde{S}_{p-1},\tilde{S}_p,S_{p+1}$ are also not flat. 
  
  If the edge $S_{p+1} \cap \tilde{S}_p$ is not critically resonant, the proof is complete.
 Otherwise, the edge $\tilde{S}_p \cap S_{p+1}$ is critically resonant. Denote by $[z_p]$ and $v'$ the vertices
 in $\tilde{S}_q$ distinct from $x_1$ and lying in $S'$ and $\tilde{S}_{p-1}$ respectively.
  By Lemma~\ref{lem:step1}, we have $\nu(z_p) < \nu(v') < \nu(x_1)$, and we have the following picture. 
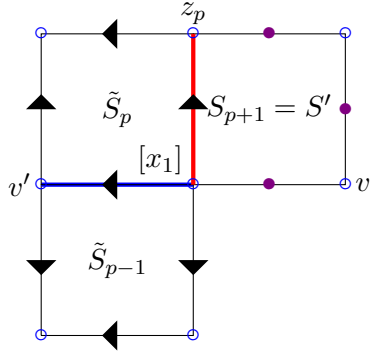
\begin{center}
\begin{tikzpicture} 
\draw  (0,0)-- (0,2) -- (2,2) -- (2, 0) --(0,0); 
\draw[ultra thick, blue] (0,2)--(2,2);

\draw  (0,2)-- (0,4) -- (2,4) -- (2, 2) --(0,2); 
\draw  (2,2)-- (2,4) -- (4,4) -- (4, 2) --(2,2); 
\draw[red, ultra thick] (2,2)--(2,4);
\draw (1,1) node {$\tilde{S}_{p-1}$};
\draw (3,3) node {$S_{p+1}=S'$};
\draw (1,3) node {$\tilde{S}_{p}$};
\draw (2, 4) node[above  ] {$ z_p$};
\draw (2, 2) node[above left ] {$ [x_1]$};
\draw (0, 2) node[ left ] {$ v'$};
\draw (4, 2) node[ right ] {$ v$};

\fill (-0.2,1.0)--(0.2,1.0) -- (0.0,0.8) --cycle;
\fill (-0.2,3.0)--(0.2,3.0) -- (0.0,3.2) --cycle;
\fill (2.2,3.0)--(1.8,3.0) -- (2.0,3.2) --cycle;
\fill (2.2,1.0)--(1.8,1.0) -- (2.0,0.8) --cycle;
\fill (1.0,-0.2)--(1.0,0.2) -- (0.8,0.0) --cycle;
\fill (1.0,1.8)--(1.0,2.2) -- (0.8,2.0) --cycle;
\fill (1.0,3.8)--(1.0,4.2) -- (0.8,4.0) --cycle;

\dtypeI{(0,0)}
\dtypeI{(0,2)}
\dtypeI{(0,4)}
\dtypeI{(2,0)}
\dtypeI{(2,2)}
\dtypeI{(2,4)}
\dtypeI{(4,2)}
\dtypeI{(4,4)}
\dtypeII{(3,2)}
\dtypeII{(3,4)}
\dtypeII{(4,3)}
\end{tikzpicture}
\end{center}
We claim that
\begin{equation*}
\nu(z_p - x_1R(x_1)) < \nu(x_1).
\end{equation*}
for any polynomial $R \in \C[x] \setminus \C$.
Taking this claim for granted we conclude the proof of the lemma. 
By Proposition \ref{prop_prop_R_practice}, we may find a square $S_p''$ adjacent to $\tilde{S}_p$ along the edge containing $[x_1], v'$ 
whose edges containing $[x_1]$ are not critically resonant and such that the triple $\tilde{S}_p,S_p'',S_{p+1}$ is flat. 
Let $\tilde{S}_{p+1}$ be the $2 \times 2$ square completing the $4\times 4$ square containing $\tilde{S}_p,S_p'',S_{p+1}$.  

Since the squares $\tilde{S}_{p-1}, \tilde{S}_{p}$ and $S_{p+1}$ are not flat,  Lemma \ref{lem_flat_replace}.$(ii)$ implies that the triple $\tilde{S}_{p-1}, S_{p}''$ and $S_{p+1}'$ is also not flat, so that 
the sequence $(S_1,\tilde{S}_2, \ldots ,\tilde{S}_{p-1}, S_{p}'', \tilde{S}_{p+1})$ 
is contained in a spiral staircase such that any edge lying in two consecutive squares
is not critically resonant. 
Lemma \ref{lem_flat_replace}.$(ii)$ applied to $\tilde{S}_{p-1},\tilde{S}_p, S_{p+1}$
implies the existence of an element $g \in A_{[x_1]}$ such that $g \cdot \tilde{S}_{p+1} = S_{p+1}$. This finishes the proof of the
 induction step.

\smallskip

We now prove our claim. 
Fix a polynomial $R\in \C[x] \setminus \C$, and consider the square $S_R$ containing $[x_1], [z_p - x_1R(x_1)]$ and $v'$. 
Since $x R(x) \in \C[x]$, the squares $S_R, \tilde{S}_p$ and $S_{p+1}$ are flat by Lemma \ref{lem_technical_flat}.$(ii)$. 
We thus have the following picture.
%
%
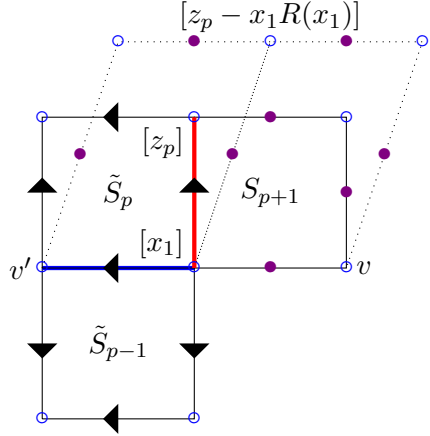
\begin{center}
\begin{tikzpicture} 
\draw  (0,0)-- (0,2) -- (2,2) -- (2, 0) --(0,0); 
\draw[ultra thick, blue] (0,2)--(2,2);
\draw  (0,2)-- (0,4) -- (2,4) -- (2, 2) --(0,2); 
\draw  (2,2)-- (2,4) -- (4,4) -- (4, 2) --(2,2); 
\draw[dotted] (0,2) -- (1,5) --(3, 5) --(2,2);
\draw[dotted] (2,2) -- (3,5) --(5, 5) --(4,2);
\draw[red, ultra thick] (2,2)--(2,4);
\draw (1,1) node {$\tilde{S}_{p-1}$};
\draw (3,3) node {$S_{p+1}$};
\draw (1,3) node {$\tilde{S}_{p}$};
\draw (2, 4) node[below left  ] {$ [z_p]$};
\draw (2, 2) node[above left ] {$ [x_1]$};
\draw (0, 2) node[ left ] {$ v'$};
\draw (3, 5) node[above  ] {$ [z_p - x_1R(x_1)]$};
\draw (4, 2) node[ right ] {$ v$};

\fill (-0.2,1.0)--(0.2,1.0) -- (0.0,0.8) --cycle;
\fill (-0.2,3.0)--(0.2,3.0) -- (0.0,3.2) --cycle;
\fill (2.2,3.0)--(1.8,3.0) -- (2.0,3.2) --cycle;
\fill (2.2,1.0)--(1.8,1.0) -- (2.0,0.8) --cycle;
\fill (1.0,-0.2)--(1.0,0.2) -- (0.8,0.0) --cycle;
\fill (1.0,1.8)--(1.0,2.2) -- (0.8,2.0) --cycle;
\fill (1.0,3.8)--(1.0,4.2) -- (0.8,4.0) --cycle;
\dtypeI{(0,0)}
\dtypeI{(0,2)}
\dtypeI{(0,4)}
\dtypeI{(2,0)}
\dtypeI{(2,2)}
\dtypeI{(2,4)}
\dtypeI{(4,2)}
\dtypeI{(4,4)}
\dtypeI{(1,5)}
\dtypeI{(3,5)}
\dtypeI{(5,5)}
\dtypeII{(3,2)}
\dtypeII{(3,4)}
\dtypeII{(4,3)}
\dtypeII{(2,5)}
\dtypeII{(4,5)}
\dtypeII{(0.5,3.5)}
\dtypeII{(2.5,3.5)}
\dtypeII{(4.5,3.5)}

\end{tikzpicture}
\end{center}
By Lemma \ref{lem_technical_flat}, there exists an element $g \in A_{[x_1]}$ such that $ g \cdot \tilde{S}_p = S_R$. 
By Lemma \ref{lem_flat_replace}.$(i)$  the triple $\tilde{S}_{p-2}, \tilde{S}_{p-1} , S_R$ are not flat since $\tilde{S}_{p-2}, \tilde{S}_{p-1}, \tilde{S}_q$ are not flat.
We have thus proven that the sequence $(S, S_1,\tilde{S}_2, \ldots , \tilde{S}_{p-1},S_R)$ is contained in a spiral staircase for which any edge lying in two consecutive squares is
not critically resonant. By Lemma~\ref{lem:step1}
the vertex $[x_1]$ is $\nu$-maximal in $S_R$, hence:
\begin{equation*}
\nu(z_p - x_1R(x_1)) < \nu(x_1),
\end{equation*}
as required.
\end{proof}

\begin{proof}[Proof of Lemma \ref{lem_flat_replace}]

By transitivity of the action of $\stame$ on the $2\times 2$ squares, we can suppose that $S_2$ is the standard $2\times 2$ square containing $[x],[t],[y],[z]$ and that $S_1$ and $S_3$ are adjacent along the vertical and horizontal edge containing $[x]$ respectively. 
Take $g_1 , g_3\in \stab([x]) \cap \stame$ such that $g_1 \cdot S_2 = S_1 $ and $g_3 S_2 = S_3$. 

Let us prove assertion $(i)$.
Since $S_1,S_2,S_3$ are not flat, Lemma \ref{lem_complete_square_1} implies that $g_1, g_3 \notin A_{[x]}$. Observe that $ g g_1\cdot  S_2 = S_1'$ and $g_3\cdot S_2 = S_3$ where $g \circ g_1 \notin A_{[x]}$, hence the squares $S_1',S_2,S_3$ are also not flat by Lemma \ref{lem_complete_square_1}.

\medskip

Let us prove assertion $(ii)$.

Consider $g, g' \in \stab([x])\cap \stame$ such that $g \cdot S_1 = S_1'  $, $g' \cdot S_2=S_2'$.
Since $g_1 \notin A_{[x]}$ but the squares $S_1',S_1,S_2$ are flat, Lemma \ref{lem_complete_square_1} implies that $g,g' \in A_{[x]}$.
Observe that $g g_1 g'^{-1} \cdot S_2' = S_1' $ and $g_3 g'^{-1} \cdot S_2' = S_3$ and that $g \circ g_1 \circ g'^{-1} , g_3 \circ g'^{-1}\notin A_{[x]}$, hence the squares $S_1',S_2',S_3$ are not flat by Lemma \ref{lem_complete_square_1}.
\end{proof}

\subsection{Degree at a non-extremal vertex}
\label{subsection_degree_non_extremal}

\begin{thm} \label{thm_bound_vert_I}
Take a valuation $\nu \in \mathcal{V}_0$. 
Consider two $2 \times 2 $ adherent squares $S$ and $S'$ at a  vertex of type I given by $[x_1]$ with $x_1 \in \CSL2$ such that the pair $(S,S')$ is contained in a vertical spiral staircase. 
Denote by $[y_1]$ the unique vertex in $S$ distinct from $[x_1]$ which belongs to the horizontal edge containing $[x_1]$.
Suppose that the edge  containing $[x_1], [y_1]$ is not critically resonant. 
Then for any vertex $v$ distinct from $[x_1]$ in $S'$ one has:
%
\begin{equation*}
\nu(v) < \dfrac{4}{3} \nu(x_1).
\end{equation*}
%
\end{thm}
One has the following picture:
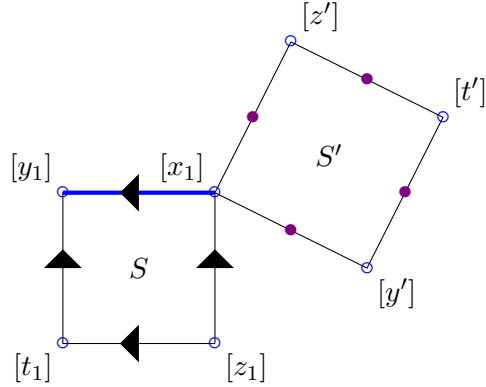
\begin{center}
\begin{tikzpicture} 
\draw    (2, 0) --(0,0)-- (0,2)-- (2,2) ; 

\draw  (3, 4)-- (5,3) -- (4,1) --(2,2);
\draw (2,0) -- (2,2) -- (3,4); 
\draw[blue,ultra thick] (2,2)--(0,2);
\fill (1,0.25)--(1, -0.25) -- (0.75, 0) --cycle;
\fill (1,2.25)--(1, 1.75) -- (0.75, 2) --cycle;
\fill (-0.25,1)--(0.25, 1) -- (0, 1.25) --cycle;
\fill (1.75,1)--(2.25, 1) -- (2, 1.25) --cycle;
\draw (1,1) node {$ S $};
\draw (3.5,2.5) node {$S'$};
\draw (0, 0) node[blue] {$\circ$} node[below left ] {$ [t_1]$};
\draw (2, 0) node[blue] {$\circ$} node[below right ] {$ [z_1]$};
\draw (2, 2)node[blue] {$\circ$} node[above left ] {$ [x_1]$};
\draw (0, 2)node[blue] {$\circ$} node[above left ] {$ [y_1]$};
\draw (3, 4)node[blue] {$\circ$} node[above  right] {$ [z']$};
\draw (5, 3)node[blue] {$\circ$} node[above  right] {$ [t']$};
\draw (4, 1) node[blue] {$\circ$} node[below  right] {$ [y']$};

\dtypeII{(2.5,3)}
\dtypeII{(4.5,2)}
\dtypeII{(3,1.5)}
\dtypeII{(4,3.5)}
\end{tikzpicture}
\end{center}


%

\begin{rem} By symmetry, observe that the same assertion holds if $[z_1]$ is $\nu$-minimal in $S$ and the pair $(S,S')$ is contained in a horizontal spiral staircase.
\end{rem}

\begin{proof} 
Consider two squares $S,S'$ and the vertices $[x_1], [y_1] \in S$ satisfying the conditions of the Theorem. 
By definition, there exists an integer $p$ and $p$ adjacent squares $S_0 = S , \ldots, S_p = S'$ containing $[x_1]$ connecting $S$ and $S'$.

Since $S_0= S$ and $S_1$ are adjacent, the vertex $[x_1]$ is $\nu$-maximal in $S_1$ by Lemma \ref{lem_easy}. 

Denote by $[z_1]$ the vertex in $S_1$ such that the vertices $[x_1]$ and $[z_1]$ are contained in the vertical edge of $S_1$ so that we are in the following situation:
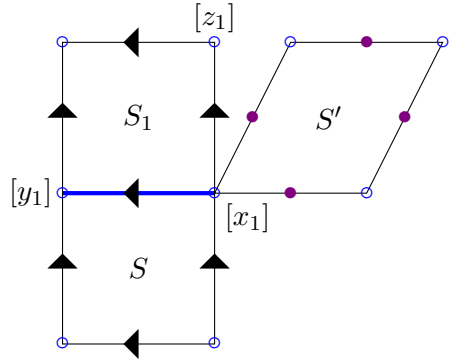
\begin{center}
\begin{tikzpicture} 
\draw[ultra thick,blue] (0,2)--(2,2);

\draw (0,2)--(0, 0) --(2,0) --(2,2);
\draw (1,1) node {$S$};
\draw (3.5,3) node {$S'$};

\fill (1.0,2.2)--(1.0,1.8) -- (0.8,2.0) --cycle;
\fill (1.0,-0.2)--(1.0,0.2) -- (0.8,0.0) --cycle;

\fill (-0.2,1.0)--(0.2,1.0) -- (0.0,1.2) --cycle;
\fill (2.2,1.0)--(1.8,1.0) -- (2.0,1.2) --cycle;

\draw (0,2) node[left] {$[y_1]$};
\draw (2,2) node[below right] {$[x_1]$};

\draw (0,2) --(0,4)--(2,4)--(2,2);
\fill (1.0,4.2)--(1.0,3.8) -- (0.8,4.0) --cycle;
\fill (-0.2,3.0)--(0.2,3.0) -- (0.0,3.2) --cycle;
\fill (2.2,3.0)--(1.8,3.0) -- (2.0,3.2) --cycle;

\draw (2,2) --(3,4)--(5,4)--(4,2)--cycle;
\draw (1,3) node {$S_1$} ;

\draw (2,4) node[above] {$[z_1]$};

\dtypeI{(0,0)}
\dtypeI{(2,0)}
\dtypeI{(0,2)}
\dtypeI{(2,2)}
\dtypeI{(0,4)}
\dtypeI{(2,4)}
\dtypeI{(3,4)}
\dtypeI{(5,4)}
\dtypeI{(4,2)}

\dtypeII{(2.5,3)}
\dtypeII{(4.5,3)}
\dtypeII{(3,2)}
\dtypeII{(4,4)}
\end{tikzpicture}
\end{center}
Fix any polynomial $R \in \C[x] \setminus \C$.
Consider $S_R$ the square containing $[x_1],[y_1]$ and $[z_1 - x_1R(x_1)]$.
By Lemma \ref{lem_technical_flat}, the squares $S_1,S_R,S_2$ are flat.
Take $\tilde{S}_R$ the $2\times 2$ square completing the $4\times 4$ square containing $S_1,S_R,S_2$. 
Lemma \ref{lem_flat_replace}.$(ii)$ implies that $S, S_R, \tilde{S}_R$ are not flat since $S,S_1,S_2$ are not flat. 
This proves in particular that the vertex $[x_1]$ is $\nu$-maximal in $S_R$, hence:
 \begin{equation*}
 \nu(z_1 - x_1R(x_1)) < \nu(x_1).
 \end{equation*}
 By Proposition \ref{prop_prop_R_practice}, there exists a square $S_1'$ adjacent to $S$ along $[x_1],[y_1]$ such that the squares $S_1', S_1,S_2$ are flat and such that the vertical edge in  $S_1'$ containing $[x_1]$ is not critically resonant. 
 Consider the square $S_2'$ completing the $4\times 4$ square containing $S_1',S_1,S_2$. 
By construction, the edge $S_1' \cap S_2'$ is not critically resonant.
Observe also that Lemma \ref{lem_apply_R} implies that for any vertex $v \in S_1'$ distinct from $[x_1]$ and $[y_1]$, one has:
\begin{equation*}
\nu(v)< \max\left (\nu(y_1), \dfrac{4}{3} \nu(x_1) \right ). 
\end{equation*}

%

Suppose that $p\geqslant 3$, then the triple $(S,S_1',S')$ satisfies the assumptions of  Theorem \ref{thm_technical_degree_stab_I} and we conclude  that for any vertex $v$ distinct from $[x_1]$ in  $S'$:
\begin{equation*}
\nu(v) < \dfrac{4}{3} \nu(x_1).
\end{equation*}
We have thus proven the theorem.

\medskip

Suppose that $p=2$ and the squares $S' $ and $S_1$ are adjacent. We are thus in the following situation:
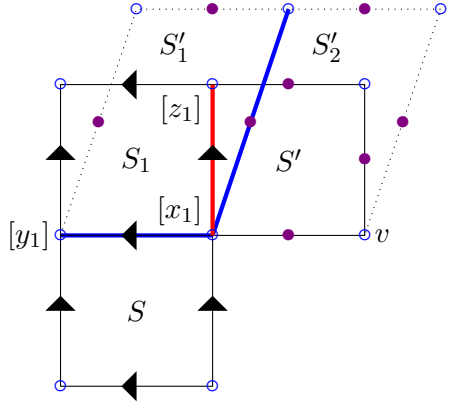
\begin{center}
\begin{tikzpicture} 
\draw  (0,0)-- (0,2) -- (2,2) -- (2, 0) --(0,0); 
\draw[ultra thick, blue] (0,2)--(2,2);

\draw (1.5,4.5) node {$S'_1$};
\draw (3.5,4.5) node {$S'_2$};

\draw  (0,2)-- (0,4) -- (2,4) -- (2, 2) --(0,2); 
\draw  (2,2)-- (2,4) -- (4,4) -- (4, 2) --(2,2); 
\draw[dotted] (0,2) -- (1,5) --(3, 5) --(2,2);
\draw[dotted] (2,2) -- (3,5) --(5, 5) --(4,2);
\draw[ultra thick,blue] (2,2)--(3,5);
\draw[red, ultra thick] (2,2)--(2,4);
\draw (1,1) node {$S$};
\draw (3,3) node {$S'$};
\draw (1,3) node {$S_{1}$};
\draw (2, 4) node[below left  ] {$ [z_1]$};
\draw (2, 2) node[above left ] {$ [x_1]$};
\draw (0, 2) node[ left ] {$[y_1]$};
\draw (4, 2) node[ right ] {$ v$};
\fill (-0.2,1.0)--(0.2,1.0) -- (0.0,1.2) --cycle;
\fill (2.2,1.0)--(1.8,1.0) -- (2.0,1.2) --cycle;

\fill (2.2,3.0)--(1.8,3.0) -- (2.0,3.2) --cycle;
\fill (-0.2,3.0)--(0.2,3.0) -- (0.0,3.2) --cycle;

\fill (1.0,-0.2)--(1.0,0.2) -- (0.8,0.0) --cycle;
\fill (1.0,1.8)--(1.0,2.2) -- (0.8,2.0) --cycle;
\fill (1.0,3.8)--(1.0,4.2) -- (0.8,4.0) --cycle;
\dtypeI{(0,0)}
\dtypeI{(0,2)}
\dtypeI{(0,4)}
\dtypeI{(2,0)}
\dtypeI{(2,2)}
\dtypeI{(2,4)}
\dtypeI{(4,2)}
\dtypeI{(4,4)}
\dtypeI{(1,5)}
\dtypeI{(3,5)}
\dtypeI{(5,5)}
\dtypeII{(3,2)}
\dtypeII{(3,4)}
\dtypeII{(4,3)}
\dtypeII{(2,5)}
\dtypeII{(4,5)}
\dtypeII{(0.5,3.5)}
\dtypeII{(2.5,3.5)}
\dtypeII{(4.5,3.5)}

\end{tikzpicture}
\end{center}
where $v$ is the unique vertex in $S'$ distinct from $[x_1]$ which belongs to the horizontal edge containing $[x_1]$. 
By Theorem \ref{thm_technical_degree_stab_I}, $[x_1]$ is $\nu$-maximal in $S_2'$, hence it is also $\nu$-maximal in $S'$ and $\nu(v) < 4/3 \nu(x_1)$.
Observe also that Lemma \ref{lem_apply_R} implies that:
\begin{equation*}
\nu(z_1) < \dfrac{4}{3}\nu(x_1).
\end{equation*}
This proves that for any $v \in S'$ distinct from $[x_1]$, one has:
\begin{equation*}
\nu(v)< \dfrac{4}{3} \nu(x_1),
\end{equation*}
and the theorem holds.
\end{proof}

\red 
\black 
\subsection{Degree estimates at a $\nu$-minimal vertex}
\label{subsection_degree_adherent_minimal}

\begin{thm} \label{thm_bound_stab_I_minimal} 
%
%
Consider any valuation $\nu \in \mathcal{V}_0$. Let $S$ and $S'$ be two adherent $2\times 2$ squares intersecting at a vertex $v$ which is $\nu$-minimal in $S$. 
Then the following holds. 
\begin{enumerate}
\item[(i)] The vertex $v $ is the $\nu$-maximal vertex of $S'$.
\item[(ii)] If $v' $ is a vertex in $S'$ which does not belong to any square adjacent to $S$, then we have:
\begin{equation*}
\nu(v') < \dfrac{4}{3} \nu(v)
\end{equation*} 
\end{enumerate}

\end{thm}

\begin{rem}  Suppose that the vertex $v \in S'$ belongs to a square adjacent to $S$, then we will apply the estimates in Theorem \ref{thm_vert_band} instead. 
%
%
%
%
%
%
\end{rem}

\begin{proof}


Let us prove assertions $(i)$ and $(ii)$.

Suppose first that $S$ and $S'$ belong to a 4x4 squares containing $S, S', S_1$ and $S_2$ as in the figure below. 
Since $S,S_1$ and $S,S_2$ are adjacent along an edge containing $v$, Lemma \ref{lem_easy} implies that we are in the following situation:
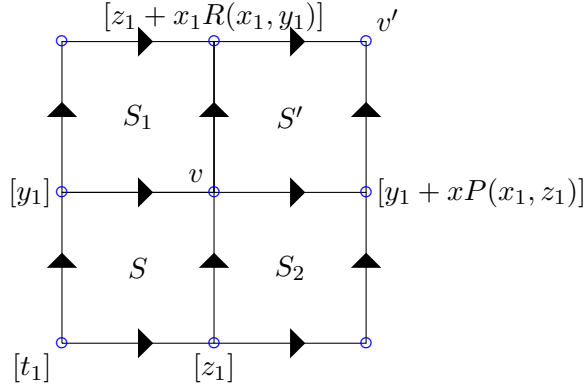
\begin{center}
\begin{tikzpicture} 
\draw  (0,0)-- (0,2) -- (2,2) -- (2, 0) --(0,0); 
\draw  (0,2)-- (0,4) -- (2,4) -- (2, 2) --(0,2); 
\draw (2,4)--(4,4)--(4,2)-- (2,2)-- (2,4)  ;
\draw(2,2)-- (2,4);
\draw (1,1) node {$S$};
\draw (1,3) node {$S_1$};
\draw (3,3) node {$S'$}; 
\draw (2, 4) node[above  ] {$ [z_1 + x_1R(x_1,y_1)]$};
\draw (2, 2) node[above left ] {$ v$};
\draw (4, 4) node[above right ] {$ v'$};
\draw (4, 2) node[ right ] {$ [y_1+ xP(x_1, z_1)]$};
\draw(0,2) node [left] {$[y_1]$};
\draw (2,0) node [below]{$[z_1]$};

\draw (0,0) node [below left]{$[t_1]$};
\fill (1.0,-0.2)--(1.0,0.2) -- (1.2,0.0) --cycle;

\fill (-0.2,3.0)--(0.2,3.0) -- (0.0,3.2) --cycle;
\fill (-0.2,1.0)--(0.2,1.0) -- (0.0,1.2) --cycle;
\fill (2.2,1.0)--(1.8,1.0) -- (2.0,1.2) --cycle;

\fill (4.2,3.0)--(3.8,3.0) -- (4.0,3.2) --cycle;
\fill (1.0,2.2)--(1.0,1.8) -- (1.2,2.0) --cycle;

\fill (2.2,3.0)--(1.8,3.0) -- (2.0,3.2) --cycle;
\fill (1.0,4.2)--(1.0,3.8) -- (1.2,4.0) --cycle;

\fill (3.0,4.2)--(3.0,3.8) -- (3.2,4.0) --cycle;
\fill (3.0,2.2)--(3.0,1.8) -- (3.2,2.0) --cycle;

\draw (2,0) --(4,0)--(4,2);
\draw (3,1) node {$S_2$};
\fill (3.0,-0.2)--(3.0,0.2) -- (3.2,0.0) --cycle;
\fill (4.2,1.0)--(3.8,1.0) -- (4.0,1.2) --cycle;
\dtypeI{(0,0)}
\dtypeI{(2,0)}
\dtypeI{(4,0)}
\dtypeI{(2,4)}
\dtypeI{(0,4)}
\dtypeI{(4,4)}
\dtypeI{(0,2)}
\dtypeI{(2,2)}
\dtypeI{(4,2)}
\end{tikzpicture}
\end{center}
where $v = [x_1], [y_1], [z_1],[t_1] \in S$ and $P, R \in \C[x,y]\setminus \C$. 
Observe that $v$ is $\nu$-maximal in $S'$ and we have proved assertion $(i)$.
Since the squares $S,S_1, S_2$ are flat, Lemma \ref{lem_technical_flat} and Lemma \ref{lem_complete_square_1} imply that $P \in \C[x] \setminus \C$ or $R\in \C[x] \setminus \C $.
Suppose that $P\in \C[x] \setminus \C$, then we have $  (4/3) \nu(x_1)>\nu(y_1 + x_1P(x_1)) = (\deg(P)+1) \nu(x_1) > \nu(v')$ proving $(ii)$ as required.

\bigskip

Suppose next that $(S,S')$ is contained in a spiral staircase. 
Choose a sequence of squares $S_0= S, \ldots ,S_p= S'$ of squares containing $v$ and connecting $S$ and $S'$ such that each triple of consecutives squares is not flat. 
By symmetry, we can suppose that $S_0$ and $S_1$ are adjacent along a horizontal edge containing $v$. 
Observe that Lemma \ref{lem_easy} applied to $S,S_1$ implies that the edge $S_1 \cap S_2$ contains the $\nu$-minimal vertex in $S_1$.

\smallskip

If the edge $S_1 \cap S_2$ is not critically resonant, then the pair $(S_1,S')$ is contained in a horizontal staircase so that one has the following picture:
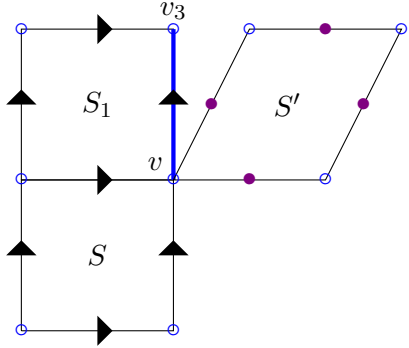
\begin{center}
\begin{tikzpicture} 
\draw  (0,0)-- (0,2) -- (2,2) -- (2, 0) --(0,0); 
\draw  (0,2)-- (0,4) -- (2,4) -- (2, 2) --(0,2); 
\draw[ultra thick,blue](2,2)-- (2,4);
\draw (1,1) node {$S$};
\draw (1,3) node {$S_1$};

\draw (2,2)--(3,4)--(5,4)--(4,2)--(2,2);
\draw (3.5,3) node {$S'$};
\draw (2, 4) node[above  ] {$ v_3$};
\draw (2, 2) node[above left ] {$ v$};
\fill (1.0,-0.2)--(1.0,0.2) -- (1.2,0.0) --cycle;

\fill (-0.2,3.0)--(0.2,3.0) -- (0.0,3.2) --cycle;
\fill (-0.2,1.0)--(0.2,1.0) -- (0.0,1.2) --cycle;
\fill (2.2,1.0)--(1.8,1.0) -- (2.0,1.2) --cycle;

\fill (1.0,2.2)--(1.0,1.8) -- (1.2,2.0) --cycle;

\fill (2.2,3.0)--(1.8,3.0) -- (2.0,3.2) --cycle;
\fill (1.0,4.2)--(1.0,3.8) -- (1.2,4.0) --cycle;

\dtypeI{(0,0)}
\dtypeI{(2,0)}
\dtypeI{(0,2)}
\dtypeI{(2,2)}
\dtypeI{(0,4)}
\dtypeI{(2,4)}
\dtypeI{(3,4)}
\dtypeI{(5,4)}
\dtypeI{(4,2)}

\dtypeII{(2.5,3)}
\dtypeII{(4.5,3)}
\dtypeII{(3,2)}
\dtypeII{(4,4)}
\end{tikzpicture}
\end{center}

By Theorem \ref{thm_bound_vert_I}, the vertex $v$ is $\nu$-minimal in $S'$ and one has
$\nu(v') < (4/3) \nu(v)$ for all $v' \neq v $ in $S'$. 
We have thus proved assertion $(i)$ and $(ii)$.

\smallskip 

We now suppose that the edge $S_1\cap S_2$ is critically resonant. 
Denote by $[f_1]$ the $\nu$-minimal vertex in $S_1$ and by $v = [f_2]$. 
Fix any polynomial $R\in\C[x]\setminus \C$ and take $S_R$ the square containing $[f_1 - f_2R(f_2)],[f_2]$ and the edge $S_1\cap S_0$.
Lemma \ref{lem_technical_flat}.$(ii)$ implies that the squares $S_1,S_R,S_2$ are flat. 
Take $S_R'$ the $2\times2 $ square completing the $4\times4$  square containing $S_1,S_R,S_2$. 
Since the squares $S,S_1,S_2$ are not flat, Lemma \ref{lem_flat_replace} implies that $S,S_R,S_R'$ are also not flat. 
In particular, the squares $S$ and $ S_R$ intersect along an edge containing $v$, Lemma \ref{lem_easy} implies that 
\begin{equation*}
\nu(f_1 - f_2R(f_2)) < \nu(f_2). 
\end{equation*}
By Proposition \ref{prop_prop_R_practice} applied to the edge $[f_1],[f_2]$, we can find a square $S_1'$ adjacent to $S$ along $S\cap S_1$ and $g \in A_{v}$ such that $g \cdot S_1 = S_1'$ and such that the vertical edge containing $v$ in $S_1$ is not critically resonant.
By Lemma \ref{lem_complete_square_1}, the squares  $S_1,S_1',S_2$ are flat. 
Take $S_2'$ the $2\times2 $ square completing the $4\times 4$ square containing $S_1,S_2,S_1'$. 
As the three squares $S,S_1,S_2$ are not flat, Lemma \ref{lem_flat_replace} implies that the squares $S, S_1', S_2'$ are also not flat. 
\smallskip

If $p\geqslant 3$, then  the pair $(S_1', S_p)$ is contained in a horizontal spiral staircase and the edge $S_1' \cap S_2'$ is not critically resonant. Hence, by Theorem \ref{thm_bound_vert_I}, the vertex $v $ is $\nu$-maximal in $S'$ and for any vertex $v'$ distinct from $v$ in $S'$, one has:
\begin{equation*}
\nu(v') < \dfrac{4}{3}\nu(v),
\end{equation*} 
proving $(i)$ and $(ii)$ as required.

\smallskip

Suppose that $p=2$ so that $S_2 = S'$. 
Observe that $S_2'$ and $S'$ are adjacent along a horizontal edge containing $v$. 
Since $v$ is $\nu$-maximal in $S_2'$, it is also $\nu$-maximal on the edge $S_2' \cap S'$. 
Since $v$ is $\nu$-maximal on the vertical  edge $S_1 \cap S'$, we have thus proven that $v$ is $\nu$-maximal in $S'$ and assertion $(i)$ holds.
Take $v_2$ the vertex contained in $S' \cap S_2'$ distinct from $v$. Since the edge $S_1' \cap S_2'$ is not critically resonant, Lemma \ref{lem_apply_R} implies that $\nu(v_2) < 4/3 \nu(v) $. Hence, for any vertex $v' \in S' $ not contained in the same band as $S$, one has $\nu(v') < (4/3) \nu(v)$ proving $(ii)$ as required.
\end{proof}

\black

\subsection{Proof of Theorem \ref{thm_degree_versus_distance}}
\label{subsection_proof_thm_degree_distance}

Take $S_0$ the standard square containing $[x],[y],[z],[t]$.
Fix a valuation $\nu \in \mathcal{V}_0$ such that:
\begin{equation*}
\max(\nu(y) + \nu(t), \nu(z) + \nu(t)) <\nu(x) < \min(\nu(y), \nu(z), \nu(t)).
\end{equation*}
Pick any vertex $v$ of type I such that the geodesic segment in $\mathcal{C}$ joining $[\Id]$ to $v$ intersects an edge of the standard square. 
%
Choose any geodesic segment $\gamma : [0,n] \to \mathcal{C}_\nu$ joining $[t]$ to $v$ such that the sequence $(\nu(\gamma(i)))_{0 \le i \le n}$ is maximal 
for the lexicographic order in $\mathbb{R}^{n+1}$
among all geodesic segments joining $[t]$ to $v$. 
Pick any sequence $\tilde{S}_0,\ldots , \tilde{S}_{n-1}$ of $2\times 2$ squares such that $\gamma(i) , \gamma(i+1) \in \tilde{S}_{i}$ for all $i\leqslant n-1$. 
We claim that the following properties hold. 
\begin{enumerate}
\item[(A)] The vertex $\gamma(i)$ is the unique $\nu$-maximal vertex in $\tilde{S}_i$ for all $0 \le i \le n-1$.
\item[(B)] We have  $\nu(\gamma(i+1)) < \dfrac{4}{3} \nu(\gamma(i))$ for all $1\leqslant i \leqslant n-1$.
\item[(C)] For any other valuation $\nu'\in \mathcal{V}_0$ satisfying \eqref{eq_spe_valuation}, the vertex $\gamma(i)$ is also $\nu'$-maximal in $\tilde{S}_i$ for all $0 \le i \le n-1$.
\end{enumerate}
Observe first that these properties $(A), (B)$ and $(C)$ imply Theorem $(i)$ and $(ii)$. 
 
Observe the slight discrepancy in the indices between $(A) ,(C)$ and $(B)$. We do not claim that $\nu(\gamma(1)) < \dfrac{4}{3} \nu([t])$ in general. 
This claim is however sufficient to imply Theorem~\ref{thm_degree_versus_distance} (1) and (2).

Observe that assertion $(C)$ implies that $d_\nu([t] , v) \geq d_{\nu'}([t],v)$ and we conclude by symmetry that $d_\nu([t],v) = d_{\nu'}([t],v)$ for any other valuation $\nu' \in \mathcal{V}_0$ satisfying \eqref{eq_spe_valuation}.
This proves that assertion $2$ of the theorem holds.
 

\smallskip

We shall prove the claim by induction on $n\geqslant 1$. 
Fix another valuation $\nu' \in \mathcal{V}_0$ satisfying \eqref{eq_spe_valuation}.

Suppose $n=1$. There is only one square $\tilde{S}_0$ containing $[t]$ and $v$ (it may not be the standard square).
Since $n=1$, we only need to prove assertions $(A)$ and $(C)$.
\begin{lem} \label{lem_adjacent_standard_square} 
Take any $2\times 2$ square $S$ adjacent to the standard square $S_0$ along an edge containing $[t]$. Then the vertex $[t]$ is $\nu$-maximal in $S$.

Moreover, denote by $v_1$ the vertex in $S \cap S_0$ distinct from $[t]$ in $S$ and by $v_2$ the vertex distinct from $v_1$ for which the vertices $[t],v_2$ form an edge of $S$. Then one has $\nu(v_2) < \nu(v_1)$.   
\end{lem} 
Grant this lemma. If $\tilde{S}_0$ and $S_0$ are adjacent along an edge containing $[t]$ Lemma \ref{lem_adjacent_standard_square}  implies assertions (A) and (C) immediately.
%
Suppose now that $\tilde{S}_0$ and $S_0$ are adherent at $[t]$. 
If the squares $\tilde{S}_0$ and $S_0$ are flat, then Lemma \ref{lem_adjacent_standard_square} applied to the two squares adjacent to both $S_0$ and $\tilde{S}_0$ again implies that $[t]$ is also $\nu$-maximal and $\nu'$-maximal in $\tilde{S}_0$.  

Otherwise $(S_0,\tilde{S}_0)$ are contained in a spiral staircase. 
Take an integer $p\geqslant 2$ and a sequence of squares $S_0,S'_1, \ldots , S'_p = \tilde{S}_0$ connecting $S_0$ to $\tilde{S}_0$ such that each three consecutive squares are not flat.
If the edge $S_0 \cap S_1'$ is not critically resonant, take $[f_1]$ the vertex distinct from $[t]$ of the edge $S_2'\cap S_1'$. 
Denote by $[f_2]$ the vertex in $ S_0 \cap S_1'$ distinct from $[t]$.  
By Lemma \ref{lem_adjacent_standard_square}, one has $\nu(f_1) < \nu(t)$ 	and $\nu(f_1) < \nu(f_2) $.
Take any polynomial $R\in \C[x] \setminus \C$, denote by $S_R$ the square containing $[f_1- tR(t)], [t] , [f_2] $. 
By construction, $S_R$ is adjacent to $S_0$ and Lemma \ref{lem_adjacent_standard_square} implies that $\nu(f_1-tR(t) ) < \nu(t)$. 
By Proposition  \ref{prop_prop_R_practice}, we can find a square $S_1'' =  g \cdot S'_1$ with $g\in A_{[t]}$ such that $S_1',S_1'', S_2'$ are flat and the edge containing $[t]$ in $S_1''$ distinct from $S_0\cap S_1'$ is not critically resonant.
Take $S_2''$ the $2\times 2$ square completing the $4\times 4$ square containing $S_1',S_2',S_1''$. 

If $p\geqslant 3$, the triple $S_0,S_1',S_2'$ is not flat by Lemma \ref{lem_flat_replace}.$(ii)$, hence $S_0,S_1'', S_2''$ are also not flat. The squares $(S_0,S_1'', \tilde{S}_0)$ thus satisfy the conditions of Theorem \ref{thm_technical_degree_stab_I}, and 
 $[t]$ is $\nu$-maximal in $\tilde{S}_0$. 
 If $p = 2$, then $S_2' = \tilde{S}_0$ and by Theorem \ref{thm_technical_degree_stab_I} applied to $(S_0,S_1'',S_2'')$, the vertex is $\nu$-maximal in $S_2''$.
 Since $S_2''$ and $\tilde{S}_0$ are adjacent along an edge containing $[t]$ and $[t]$ is also $\nu$-maximal in $S_1'$, it is also $\nu$-maximal in $\tilde{S}_0$, 
 proving assertion $(A)$ as required.
 Observe that the same argument also applies for $\nu' \in \mathcal{V}_0$, hence assertion $(C)$ also holds.
 
  
 We have thus proven the claim for $n=1$.

\bigskip

Let us suppose that the claim is true for $n \geqslant 1$.  We shall prove it for $n+1$.
Choose any geodesic $\gamma: [0,n+1] \to \mathcal{C}_\nu$ joining $[t]$ to a vertex $v$ for which the sequence $(\nu(\gamma(i)))_{0 \le i \le n}$  is maximal. 
Denote by $v_i = \gamma(i)$. Take any sequence of squares $\tilde{S}_0, \ldots, \tilde{S}_n$ for which $v_i, v_{i+1} \in \tilde{S}_i$.

By our induction hypothesis applied to the vertex $v_n$, the sequence $\tilde{S}_0, \ldots , \tilde{S}_{n-1}$ satisfy assertions $(A)$, $(B)$ and $(C)$.
 \smallskip

 Suppose first that $\tilde{S}_{n-1}$ and $\tilde{S}_n$ are adjacent or equal. 
 Observe that assertion $(A)$ implies that $v=\gamma_{n+1}$ cannot belong to the square $\tilde{S}_{n-1}$, otherwise it would contradict the fact that $\gamma$ is a geodesic in $\mathcal{C}_\nu$ (recall that in this graph we draw an edge joining the $\nu$-maximal to the $\nu$-minimal edge).
This implies that $\tilde{S}_{n-1}$ and $\tilde{S}_n$ are adjacent along an edge containing the $\nu$-minimal vertex in $\tilde{S}_{n-1}$. Lemma \ref{lem_easy} 
shows that the vertex in $\tilde{S}_{n-1} \cap \tilde{S}_n$ which is not $\nu$-minimal in $\tilde{S}_{n-1}$ is $\nu$-maximal in $\tilde{S}_n$.
By the maximality of the sequence $(\nu(\gamma(i)))_{0 \le i \le n}$ the vertex $v_n$ cannot be $\nu$-minimal in $\tilde{S}_{n-1}$, hence
 is $\nu$-maximal in $\tilde{S}_n$, proving assertion $(A)$. The following figure summarizes the situation:
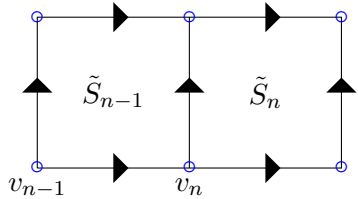
\begin{center}
\begin{tikzpicture}
\draw (0,0) --(4,0)--(4,2)--(0,2) --(0,0); 
\draw (2,0) --(2,2);

\draw (0,0) node[below] {$v_{n-1}$};

\draw (2,0) node[below ] {$v_n$};

\fill (1.0,-0.2)--(1.0,0.2) -- (1.2,0.0) --cycle;
\fill (3.0,-0.2)--(3.0,0.2) -- (3.2,0.0) --cycle;
\fill (1.0,1.8)--(1.0,2.2) -- (1.2,2.0) --cycle;
\fill (3.0,1.8)--(3.0,2.2) -- (3.2,2.0) --cycle;

\fill (-0.2,1.0)--(0.2,1.0) -- (0.0,1.2) --cycle;
\fill (1.8,1.0)--(2.2,1.0) -- (2.0,1.2) --cycle;
\fill (3.8,1.0)--(4.2,1.0) -- (4.0,1.2) --cycle;

\draw (3,1) node {$\tilde{S}_{n}$};
\draw (1,1) node {$\tilde{S}_{n-1}$};
\dtypeI{(0,0)}
\dtypeI{(2,0)}
\dtypeI{(4,0)}
\dtypeI{(0,2)}
\dtypeI{(2,2)}
\dtypeI{(4,2)}
\end{tikzpicture}
\end{center}
 Since $v_{n-1}$ is also $\nu'$-maximal in $\tilde{S}_{n-1}$, the vertex $v_n$ is also $\nu'$-maximal in $\tilde{S}_n$ by Lemma \ref{lem_easy}. We have thus proven assertion $(C)$.

Let us check that $\tilde{S}_{n-1}$ satisfies the condition of Theorem \ref{thm_vert_band}.
Take another square $\tilde{S}$ adjacent to $\tilde{S}_{n-1}$ containing $v_{n-1},v_n$. 
Observe that the sequence $\tilde{S}_0, \ldots , \tilde{S}_{n-2}, \tilde{S}$ satisfies the conditions of the theorem and contains $v_n$ which is at distance $n$.
We apply our induction hypothesis to the vertex $v_n$ and to the sequence of squares $\tilde{S}_0, \ldots , \tilde{S}_{n-2}, \tilde{S}$. Assertion $(A)$ implies that the vertex $v_{n-1}$ is $\nu$-minimal in $\tilde{S}$, as required.

We may thus apply Theorem \ref{thm_vert_band} to the band  $\tilde{S}_{n-1} \cup \tilde{S}_n$ which yields
\begin{equation*}
\nu(v_{n+1}) < \dfrac{4}{3} \nu(v_n),
\end{equation*}
proving $(B)$, as required.
\bigskip

Suppose that the squares $\tilde{S}_{n-1}, \tilde{S}_n$ are adherent and flat.
If $v_n, v_{n-1}$ form an edge of $\tilde{S}_{n-1}$, then we can find a band of two squares containing $v_{n-1},v_n,v_{n+1}$, which corresponds to the previous situation.
Otherwise $(v_n ,  v_{n-1})$ is not an edge of $\tilde{S}_{n-1}$, and since $v_{n-1}$ is $\nu$-maximal and $\nu'$-maximal in $\tilde{S}_{n-1}$ by assertions $(A)$ and $(C)$, the vertex $v_n$ is $\nu$-minimal and $\nu'$-minimal in $\tilde{S}_{n-1}$.
Observe that the vertex $v_{n+1}$ cannot belong to a band containing $v_n,v_{n-1}$ since we have chosen a geodesic $\gamma$ for which the sequence $(\nu(\gamma(i)))_{0 \le i \le n}$ is maximal.
We thus arrive at the following situation:
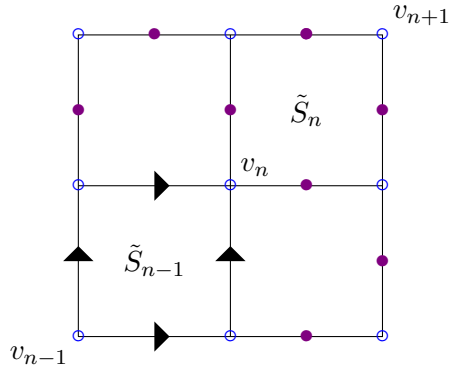
\begin{center}
\begin{tikzpicture} 
\draw (0,0)--(4,0)--(4,4)--(0,4)--(0,0);

\draw (2,0) --(2,4);
\draw (0,2)--(4,2);
\draw (1,1) node {$\tilde{S}_{n-1}$};

\draw (0,0) node[below left] {$v_{n-1}$};

\draw (2,2)node[above right] {$v_n$};

\draw (4,4) node [above right] {$v_{n+1}$};

\draw (3,3) node {$\tilde{S}_n$};

\fill (1.0,-0.2)--(1.0,0.2) -- (1.2,0.0) --cycle;
\fill (1.0,1.8)--(1.0,2.2) -- (1.2,2.0) --cycle;

\fill (-0.2,1.0)--(0.2,1.0) -- (0.0,1.2) --cycle;
\fill (1.8,1.0)--(2.2,1.0) -- (2.0,1.2) --cycle;
\dtypeI{(0,0)}
\dtypeI{(2,0)}
\dtypeI{(4,0)}
\dtypeI{(2,4)}
\dtypeI{(0,4)}
\dtypeI{(4,4)}
\dtypeI{(0,2)}
\dtypeI{(2,2)}
\dtypeI{(4,2)}
\dtypeII{(3,0)}
\dtypeII{(3,2)}
\dtypeII{(3,4)}
\dtypeII{(0,3)}
\dtypeII{(2,3)}
\dtypeII{(4,3)}
\dtypeII{(1,4)}
\dtypeII{(3,4)}
\dtypeII{(4,1)}
\end{tikzpicture}
\end{center}
By Theorem \ref{thm_bound_stab_I_minimal} $(i)$ and $(ii)$ applied to $\tilde{S}_{n-1}$ and $\tilde{S}_n$,  the vertex $\nu_{n}$ is $\nu$-maximal and $\nu'$-maximal  in $\tilde{S}_n$ (hence $(A),(C)$ hold), and one has $\nu(v_{n+1}) < 4/3 \nu(v_n)$, and assertion $(B)$ holds.

\bigskip

Suppose that the squares $\tilde{S}_{n-1}, \tilde{S}_n$ are contained in a spiral staircase.

Let us suppose first that the vertices $v_{n-1},v_n$ do not belong to the same edge of $\tilde{S}_{n-1}$. By assertions $(A)$ and $(C)$ applied to $v_{n-1}$, the vertex $v_{n-1}$ is $\nu$-maximal and $\nu'$-maximal in $\tilde{S}_{n-1}$, hence $v_n$ is $\nu$-minimal and $\nu'$-minimal in $\tilde{S}_{n-1}$. 
We thus have the following figure:
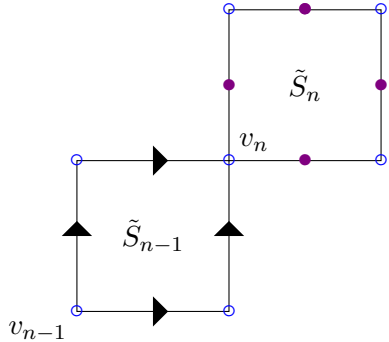
\begin{center}
\begin{tikzpicture} 
\draw (2,2)--(2,0) --(0,0) --(0,2) --(2,2) --(4,2) --(4,4)--(2,4)--(2,2);
\draw (1,1) node {$\tilde{S}_{n-1}$};

\draw (0,0) node[below left] {$v_{n-1}$};

\draw (2,2)node[above right] {$v_n$};


\draw (3,3) node {$\tilde{S}_n$};

\fill (1.0,-0.2)--(1.0,0.2) -- (1.2,0.0) --cycle;
\fill (1.0,1.8)--(1.0,2.2) -- (1.2,2.0) --cycle;

\fill (-0.2,1.0)--(0.2,1.0) -- (0.0,1.2) --cycle;
\fill (1.8,1.0)--(2.2,1.0) -- (2.0,1.2) --cycle;
\dtypeI{(0,0)}
\dtypeI{(2,0)}
\dtypeI{(2,4)}
\dtypeI{(4,4)}
\dtypeI{(0,2)}
\dtypeI{(2,2)}
\dtypeI{(4,2)}
\dtypeII{(3,2)}
\dtypeII{(3,4)}
\dtypeII{(2,3)}
\dtypeII{(4,3)}
\dtypeII{(3,4)}
\end{tikzpicture}
\end{center}
In particular, by Theorem \ref{thm_bound_stab_I_minimal}.$(i)$ applied to the squares $\tilde{S}_{n-1}, \tilde{S}_n$ implies that $v_n$ is $\nu$-maximal and $\nu'$-maximal in $\tilde{S}$, proving $(A)$ and $(C)$.  
Observe that $v_{n+1}$ cannot belong to a band containing $v_{n-1}, v_n$ since we have chosen the geodesic such that $\nu(\gamma(i))$ is maximal.
In particular, Theorem \ref{thm_bound_stab_I_minimal}.$(ii)$ implies that:
\begin{equation*}
\nu(v_{n+1}) < \dfrac{4}{3} \nu(v_n),
\end{equation*}
proving $(B)$ as required.
\bigskip 

Let us suppose that the vertices $v_{n-1},v_n$ belong to an edge of $\tilde{S}_{n-1}$.
Since the argument are similar for horizontal edges, we can suppose that the edge joining $v_{n-1}, v_n$ is vertical, and the pair $(\tilde{S}_{n-1}, \tilde{S}_n)$ belongs to a vertical spiral staircase.

Write by $v_n = [f_2]$ and let $[f_1]$ be the vertex distinct from $v_n$ in $\tilde{S}_{n-1}$ which belongs to the horizontal edge containing $v_n$.
For any polynomial $R\in \C[x] \setminus \C$, denote by  $S_R$ the $2\times 2$ containing $[f_2], [f_1 - f_2R(f_2)], v_{n-1}$.
We thus have the following figure:
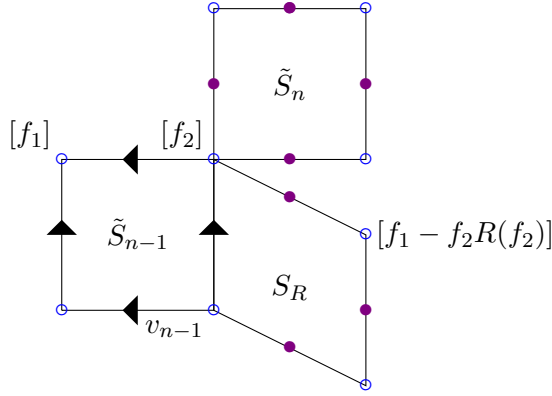
\begin{center}
\begin{tikzpicture} 
\draw (2,2)--(2,4)--(4,4)--(4,2)--(2,2);
\draw (2,2) --(4,1)--(4,-1)--(2,0)--(2,2);
\draw (4,1) node [right] {$[f_1 - f_2R(f_2)]$};
\draw (3,0.3) node {$S_R$};
\draw (0,0) -- (0,2) -- (2,2) --(2, 0)--(0,0);
\fill (1.0,2.2)--(1.0,1.8) -- (0.8,2.0) --cycle;

\fill (1.0,-0.2)--(1.0,0.2) -- (0.8,0.0) --cycle;
\fill (-0.2,1.0)--(0.2,1.0) -- (0.0,1.2) --cycle;

\fill (2.2,1.0)--(1.8,1.0) -- (2.0,1.2) --cycle;
\draw (3,3) node {$\tilde{S}_{n}$};

\draw (1,1) node {$\tilde{S}_{n-1}$};
\draw (2,0) node[below left] {$v_{n-1}$};
\draw (2,2) node[above left] {$[f_2]$};

\draw (0,2) node[above left] {$[f_1]$};
\dtypeI{(0,0)}
\dtypeI{(2,0)}
\dtypeI{(2,4)}
\dtypeI{(4,4)}
\dtypeI{(0,2)}
\dtypeI{(2,2)}
\dtypeI{(4,2)}
\dtypeII{(3,2)}
\dtypeII{(3,4)}
\dtypeII{(2,3)}
\dtypeII{(4,3)}
\dtypeII{(3,4)}
\dtypeI{(4,1)}
\dtypeI{(4,-1)}
\dtypeII{(3,1.5)}
\dtypeII{(3,-0.5)}
\dtypeII{(4,0)}
\end{tikzpicture}
\end{center}
Using our induction hypothesis for the vertex $v_n$ and to the sequence of squares $\tilde{S}_0, \ldots  , \tilde{S}_{n-2}, S_R$, assertions $(A)$ and $(C)$ imply that the vertex $v_{n-1}$ is $\nu$-maximal and $\nu'$-maximal in $S_R$, hence
$\nu(f_1 - f_2R(f_2)) < \nu(f_1)$ and $\nu'(f_1 - f_2R(f_2)) < \nu'(f_1)$.
By Proposition \ref{prop_prop_R_practice}, we can find a square $S'$ containing $v_{n-1} , v_n$ for which the horizontal edge containing $v_n$ is not critically resonant and such that there exists $g \in A_{v_n}$ such that $g \cdot S' = \tilde{S}_{n-1}$. 
By Lemma \ref{lem_flat_replace}, since $(\tilde{S}_{n-1}, \tilde{S}_n)$ is contained in a vertical spiral staircase, this implies that the pair $(S', \tilde{S}_n)$ is also contained in a vertical spiral staircase. 
 Since $v_n$ is neither $\nu$-maximal nor $\nu$-minimal in $S'$, the pair $(S',\tilde{S}_n)$ satisfies the conditions of Theorem \ref{thm_bound_vert_I}.

One has the following figure:
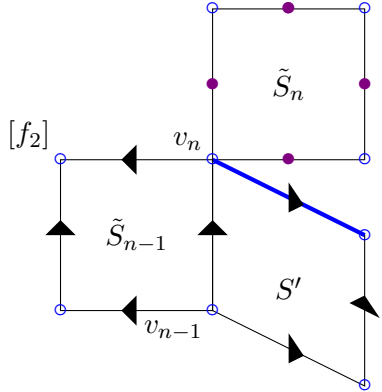
\begin{center}
\begin{tikzpicture} 
\draw (2,2)--(2,4)--(4,4)--(4,2)--(2,2);
\draw (2,2) --(4,1)--(4,-1)--(2,0)--(2,2);
\draw [blue, ultra thick] (2,2)--(4,1);
\fill (3.0,1.7)--(3.0,1.3) -- (3.2,1.4) --cycle;
\fill (3.0,-0.7)--(3.0,-0.3) -- (3.2,-0.6) --cycle;
\fill (4.2,-0.1)--(3.8,0.1) -- (4.0,0.2) --cycle;

\draw (3,0.3) node {$S'$};
\draw (0,0) -- (0,2) -- (2,2) --(2, 0)--(0,0);
\fill (1.0,2.2)--(1.0,1.8) -- (0.8,2.0) --cycle;

\fill (1.0,-0.2)--(1.0,0.2) -- (0.8,0.0) --cycle;
\fill (-0.2,1.0)--(0.2,1.0) -- (0.0,1.2) --cycle;

\fill (2.2,1.0)--(1.8,1.0) -- (2.0,1.2) --cycle;
\draw (3,3) node {$\tilde{S}_{n}$};

\draw (1,1) node {$\tilde{S}_{n-1}$};
\draw (2,0) node[below left] {$v_{n-1}$};
\draw (2,2) node[above left] {$v_n$};

\draw (0,2) node[above left] {$[f_2]$};
\dtypeI{(0,0)}
\dtypeI{(2,0)}
\dtypeI{(2,4)}
\dtypeI{(4,4)}
\dtypeI{(0,2)}
\dtypeI{(2,2)}
\dtypeI{(4,2)}
\dtypeII{(3,2)}
\dtypeII{(3,4)}
\dtypeII{(2,3)}
\dtypeII{(4,3)}
\dtypeII{(3,4)}
\dtypeI{(4,1)}
\dtypeI{(4,-1)}
\end{tikzpicture}
\end{center}
Observe that the same argument applies for $\nu'$ and we can find another square $S''$ adjacent to $\tilde{S}_{n-1}$ along $v_n, v_{n-1}$ such that $S'', \tilde{S}_{n-1}$ is contained in a vertical spiral staircase and such that the horizontal edge in $S''$ containing $v_n$ is not critically resonant for $\nu'$. 
By Theorem \ref{thm_bound_vert_I}, the vertex 
 $v_n$ is $\nu$-maximal and $\nu'$-maximal in $\tilde{S}_n$ and 
 $\nu(v_{n+1}) <(4/3) \nu(v_n)$, proving $(A),(B)$ and $(C)$ as required.

 We have thus proven that our induction step is valid, and the theorem is proved. 
 
 \begin{proof}[Proof of Lemma \ref{lem_adjacent_standard_square} ]
 Fix a valuation $\nu \in \mathcal{V}_0$ satisfying \eqref{eq_spe_valuation} and take a square $S$ adjacent to $S_0$ along an edge containing $[t]$. 
 
 Observe that the edge $S\cap S_0$ is either vertical or horizontal. Since the proof is similar for both cases, we can suppose that $S\cap S_0$ is vertical so that $S$ and $S_0$ intersect along the edge containing $[y],[t]$.  
Remark that in this case, we have $v_1= [y]$ and $v_2 $ is the vertex distinct from $[t]$ which belongs to the horizontal edge in $S$ containing $[t]$.
 
  We are thus in the following situation:
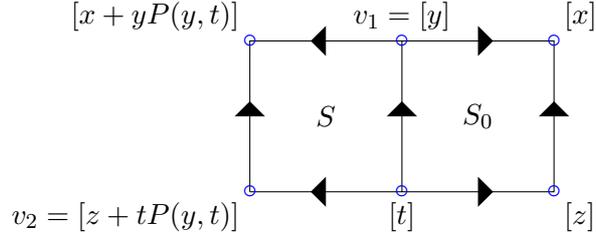
\begin{center}
\begin{tikzpicture} 
\draw (0,0) -- (4,0)--(4,2) -- (2,2) -- (2,0);
\draw (0,0) -- (0,2) -- (2,2);
\fill (3.0,-0.2)--(3.0,0.2) -- (3.2,0.0) --cycle;
\fill (3.0,2.2)--(3.0,1.8) -- (3.2,2.0) --cycle;
\fill (1.0,2.2)--(1.0,1.8) -- (0.8,2.0) --cycle;

\fill (1.0,-0.2)--(1.0,0.2) -- (0.8,0.0) --cycle;
\fill (-0.2,1.0)--(0.2,1.0) -- (0.0,1.2) --cycle;

\fill (2.2,1.0)--(1.8,1.0) -- (2.0,1.2) --cycle;
\fill (4.2,1.0)--(3.8,1.0) -- (4.0,1.2) --cycle;

\draw (1,1) node {$S$};
\draw (3,1) node {$S_0$};
\draw (0,0) node[below left] {$v_2=[z + tP(y,t)]$};
\draw (4,2) node[above right] {$[x]$};
\draw (4,0) node[below right] {$[z]$};
\draw (2,0) node[below] {$[t]$};
\draw (2,2) node[above] {$v_1=[y]$};

\draw (0,2) node[above left] {$[x + yP(y,t)]$};
\dtypeI{(0,0)}
\dtypeI{(2,0)}
\dtypeI{(4,0)}

\dtypeI{(0,2)}
\dtypeI{(2,2)}
\dtypeI{(4,2)}
\end{tikzpicture}
\end{center}
where $P \in \C[x,y] \setminus \C$. 

 Observe also that the edge $S\cap S_0$ is not critically resonant.

Since $\nu(P(y,t)) \leqslant  \min(\nu(y) , \nu(t))$ and since \eqref{eq_spe_valuation} implies that $2\nu(t) < \nu(z)$ and $\nu(y)+ \nu(t) < \nu(z)$, we get:
\begin{equation*}
\nu(tP(y,t)) < \nu(z),
\end{equation*} 
hence $\nu(z+ yP(y,t)) < \nu(z)$ and the vertex $[z+ tP(y,t)]$ is $\nu$-maximal in $S$.
%
%
Observe also that the component $(y,t)$ is not critically resonant. 
By Corollary \ref{cor_parachute}, we obtain:
\begin{equation*}
\nu(z + tP(y,t)) < \nu(y),
\end{equation*}
hence $\nu(v_2) < \nu(v_1)$, as required.

 \end{proof}

\subsection{Proof of Theorem \ref{thm_int_degree_growth}}
\label{subsection_proof_thm_degree_growth}

Consider a tame automorphism $f \in \tame$. 
Since the complex $\mathcal{C}$ is $\CAT(0)$ and since the action of $f$ is an isometry and a morphism of complex, the action of $f$ on the complex either fixes a vertex or a geodesic line.
In the first case, $f$ is elliptic and by Theorem \ref{thm_growth_elliptic}, the sequences $(\deg(f^n))$, $(\deg(f^{-n}))$ are either both bounded, both linear or both equivalent to $C d^n$ where $C>0$ and $d \in \mathbb{N}$. 
\smallskip

We are thus reduced to prove the theorem in the case where $f$ induces an action which fixes a geodesic line $\gamma : \mathbb{R} \to \mathcal{C}$. Take an hyperbolic automorphism $f$ and a geodesic line $\gamma: \mathbb{R} \to \mathcal{C} $ fixed by $f$.  
Denote by $S_0$ the standard $2\times 2$ square containing $[x],[y],[z]$ and $[t]$.
Since for any tame automorphism $h \in \tame$, there exists a constant $C>0$ such that:
\begin{equation*}
\dfrac{1}{C} \leqslant \dfrac{\deg(f^n)}{\deg(h^{-1}f^nh)} \leqslant C,
\end{equation*}
by taking an appropriate conjugate of $f$,  we can suppose that $\gamma$ starts in  $S_0$ and intersects an edge of $S_0$. 
Consider the geodesic segment $\gamma'_n$ joining $[\Id]$ and $[x \circ f^{-n}]$. By construction, $\gamma'_n$ intersects an edge of the standard square $S_0$ as $\gamma$ starts in $S_0$.
Fix any valuation $\nu$ such that \eqref{eq_spe_valuation} is satisfied.
There are infinitely many valuations in $\mathcal{V}_0$ satisfying \eqref{eq_spe_valuation} arbitrarily close to $-\deg$. 
Indeed, consider the sequence of weight $\alpha_i = ( -1 , -1 + 3/i, -1 + 5/i, -1 +7/i)$, then by Proposition \ref{prop_minimal_valuation}, there exists a sequence of valuations $\nu_i$ with weight $\alpha_i$ on $(x,y,z,t)$ which converges to $-\deg$.

 All assumptions of Theorem \ref{thm_degree_versus_distance} are then satisfied and we get:
\begin{equation*}
\nu_i(f^n \cdot [x]) = \nu_i(x \circ f^{-n}) \leqslant \left ( \dfrac{4}{3} \right )^{d_{\nu_i}([t], [x\circ f^{-n}]) - 1} \max(\nu_i(y), \nu_i(z), \nu_i(x), \nu_i(t) ).
\end{equation*} 
Observe that $\nu_i$ tends to $-\deg$, moreover, assertion $(2)$ of Theorem \ref{thm_degree_versus_distance} implies that the distance $d_{\nu_i}( [t] , [x \circ f^{-n}])$ are all equal for all $i$ which implies:
\begin{equation*}
\deg(f^{-n}) \geqslant  \left ( \dfrac{4}{3} \right )^{d_\nu([t], [x\circ f^{-n}]) - 1},
\end{equation*} 
for a given valuation $\nu$ satisfying \eqref{eq_spe_valuation}.
\medskip

We now prove that the sequence $(d_{\nu} ([t], [x \circ f^{-n}]))_n$ grows at least linearly. 
Indeed since the invariant geodesic $\gamma$ passes through $S_0$, then it passes through all the squares $f^i \cdot S_0$ for all $i\leqslant n$. Observe that all the squares $f^i \cdot S_0$ are distinct and there are at least $n$ squares. 
Consider a geodesic segment $\gamma_{1n} $ in $\mathcal{C}_\nu$ joining $[t]$ and $[x \circ f^{-n}]$ and a shortest path $\gamma_{2n}$ in $\mathcal{C}_\nu$ contained in a sequence of squares containing the geodesic $\gamma$ between these two vertices.
The hyperbolicity of $\mathcal{C}$ implies that the lengths $l(\gamma_{1n}), l(\gamma_{2n})$  in $\mathcal{C}_\nu$ of $\gamma_1$ and $\gamma_2$ are comparable as $n$ tends to infinity:
\begin{equation*}
\liminf_{n \rightarrow +\infty} \dfrac{l(\gamma_{1n})}{l(\gamma_{2n})} = 1.
\end{equation*}
Since the length in $\mathcal{C}_{\nu}$ of $\gamma_{2n}$ is larger or equal than $n$, we have proven that:
\begin{equation*}
\lim_{n \rightarrow +\infty}\dfrac{1}{n} d_\nu([t], [x \circ f^{-n}]) \geqslant  1.
\end{equation*} 
Hence
\begin{equation*}
\deg(f^{-n}) \geqslant C \left ( \dfrac{4}{3} \right )^{n-1},
\end{equation*}  
where $C>0$.
Since the argument is similar for $\deg(f^n)$, we have thus proven that:
\begin{equation*}
\min(\deg(f^n),\deg(f^{-n})) \geqslant C \left ( \dfrac{4}{3} \right )^n
\end{equation*}
where $C>0$.

\black 

\black

\subsection{Proof of Theorem \ref{thm_int_degree_versus_distance}}
\label{subsection_proof_thm_degree_vs_distance}
Take $f,g \in \tame$. 
Since the tame group acts by isometries on the complex, we can suppose that $g = \Id$. 
Consider $\gamma$ the geodesic in $\mathcal{C}$ joining $[\Id]$ to $[ x\circ f]$. 
Since the stabilizer of $[\Id]$ is the group $\O4$ by Proposition \ref{prop_action_tame_complex} and since the group $\O4$ acts transitively on the $1\times 1$ squares containing $[\Id]$ by Proposition \ref{prop_correspondence_vertex_III}, we can suppose that the geodesic $\gamma$ intersects an edge of type I containing $[x]$ of the $1\times 1$ square containing $[x]$, $[\Id]$, $[z,x]$ and $[x,y]$. 
In particular, the geodesic $\gamma$ intersects an edge of the standard square $S_0$.
We have proved that the vertex $v = [x \circ f] $ satisfies the conditions of Theorem \ref{thm_degree_versus_distance}, 
and by considering a sequence of valuations $\nu_p \in \mathcal{V}_0$ converging to $-\deg$ satisfying \eqref{eq_spe_valuation},
we have:
\begin{equation*}
\nu_p(x \circ f) \leqslant \left (\dfrac{4}{3} \right )^{d_{\nu_p}([t],[x \circ f] ) -1} \max(\nu_p(y),\nu_p(z),\nu_p(x), \nu_p(t)).
\end{equation*}
By Proposition \ref{prop_length_vs_distance}, we have for all integer $p$:
\begin{equation*}
\dfrac{1}{2\sqrt{2}} d_{\mathcal{C}} (v_1,v_2) \leqslant d_{\nu_p}(v_1,v_2).
\end{equation*}
for any vertices $v_1,v_2$ of type I.
Since $d_\mathcal{C}([t], [x \circ f]) \geqslant d_{\mathcal{C}}([\Id], [f]) - 2\sqrt{2}$, we thus obtain after taking the limit as $p\rightarrow + \infty$:
\begin{equation*}
 \log \deg(f ) \geqslant C d_\mathcal{C}([f],[\Id])  - C' ,
\end{equation*}
where $C' = 2\log(4/3)  $ and  $C = \log(4/3)/(2\sqrt{2})$ so that:
\begin{equation*}
\log \deg(f^{-1} \circ g) \geqslant \dfrac{\log(4/3)}{2\sqrt{2}} d_\mathcal{C}( f\cdot [\Id], g \cdot [\Id]) - 2 \log(4/3), 
\end{equation*}
as required.
%

%

\section{Application to random walks on the tame group}

In this section, we consider a random walk on the tame group and its associated degree sequence. After recalling some general facts on random walks on groups (\S\ref{gen_random}), we then discuss when the degree exponents of a random walk are well-defined and  their properties (\S\ref{subsection_lyapounov}). 
 We then classify in \S \ref{section_elem_random} the finitely generated subgroup of $\tame$.
Finally we prove Theorem \ref{thm_random_walk}, which asserts that the degree exponent of a symmetric random walk on a finitely generated group $G$ is strictly positive if and only if it contains two non-commuting automorphisms with dynamical degree strictly larger than $1$ generating a free group of rank $2$.

\subsection{General facts on random walks on groups}
\label{gen_random}

Let $G$ be a finitely generated subgroup of the tame group and let $\mu$ be an atomic probability measure on $G$.  
The (left) random walk on $G$ with respect to the measure $\mu$ is the Markov chain whose initial distribution is the Dirac mass at $\Id$ with transition matrix $p(g,g') =\mu(\{  g'g^{-1}\}) $ for all $g,g'\in G$. 
We denote by $\Omega = (G^{\mathbb{N}^*}, \mu^{\otimes {\mathbb{N}^*}})$ the product probability space which encodes the successive increments of the random walk on $G$ with respect to the measure $\mu$.  
Consider an element $s=(s_1, \ldots, s_n , \ldots ) \in \Omega $, set $g_0(s) = \Id$ and 
\begin{equation*}
g_n(s) = s_n s_{n-1} \ldots s_1,
\end{equation*}
for all $n \geqslant 1$. 
The image $\mathcal{P}$ of the map $s \in \Omega \mapsto (\Id, g_1(s) , \ldots , g_n(s) , \ldots )\in G^{\mathbb{N}*}$ is called the path space and an element of $\mathcal{P}$ is a path in the group $G$. 
We naturally endow $\mathcal{P}$ with the probability measure $\Pg$  defined on the $\sigma$-algebra of cylinders as the pushforward of the product measure on $\Omega$ by the map $s \in \Omega \mapsto (g_i(s))_{i}$.
More explicitly, consider the probability measure $\nu_n$ of the projection of $\mathcal{P}$ onto the $n+1$-th component $g_n$, then  $\nu_n$ is equal to the $n$-fold convolution $\mu^{*n} * \delta_{\Id}$ so that for all $g\in G$, one has: 
\begin{equation*}
\nu_n(\{ g\}) =\Pg(g_n = g) = \sum_{ \substack{s_1, \ldots, s_n\\ 
s_n \ldots s_1 = g} } \prod_{i=1}^n \mu(s_i).
\end{equation*} 
Fix a reference vertex $v_0 = [\Id]$ in the complex $\mathcal{C}$. 
Since the tame group acts on the complex, a path in the group $(\Id, g_1, \ldots, g_n , \ldots) $ induces an element in $\mathcal{C}^{\mathbb{N}*}$ given by $(v_0, g_1 \cdot v_0, \ldots , g_n\cdot v_0, \ldots )$. The sequence $(v_0, g_1 \cdot v_0, \ldots, g_n \cdot v_0, \ldots)$ is called a path in the complex.

\subsection{Degree exponents of a random walk}
\label{subsection_lyapounov}

Let $G$ be a finitely generated subgroup of the tame group and let $\mu$ be an atomic probability measure on $G$. We shall define in this section the degree exponents of a random walk with respect to the measure $\mu$.  
To do so, the measure $\mu$ must satisfy a finiteness condition on its first moment:
\begin{equation} \label{eq_condition_moment}
\int_{g \in G} \log(\deg(g)) d\mu(g) < +\infty.
\end{equation} 
Let us define the two degree exponents $\lambda_1(\mu), \lambda_2(\mu)$ by:
\begin{equation*}
\lambda_1(\mu) :=  \limsup_{n\rightarrow +\infty}\dfrac{1}{n}\int_{g\in G} \log(\deg(g)) d\nu_n(g),
\end{equation*} 
and 
\begin{equation*}
\lambda_2(\mu) :=  \limsup_{n\rightarrow +\infty}\dfrac{1}{n}\int_{g\in G} \log(\deg(g^{-1})) d\nu_n(g),
\end{equation*}
where $\nu_n$ is the probability measure of $g_n$. 

\smallskip

The following proposition proves that these quantities are finite and give a few basic properties of these numbers.
\begin{prop} \label{prop_elem_lyapounov}Take $G$ a countably generated subgroup of the tame group and $\mu$ an atomic probability measure on $G$ satisfying condition \eqref{eq_condition_moment}. Then the following properties are satisfied.
\begin{enumerate}
\item[(i)] The  degree exponents $\lambda_1(\mu), \lambda_2(\mu)$ are finite and are equal to:
\begin{equation*}
\lambda_1(\mu) = \lim_{n\rightarrow + \infty} \dfrac{1}{n} \int_{g \in G} \log(\deg(g)) d\mu_n(g),
\end{equation*}
and 
\begin{equation*}
\lambda_2(\mu) = \lim_{n\rightarrow + \infty} \dfrac{1}{n} \int_{g \in G} \log(\deg(g^{-1})) d\mu_n(g),
\end{equation*}
\item[(ii)] The following inequality holds:
\begin{equation*}
\lambda_1(\mu) \geqslant \dfrac{\lambda_2(\mu)}{2}. 
\end{equation*}
\item[(iii)] Consider $\sigma : G\to G$ the inverse map, then $\lambda_2(\mu) = \lambda_1(\sigma_* \mu)$.
\item[(iv)]  The degree exponents are invariant by conjugation, i.e for any $h \in \tame$, we have: 
\begin{equation*}
\lambda_i({\Conj(h)}_* \mu) = \lambda_i(\mu),
\end{equation*}
where $\Conj(h):   \tame \to  \tame$ denotes the conjugation by $h$ in $G$.
\end{enumerate} 

\end{prop}

\begin{proof} 
Let us first prove $(i)$. 
Using the fact that  $\deg(g)^2 \geqslant \deg(g^{-1})$ for all $g \in G$, we obtain a finiteness condition on the inverse:
\begin{equation*}
\int_{g \in G} \log(\deg(g^{-1})) d\mu(g) \leqslant 2 \int_{g\in G} \log(\deg(g)) d\mu(g) < +\infty
\end{equation*}
Since the function $\deg$ is submultiplicative, according to Kingman's subadditivity theorem applied to the functions $\log(\deg(g)) $ and $\log(\deg(g^{-1})$, the degree exponents are finite and the $\limsup$ in the definition is a limit. This proves assertion $(i)$. 

\smallskip

Assertion $(ii)$ follows from the fact that $\deg(g)^2 \geqslant \deg(g^{-1})$ for all $g\in \tame$. To prove assertion $(iii)$, observe that for all $(s_1, \ldots, s_n ,\ldots ) \in \Omega$, we have:
\begin{equation*}
s_n^{-1} s_{n-1}^{-1} \ldots s_1^{-1} = (s_1 s_2 \ldots s_n)^{-1}.
\end{equation*}
In particular, we obtain:
\begin{equation*}
 \lim_{n \rightarrow +\infty }\dfrac{1}{n}\int_{\Omega} \log(\deg( (s_1 s_2 \ldots s_n)^{-1}))) d\mu^{\otimes n} =  \lim_{n \rightarrow +\infty }\dfrac{1}{n}\int_{\Omega} \log(\deg(s_n s_{n-1} \ldots s_1)) d\sigma_*\mu^{\otimes n}.    
\end{equation*}
Since the right hand side of the equality is equal to $\lambda_1(\sigma_*\mu)$ and the left hand side to $\lambda_2(\mu)$, we have thus proven $(iii)$. 

\smallskip

Finally, let us prove assertion $(iv)$. Fix $h\in\tame$, recall that  there exists a constant $C(h)>0$ such that for all $g \in \tame$, we have:
\begin{equation*}
\dfrac{\deg(g)}{C(h)} \leqslant \deg(hgh^{-1}) \leqslant C(h) \deg(g).
\end{equation*} 
The last inequality directly implies that $\lambda_i(\Conj(h)_* \mu) = \lambda_i(\mu)$ for all $i=1,2$ and all $h \in \tame$.
\end{proof}

\subsection{Classification of finitely generated subgroups}
\label{section_elem_random}
In this section, we give a classification of the finitely generated subgroups of the tame group. To that end, we recall the terminology due to Gromov (\cite{gromov}) on subgroups of isometries of a hyperbolic space.


Fix a Gromov hyperbolic space $X$ and a group $G$ acting on it by isometry.
 The action  of $G$ on $X$ is called \textbf{elementary} if it does not contain two hyperbolic isometries whose action do not fix the same geodesic line. 
 We call the action of $G$ on $X$  \textbf{elliptic} if it globally fixes a point in $X$ and we shall say that the action of $G$ is \textbf{lineal} if there exists an elliptic subgroup $H$ of $G$, a geodesic line $\gamma$ on $X$ invariant by $G$, pointwised fixed by $H$ on which the quotient $G/H$ acts  faithfully by translation. 
\smallskip

In our case, any element of the tame group induces an isometry of the complex. 
We will also need to distinguish among the subgroups which fix a vertex in the complex, more particularly when the fixed vertex is of type I. 
Remark that a subgroup $G$ of the tame group which fixes a vertex  of type I is conjugated to a subgroup of $\stab([x])$ and recall that  we have constructed in Subsection \ref{subsection_link_type_I} a natural action from the stabilizer subgroup $\stab([x])$ on a subtree of the Bass-Serre tree.
We have have the following classification.

\begin{thm} \label{thm_classification_gen} Let $G$ be a finitely generated subgroup of the tame group. Then one of the following situation occurs.
\begin{enumerate}
\item[(i)] The action of $G$ on the complex is non-elementary. 
\item[(ii)] There exists an automorphism $h $ in $G$ whose action in the complex is hyperbolic and such that any automorphism $f \in G$ can be decomposed into $f = g \circ h^p$ where $p$ is an integer and where $g$ belongs to a subgroup $H$.
Moreover, the subgroup $H$ is conjugated in $\tame$ to a subgroup of $\O4$ or to one of
$$\EH \rtimes \left \{ \left ( \begin{array}{ll}
ax & b y \\
b^{-1} z & a^{-1} t
\end{array} \right ) \ | \ a,b \in \C^* \right \}.$$
\item[(iii)] The group $G$ is conjugated to a subgroup of the linear group $\O4$.
\item[(iv)] There exists a $G$-equivariant morphism $\varphi: Q \to \mathbb{A}^2 \setminus \{(0,0)\}$ where $G$ acts on $\mathbb{A}^2 \setminus \{(0,0) \}$ linearly.
\item[(v)]  The group $G$ contains two non-commuting automorphisms with dynamical degree larger or equal that $2$ and there exists a $G$-equivariant morphism $\varphi: Q \to \mathbb{A}^1$ where $G$ acts on $\mathbb{A}^1$ by multiplication. 
\item[(vi)] The group $G$ contains an automorphism $h$ with $\lambda_1(h)\geqslant 2$ and there exists a $G$-equivariant morphism $\varphi: Q \to \mathbb{A}^1$ on which $G$ acts on $\mathbb{A}^1$ by multiplication and an isomorphism $\varphi^{-1}(\mathbb{A}^1 \setminus \{0 \}) \simeq \mathbb{A}^1 \setminus \{ 0\} \times \mathbb{A}^2 $ such that  any automorphisms $f \in G$ can be decomposed into $ g \circ h^p$ where $p$ is an integer and $g$ is of the form
\begin{equation*}
g: (x, y,z) \in \mathbb{A}^1 \setminus \{ 0\} \times \mathbb{A}^2 \mapsto ( ax , b y , cz  ) \in \mathbb{A}^1 \setminus \{ 0\} \times \mathbb{A}^2, 
\end{equation*}
where $a,b,c \in \C^*$.
\item[(vii)]  There exists a $G$-equivariant morphism $\varphi: Q \to \mathbb{A}^1$ where $G$ acts on $\mathbb{A}^1$ by multiplication and any automorphism of $G$ has dynamical degree $1$.
\end{enumerate}
\end{thm}

\begin{proof}

 Let us give a tree summarizing how our proof proceeds where each end of the tree corresponds to a conclusion of the previous theorem. Denote by $\mathcal{T}$ the associated Bass-Serre tree constructed in Subsection \ref{subsection_bass_serre}.

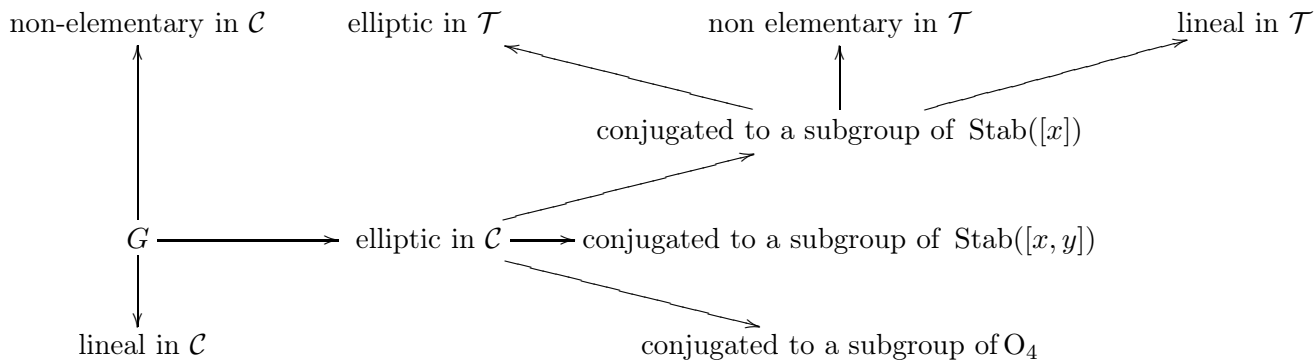
\[\xymatrix{ \text{non-elementary in }\mathcal{C}& \text{elliptic in }\mathcal{T}  &   \text{non elementary in }\mathcal{T}& \text{lineal in } \mathcal{T} \\
 & &     \text{conjugated to a subgroup of }\stab([x]) \ar[lu]\ar[u] \ar[ru] & \\
  G \ar[uu] \ar[r] \ar[d] & \text{ elliptic in }\mathcal{C} \ar[ru] \ar[r] \ar[rd] & \text{conjugated to  a subgroup of }\stab([x,y]) & \\ 
        \text{ lineal in }\mathcal{C} & & \text{conjugated to a subgroup of} \O4 & \\
}
\]

\begin{thm} \label{thm_elem_subgroup_tame} Let $G$ be a finitely generated subgroup of the tame group whose action on the complex $\mathcal{C}$ is elementary.
The following possibilities occur.
\begin{enumerate}
\item[(i)] The action of  $G$ on the complex is elliptic, i.e $G$ fixes globally a vertex in the complex.
\item[(ii)] The action of $G$ is lineal on the complex, i.e there exists an elliptic subgroup $H$ of $G$, a geodesic line $\gamma$ on $\mathcal{C}$ invariant by $G$ pointwised fixed by $H$ on which the quotient $G/H$ acts  faithfully by translation. 
Moreover, the subgroup $H$ is conjugated in $\tame$ to a subgroup of $\O4$ or to one of
$$\EH \rtimes \left \{ \left ( \begin{array}{ll}
ax & b y \\
b^{-1} z & a^{-1} t
\end{array} \right ) \ | \ a,b \in \C^* \right \}.$$
\end{enumerate}
\end{thm}

Assume that the above theorem holds, we prove that our classification holds. 
The situation $(ii)$ in Theorem \ref{thm_elem_subgroup_tame} implies that there exists an hyperbolic automorphism $h \in G$ such that any $f\in G$ can be decomposed into $f = g \circ h^p$ where $p $ is an integer and $g \in H$. 
This falls into situation $(ii)$ of the theorem.
\smallskip

Let us extend the case $(i)$ of Theorem \ref{thm_elem_subgroup_tame}. 
Take a group $G$ whose action fixes a vertex of type III, then it is naturally conjugated to a subgroup of $\O4$ by Proposition \ref{prop_stab_III} and assertion $(iii)$ holds.
If $G$ fixes a vertex of type II in the complex, then by Proposition \ref{prop_stab_II}, $G$ satisfies assertion $(iv)$ of the Theorem.
\smallskip

Suppose now that $G$ fixes a vertex of type I then $G$ is conjugated to a subgroup of $\stab([x])$ by transitivity of the action on the vertices of type I (Proposition \ref{prop_action_tame_complex}.($ii$)). 
In this situation, recall that  we have constructed in Subsection \ref{subsection_link_type_I} a natural action from the stabilizer subgroup $\stab([x])$ on a subtree of the Bass-Serre tree, in particular, there exists a $G$-equivariant morphism $\varphi : Q \to \mathbb{A}^1$ where the action of $G$ on $\mathbb{A}^1$ is multiplicative. 
In the case where the group $G$ fixes a vertex of type I, its action on the corresponding subtree of the Bass-Serre tree is either non-elementary or elementary. 
If the action of $G$ is non-elementary on the corresponding Bass-Serre tree, then equivalently $G$ contains two non-commuting morphisms with dynamical degree larger or equal to two and $G$ satisfies assertion $(v)$ in our classification. 

Suppose that the action of  $G$ on the corresponding Bass-Serre tree is elementary. 
Let us fix an isomorphism $\varphi^{-1}(\mathbb{A}^1\setminus \{ 0\} )  \simeq \mathbb{A}^1 \setminus \{ 0\} \times \mathbb{A}^2$. 
If $G$ contains only elliptic elements on the Bass-Serre tree, then by \cite[Proposition 3.11]{lamy_tits_C2} $G$ is conjugated to a subgroup whose action of the generic fiber of $\varphi$ is affine or elementary and assertion $(vii)$ is satisfied.
Otherwise, the action of $G$ on the Bass-Serre tree is lineal. 
In particular, there exists an automorphism $h \in G$ whose action on the Bass-Serre tree is hyperbolic such that any $f\in G$ can be decomposed into $f = g \circ h^p$ where $p$ is an integer and where $g$ is an automorphism whose action on the Bass-Serre tree is elliptic and fixes pointwise the geodesic line on the Bass-Serre tree fixed by $h$.
By \cite[Proposition 3.3]{lamy_tits_C2}, $g$ must be of the form:
\begin{equation*}
(x,y,z) \in \mathbb{A}^1 \setminus \{ 0\} \times \mathbb{A}^2 \mapsto (ax, by ,cz) \in\mathbb{A}^1 \setminus \{ 0\} \times \mathbb{A}^2,
\end{equation*}
where $a , b, c \in \C^*$. 
We have thus proved that assertion $(vi)$ holds.
\end{proof}

\begin{proof}[Proof of Theorem \ref{thm_elem_subgroup_tame}]
Take $G$ a finitely generated subgroup of the tame group. 
By \cite[Theorem 2]{ballmann_swiatkowski}, the subgroup $G'$ must satisfy one of the following three cases.
\begin{enumerate}
\item[(a)] The group $G'$ is elliptic.
\item[(b)] There exists an integer $2 \geqslant k\geqslant 1$, an elliptic subgroup $H \subset G$ and a subspace $E\subset \mathcal{C}$ which is isometric to a $k$-dimensional euclidean space, is pointwise fixed by $H$ and on which the group $G/H$ acts as a cocompact lattice of translation.
\item[(c)]Every automorphism of $G$ is elliptic and there exists a geodesic half-line  and a sequence of vertices $v_n$ on this half-line for which the subgroups $G_n = G\cap \stab(v_n)$ form an increasing filtration which satisfy $G= \cup G_n$.  
\end{enumerate}

Since the complex $\mathcal{C}$ is Gromov hyperbolic, it cannot contain any euclidean plane. 
As a result, the case $k=2$ in $(b)$ is excluded. The remaining possibility is when $k=1$ and there exists a geodesic line $E$ globally invariant by $G$ in the complex, a subgroup $H$ of $G$ fixing pointwisely $E$ such that $G/H$ acts faithfully transitive by translation on $E$. 
Remark also that $(c)$ cannot hold since the group $G$ cannot contain infinitely many subgroups since it is finitely generated.

\smallskip 

To prove that $(ii)$ holds amounts in proving that in case $(b)$ the elliptic subgroup $H$ is conjugated to a subgroup of $\O4$ or to a subgroup of 
$$\EH \rtimes \left \{ \left ( \begin{array}{ll}
ax & b y \\
b^{-1} z & a^{-1} t
\end{array} \right ) \ | \ a,b \in \C^* \right \}.$$
Since $H$ fixes a pointwisely geodesic line $E$, we can choose a sequence $v_n$ of distinct vertices near $E$ all fixed by $H$ lying on a quasi-geodesic line.
Consider $\gamma_n$ a geodesic path in the $1$-skeleton of $\mathcal{C}$ joining $v_n$ and $v_{n+1}$. Since the group $G$ fixes the type of vertices and the endpoint of $\gamma_n$, the geodesic $\gamma_n$ must be fixed pointwise.
If one of the geodesic $\gamma_n$ contains a vertex of type III, then $G$ is conjugated to a subgroup of $\O4$ and statement $(ii)$ is proved.
\smallskip
 
Assume now that the geodesics $\gamma_n$ contain only type I and II vertices. We prove that $(ii)$ also holds.
For simplicity, we can assume that $v_0, v_1 ,v_2$ pointwise fixed by $G$ are consecutive vertices on a geodesic line of the $1$-skeleton of $\mathcal{C}$.
Assume also that $v_0,v_2$ are of type I and that $v_1$ is of type II. 
Conjugating with an element of $\tame$, we can assume that $v_0 =[z], v_1= [x,z]$ and $v_2=[x]$. 
Since $G$ fixes $[x], [x,z]$ and $[z]$, it implies by Proposition \ref{prop_stab_II}.$(iii)$ that $G$ is conjugated to a subgroup of 
$$\EH \rtimes \left \{ \left ( \begin{array}{ll}
ax & b y \\
b^{-1} z & a^{-1} t
\end{array} \right ) \ | \ a,b \in \C^* \right \},$$
proving $(ii)$ as required.
\end{proof}

\subsection{Proof of Theorem \ref{thm_random_walk} and Corollary \ref{cor_int_random_cor}}

Take $G$ a finitely generated subgroup of the tame group and take  $\mu$ a symmetric atomic measure on $G$ whose support generates $G$ and such that:
\begin{equation*}
\int_G \log( \deg(g)) d\mu(g) < +\infty.
\end{equation*}

We denote by  $g_n$ the state of our random walk at the time $n$. 
Observe that since $\mu$ is symmetric, Proposition \ref{prop_elem_lyapounov}.$(iii)$ implies that $\lambda_2(\mu) = \lambda_1(\mu)$. 
\smallskip 

Let us explain how we proceed to prove our result.
By Theorem \ref{thm_classification_gen}, the group $G$ satisfies one of the following conditions.
\begin{enumerate}

\item[(i)] The action of $G$ on the complex is non-elementary in $\mathcal{C}$. 
\item[(ii)] There exists an automorphism $h $ in $G$ whose action in the complex is hyperbolic and such that any automorphism $f \in G$ can be decomposed into $f = g \circ h^p$ where $p$ is an integer and where $g$ belongs to a subgroup $H$.
Moreover, the subgroup $H$ is conjugated in $\tame$ to a subgroup of $\O4$ or to one of
$$\EH \rtimes \left \{ \left ( \begin{array}{ll}
ax & b y \\
b^{-1} z & a^{-1} t
\end{array} \right ) \ | \ a,b \in \C^* \right \}.$$
\item[(iii)] The group $G$ is conjugated to a subgroup of the linear group $\O4$.
\item[(iv)] There exists a $G$-equivariant morphism $\varphi: Q \to \mathbb{A}^2 \setminus \{(0,0)\}$ where $G$ acts on $\mathbb{A}^2 \setminus \{(0,0) \}$ linearly.
\item[(v)]  The group $G$ contains two non-commuting automorphisms with dynamical degree larger or equal to $2$ and there exists a $G$-equivariant morphism $\varphi: Q \to \mathbb{A}^1$ where $G$ acts on $\mathbb{A}^1$ by multiplication. 
\item[(vi)] The group $G$ contains an automorphism $h$ with $\lambda_1(h)\geqslant 2$ and there exists a $G$-equivariant morphism $\varphi: Q \to \mathbb{A}^1$ on which $G$ acts on $\mathbb{A}^1$ by multiplication and an isomorphism $\varphi^{-1}(\mathbb{A}^1 \setminus \{0 \}) \simeq \mathbb{A}^1 \setminus \{ 0\} \times \mathbb{A}^2 $ such that  any automorphisms $f \in G$ can be decomposed into $ g \circ h^p$ where $p$ is an integer and $g$ is of the form
\begin{equation*}
g: (x, y,z) \in \mathbb{A}^1 \setminus \{ 0\} \times \mathbb{A}^2 \mapsto ( ax , b y , cz  ) \in \mathbb{A}^1 \setminus \{ 0\} \times \mathbb{A}^2, 
\end{equation*}
where $a,b,c \in \C^*$.
\item[(vii)]  There exists a $G$-equivariant morphism $\varphi: Q \to \mathbb{A}^1$ where $G$ acts on $\mathbb{A}^1$ by multiplication and any automorphism of $G$ has dynamical degree $1$.
\end{enumerate}

Denote by $\lambda = \lambda_1(\mu) = \lambda_2(\mu)$.
We shall prove successively the following implications: $(i) \Rightarrow (\lambda >0)$, $(v) \Rightarrow (\lambda>0)$, 
$((ii) \text{ or } (iv) \text{ or } (vi) ) \Rightarrow (\lambda =0)$, $((iii) \text{ or } (vii) ) \Rightarrow (\lambda=0)$. 
If all the above implications hold, then both Theorem \ref{thm_random_walk} and Corollary \ref{cor_int_random_cor} hold.
\bigskip

The essential ingredient to compute the degree exponents in situation $(i)$ and $(v)$ is the following result.
Suppose that $G$ acts on a Gromov-hyperbolic space $(X,d)$ and fix a reference vertex $x_0$ in $X$, then a random path $(\Id, g_1, \ldots, g_n , \ldots)$ in the group $G$ induces a random path in $X$ given by $(x_0, g_1 \cdot x_0, \ldots , g_n \cdot x_0, \ldots )$. 
The following theorem is due to Maher-Tiozzo (\cite[Theorem 1.2]{maher_tiozzo}).

\begin{thm} \label{thm_maher_tiozzo} Let $G$ be a non-elementary countable subgroup of the tame group and let $\mu$ be an atomic measure on $G$ whose support generates $G$ and such that the integral
\begin{equation*}
\int_G  d (g\cdot v_0, v_0) d\mu(g)
\end{equation*} 
is finite. Then there exists a constant $L>0$ such that for almost every sample path in the group, one has:
\begin{equation*}
\lim_{n\rightarrow +\infty} \dfrac{ d(g_n \cdot v_0, v_0)}{n} = L.
\end{equation*}
\end{thm}
We will apply this result for $X = \mathcal{C}$ and $X= \mathcal{T}$ in situation $(i)$ and $(v)$ respectively.

Let us prove the implication $(i) \Rightarrow (\lambda>0)$.  Suppose that $G$ is non-elementary in $\mathcal{C}$.
By Theorem \ref{thm_int_degree_versus_distance}, there exists $C,C'>0$ such that for any $g \in G$:
\begin{equation}\label{eq_deg_vs_dist}
\log( \deg(g) ) \geqslant C d_\mathcal{C}([\Id], g \cdot [\Id] )  \log \left ( \dfrac{4}{3} \right ) + C'.
\end{equation}
In particular, the previous inequality and \eqref{eq_condition_moment} imply that:
\begin{equation*}
\int d_{\mathcal{C}}([\Id], g \cdot [\Id]) d\mu(g) \leqslant \int_{G} \log(\deg(g)) d\mu(g) < + \infty.
\end{equation*}

As $G$ is non-elementary, Theorem \ref{thm_maher_tiozzo}  states that there exists a constant $L>0$ such that for almost every sample path
\begin{equation} \label{eq_drift}
\liminf_{n \rightarrow + \infty }\dfrac{d_\mathcal{C}([\Id], g_n \cdot [\Id])}{n} =L. 
\end{equation}
Moreover, by Theorem \ref{thm_int_degree_versus_distance}, the following inequality holds:
 \begin{equation} \label{eq_deg_vs_dist}
 \log\deg(g_n) \geqslant C d_\mathcal{C}([\Id], g_n \cdot [ \Id]) \log\left ( \dfrac{4}{3} \right ) + C',
 \end{equation}
 where $C,C'>0$. 
 As a result, \eqref{eq_drift} and \eqref{eq_deg_vs_dist} imply that:
 \begin{equation*}
 \dfrac{ \log\deg(g_n)}{n} \geqslant \dfrac{C}{n} d_{\mathcal{C}}([\Id], g_n \cdot [\Id]) \log\left ( \dfrac{4}{3}\right ) + \dfrac{C'}{n},
 \end{equation*}
 hence taking the limit as $n\rightarrow + \infty$ yields:
 \begin{equation*}
 \lambda \geqslant C L \log \left ( \dfrac{4}{3} \right ) >0,
 \end{equation*}
and we have proved that $\lambda>0$, as required.
The implication $(i) \Rightarrow (\lambda >0)$ holds.

\bigskip

Let us prove the implication $(v) \Rightarrow (\lambda >0)$. Suppose that $G$ satisfies condition $(v)$. By conjugation, we can suppose that $G$ is a subgroup of $\stab([x])$. 
Observe also that the following inequality holds for each $g \in G$:
\begin{equation*}
\dfrac{1}{2} d_{\mathcal{T}} ( g\cdot [\Id] , [\Id]) \log(2) \leqslant \log( \deg(g)),
\end{equation*} 
where $d_\mathcal{T}$ denotes the distance in the corresponding Bass-Serre tree on which $G$ acts by isometry. 
In particular, using the fact that $G$ contains two non-commuting hyperbolic isometries on $\mathcal{T}$ and Theorem \ref{thm_maher_tiozzo}, we obtain similarly that:
\begin{equation*}
\dfrac{1}{n}\int_G \log\deg(g) d\nu_n(g) \geqslant \dfrac{1}{2n} \int_G d_\mathcal{T}(g \cdot [\Id] , [\Id] ) d\nu_n(g) \rightarrow \dfrac{L}{2}, 
\end{equation*} 
as $n \rightarrow +\infty$ where $L>0$ is the drift of the associated to the random walk on the tree $\mathcal{T}$.
We have thus proven that $\lambda >0$ and the implication $(v) \Rightarrow (\lambda >0)$ holds.

\bigskip

Let us prove that the implication $((ii) \text{ or } (vi) ) \Rightarrow (\lambda = 0)$ holds.
Since the proof of the two implications $(ii) \Rightarrow (\lambda = 0)$ and $(vi) \Rightarrow (\lambda= 0)$ are very similar, we will only give the proof of $(ii) \Rightarrow (\lambda = 0)$.
 Suppose that there exists an hyperbolic automorphism $h \in G$ such that any automorphism $f \in G$ can be decomposed into $f = u(f) \circ h^{p(f)}$ where $p(f)$ is an integer and $u(f)$ belongs to $H$. 
We have thus:
\begin{equation*}
\dfrac{1}{n}\int_{G} \log \deg(g) d\nu_n(g) = \dfrac{1}{n} \int_G \log \deg (u(g) \circ h^{p(g)} ) d\nu_n(g).  
\end{equation*}
By the submultiplicativity of the degree, we have:
\begin{equation*}
\deg ( u(g) \circ h^{p(g)}) \leqslant C \deg(u(g)) \deg (h^{p(g)})
\end{equation*}
where $C>0$.
In particular, we obtain:
\begin{equation} \label{eq_sep_lineal_C}
\dfrac{1}{n}\int_{G} \log \deg(g) d\nu_n(g) \leqslant \dfrac{C}{n} + \dfrac{1}{n} \int_G \log(\deg(u(g))) d\nu_n(g) + \dfrac{1}{n} \int_G \log \deg(h^{p(g)}) d\nu_n(g) . 
\end{equation}
Since the map $p : G \to \mathbb{Z}$ is a morphism of groups, the random walk on $G$ induces a random walk on $\mathbb{Z}$ with transition given by the distribution $p_* \mu$. 
As the measure $p_*\mu$ is also symmetric, the law of large numbers implies that
\begin{equation*}
\dfrac{1}{n} \int_G \log \deg(h^{p(g)}) d\nu_n(g) \rightarrow 0,
\end{equation*}
as $n \rightarrow +\infty$.
Observe also that there exists a constant $M>0$ such that $\deg(u(g)) \leqslant M$ for all $g \in G$. In particular, the integral
\begin{equation*}
\dfrac{1}{n} \int_G \log \deg(u(g)) d\nu_n(g) \rightarrow 0,
\end{equation*}
as $n \rightarrow +\infty$.
Since each term on the right hand side of \eqref{eq_sep_lineal_C} tends to zero, we have thus proven that $$\lambda = \lim_{n\rightarrow +\infty} \dfrac{1}{n} \int_G \log \deg(g) d\nu_n(g) = 0 $$ and the implication $(ii) \Rightarrow (\lambda = 0)$ holds.

\bigskip

Let us prove that the implication $((iii)\text{ or } (iv) \text{ or } (vii)) \Rightarrow (\lambda = 0)$ holds. Observe that if $G$ satisfies assertion $(iii)$ or $(iv)$ or $(vi)$ then the degree of any element of $G$ is uniformly bounded, hence the degree exponent is zero.
We have thus proved the implication $((iii)\text{ or } (iv) \text{ or } (vii)) \Rightarrow (\lambda = 0)$.

\bibliographystyle{amsalpha}
\bibliography{ref_tame_q}

\newcommand{\etalchar}[1]{$^{#1}$}
\providecommand{\bysame}{\leavevmode\hbox to3em{\hrulefill}\thinspace}
\providecommand{\MR}{\relax\ifhmode\unskip\space\fi MR }
\providecommand{\MRhref}[2]{%
  \href{http://www.ams.org/mathscinet-getitem?mr=#1}{#2}
}
\providecommand{\href}[2]{#2}
\begin{thebibliography}{AAdBM99}

\bibitem[AAdB{\etalchar{+}}99]{maillard_baxter_birational}
N.~Abarenkova, J.-Ch. Angl\`es~d'Auriac, S.~Boukraa, S.~Hassani, and J.-M.
  Maillard, \emph{From {Y}ang-{B}axter equations to dynamical zeta functions
  for birational transformations}, Statistical physics on the eve of the 21st
  century, Ser. Adv. Statist. Mech., vol.~14, World Sci. Publ., River Edge, NJ,
  1999, pp.~436--490. \MR{1704016}

\bibitem[AAdBM99]{maillard_growth_complexity}
N.~Abarenkova, J.-Ch. Angl\`es~d'Auriac, S.~Boukraa, and J.-M. Maillard,
  \emph{Growth-complexity spectrum of some discrete dynamical systems}, Phys. D
  \textbf{130} (1999), no.~1-2, 27--42. \MR{1694727}

\bibitem[AdMV06]{maillard_complexity}
J.-Ch. Angl\`es~d'Auriac, J.-M. Maillard, and C.~M. Viallet, \emph{On the
  complexity of some birational transformations}, J. Phys. A \textbf{39}
  (2006), no.~14, 3641--3654. \MR{2220002}

\bibitem[BC16]{blanc_cantat}
J{\'e}r{\'e}my Blanc and Serge Cantat, \emph{Dynamical degrees of birational
  transformations of projective surfaces}, J. Amer. Math. Soc. \textbf{29}
  (2016), no.~2, 415--471. \MR{3454379}

\bibitem[BD05]{bedford_diller_energy_invariant}
Eric Bedford and Jeffrey Diller, \emph{Energy and invariant measures for
  birational surface maps}, Duke Math. J. \textbf{128} (2005), no.~2, 331--368.
  \MR{2140266}

\bibitem[Bed15]{bedford_invertible_blow_ups}
Eric Bedford, \emph{Invertible dynamics on blow-ups of {$\Bbb P^k$}}, Complex
  analysis and geometry, Springer Proc. Math. Stat., vol. 144, Springer, Tokyo,
  2015, pp.~67--88. \MR{3446748}

\bibitem[BFJ08]{boucksom_favre_jonsson_deggrowth}
S{\'e}bastien Boucksom, Charles Favre, and Mattias Jonsson, \emph{Degree growth
  of meromorphic surface maps}, Duke Math. J. \textbf{141} (2008), no.~3,
  519--538. \MR{2387430}

\bibitem[BFL14]{bisi_furter_lamy}
Cinzia Bisi, Jean-Philippe Furter, and St{\'e}phane Lamy, \emph{The tame
  automorphism group of an affine quadric threefold acting on a square
  complex}, J. \'Ec. polytech. Math. \textbf{1} (2014), 161--223. \MR{3322787}

\bibitem[BH99]{bridson_haefliger}
Martin~R. Bridson and Andr\'e Haefliger, \emph{Metric spaces of non-positive
  curvature}, Grundlehren der Mathematischen Wissenschaften [Fundamental
  Principles of Mathematical Sciences], vol. 319, Springer-Verlag, Berlin,
  1999. \MR{1744486}

\bibitem[Bia16]{bianco}
Federico~Lo Bianco, \emph{On the primitivity of birational transformations of
  irreducible symplectic manifolds}, arXiv preprint arXiv:1604.05261 (2016).

\bibitem[BK14]{bedford_kim_dynamics_pseudo_3}
Eric Bedford and Kyounghee Kim, \emph{Dynamics of (pseudo) automorphisms of
  3-space: periodicity versus positive entropy}, Publ. Mat. \textbf{58} (2014),
  no.~1, 65--119. \MR{3161509}

\bibitem[BQ16]{quint_random_walk_reductive}
Yves Benoist and Jean-Fran\c{c}ois Quint, \emph{Random walks on reductive
  groups}, Ergebnisse der Mathematik und ihrer Grenzgebiete. 3. Folge. A Series
  of Modern Surveys in Mathematics [Results in Mathematics and Related Areas.
  3rd Series. A Series of Modern Surveys in Mathematics], vol.~62, Springer,
  Cham, 2016. \MR{3560700}

\bibitem[BS92]{bedford_smillie_3}
Eric Bedford and John Smillie, \emph{Polynomial diffeomorphisms of {$\bold
  C^2$}. {III}. {E}rgodicity, exponents and entropy of the equilibrium
  measure}, Math. Ann. \textbf{294} (1992), no.~3, 395--420. \MR{1188127}

\bibitem[BS99]{ballmann_swiatkowski}
Werner Ballmann and Jacek Swiatkowski, \emph{On groups acting on nonpositively
  curved cubical complexes}, Enseign. Math. (2) \textbf{45} (1999), no.~1-2,
  51--81. \MR{1703363}

\bibitem[BT10]{bedford_truong_matrix_inversion}
Eric Bedford and Tuyen~Trung Truong, \emph{Degree complexity of birational maps
  related to matrix inversion}, Comm. Math. Phys. \textbf{298} (2010), no.~2,
  357--368. \MR{2669440}

\bibitem[Can11]{cantat_bir_surfaces}
Serge Cantat, \emph{Sur les groupes de transformations birationnelles des
  surfaces}, Ann. of Math. (2) \textbf{174} (2011), no.~1, 299--340.
  \MR{2811600}

\bibitem[Dan18]{these_hyperbolique}
Nguyen-Bac Dang, \emph{{Degree growth of rational maps in dimension three}},
  Theses, {Universit{\'e} Paris-Saclay}, July 2018.

\bibitem[DF01]{diller_favre}
Jeffrey Diller and Charles Favre, \emph{Dynamics of bimeromorphic maps of
  surfaces}, Amer. J. Math. \textbf{123} (2001), no.~6, 1135--1169.
  \MR{1867314}

\bibitem[DH18]{dahmani_horbez}
Fran\c{c}ois Dahmani and Camille Horbez, \emph{Spectral theorems for random
  walks on mapping class groups and out {$(F_N)$}}, Int. Math. Res. Not. IMRN
  (2018), no.~9, 2693--2744. \MR{3801494}

\bibitem[DS05]{dinh_sibony_une_borne_sup}
Tien-Cuong Dinh and Nessim Sibony, \emph{Une borne sup\'erieure pour l'entropie
  topologique d'une application rationnelle}, Ann. of Math. (2) \textbf{161}
  (2005), no.~3, 1637--1644. \MR{2180409}

\bibitem[FH91]{fulton_harris}
William Fulton and Joe Harris, \emph{Representation theory}, Graduate Texts in
  Mathematics, vol. 129, Springer-Verlag, New York, 1991, A first course,
  Readings in Mathematics. \MR{1153249}

\bibitem[FJ11]{favre_jonsson_dynamical_compactification}
Charles Favre and Mattias Jonsson, \emph{Dynamical compactifications of {${\bf
  C}^2$}}, Ann. of Math. (2) \textbf{173} (2011), no.~1, 211--248. \MR{2753603}

\bibitem[FK60]{furstenberg_kesten_product_random_matrices}
H.~Furstenberg and H.~Kesten, \emph{Products of random matrices}, Ann. Math.
  Statist. \textbf{31} (1960), 457--469. \MR{0121828}

\bibitem[FM89]{friedland_milnor}
Shmuel Friedland and John Milnor, \emph{Dynamical properties of plane
  polynomial automorphisms}, Ergodic Theory Dynam. Systems \textbf{9} (1989),
  no.~1, 67--99. \MR{991490}

\bibitem[FS95]{fornaess_sibony_II}
John~Erik Fornaess and Nessim Sibony, \emph{Complex dynamics in higher
  dimension. {II}}, Modern methods in complex analysis ({P}rinceton, {NJ},
  1992), Ann. of Math. Stud., vol. 137, Princeton Univ. Press, Princeton, NJ,
  1995, pp.~135--182. \MR{1369137}

\bibitem[Fur63]{furstenberg_non_commuting_random_products}
Harry Furstenberg, \emph{Noncommuting random products}, Trans. Amer. Math. Soc.
  \textbf{108} (1963), 377--428. \MR{0163345}

\bibitem[FW12]{favrewulcandegree}
Charles Favre and Elizabeth Wulcan, \emph{Degree growth of monomial maps and
  {M}c{M}ullen's polytope algebra}, Indiana Univ. Math. J. \textbf{61} (2012),
  no.~2, 493--524. \MR{3043585}

\bibitem[Giz80]{gizatullin_80}
Marat~Kharisovitch Gizatullin, \emph{Rational {$G$}-surfaces}, Izv. Akad. Nauk
  SSSR Ser. Mat. \textbf{44} (1980), no.~1, 110--144, 239. \MR{563788}

\bibitem[Gro87]{gromov}
M.~Gromov, \emph{Hyperbolic groups}, Essays in group theory, Math. Sci. Res.
  Inst. Publ., vol.~8, Springer, New York, 1987, pp.~75--263. \MR{919829}

\bibitem[Gue05]{guedj_ergodic}
Vincent Guedj, \emph{Ergodic properties of rational mappings with large
  topological degree}, Ann. of Math. (2) \textbf{161} (2005), no.~3,
  1589--1607. \MR{2179389}

\bibitem[JMta12]{mustata_jonsson}
Mattias Jonsson and Mircea Musta\c~t\u a, \emph{Valuations and asymptotic
  invariants for sequences of ideals}, Ann. Inst. Fourier (Grenoble)
  \textbf{62} (2012), no.~6, 2145--2209 (2013). \MR{3060755}

\bibitem[Jun42]{jung}
Heinrich W.~E. Jung, \emph{\"uber ganze birationale {T}ransformationen der
  {E}bene}, J. Reine Angew. Math. \textbf{184} (1942), 161--174. \MR{0008915}

\bibitem[Kur16]{kuroda}
Shigeru Kuroda, \emph{Weighted multidegrees of polynomial automorphisms over a
  domain}, J. Math. Soc. Japan \textbf{68} (2016), no.~1, 119--149.
  \MR{3454556}

\bibitem[Lam01]{lamy_tits_C2}
St\'{e}phane Lamy, \emph{L'alternative de {T}its pour {${\rm Aut}[{\Bbb
  C}^2]$}}, J. Algebra \textbf{239} (2001), no.~2, 413--437. \MR{1832900}

\bibitem[Lam15]{lamy_combinatorics}
St{\'e}phane Lamy, \emph{Combinatorics of the tame automorphism group}, arXiv
  preprint arXiv:1505.05497 (2015).

\bibitem[Lan02]{lang}
Serge Lang, \emph{Algebra}, third ed., Graduate Texts in Mathematics, vol. 211,
  Springer-Verlag, New York, 2002. \MR{1878556}

\bibitem[Lin12]{lin_algebraic_stability_and_degree_growth}
Jan-Li Lin, \emph{Algebraic stability and degree growth of monomial maps},
  Math. Z. \textbf{271} (2012), no.~1-2, 293--311. \MR{2917145}

\bibitem[LP18]{lamy_przytycki}
St{\'e}phane Lamy and Piotr Przytycki, \emph{Presqu'un immeuble pour le groupe
  des automorphismes mod\'er\'es}, arXiv preprint arXiv:1802.00481 (2018).

\bibitem[LV13]{lamy_venereau}
St\'ephane Lamy and St\'ephane V\'en\'ereau, \emph{The tame and the wild
  automorphisms of an affine quadric threefold}, J. Math. Soc. Japan
  \textbf{65} (2013), no.~1, 299--320. \MR{3034406}

\bibitem[MT14]{maher_tiozzo}
Joseph Maher and Giulio Tiozzo, \emph{Random walks on weakly hyperbolic
  groups}, Journal f{\"u}r die reine und angewandte Mathematik (Crelles
  Journal) (2014).

\bibitem[OT14]{oguiso_truong_salem_threefold}
Keiji Oguiso and Tuyen~Trung Truong, \emph{Salem numbers in dynamics on
  {K}\"ahler threefolds and complex tori}, Math. Z. \textbf{278} (2014),
  no.~1-2, 93--117. \MR{3267571}

\bibitem[OT15]{oguiso_truong_explicit_rational_cy_threefolds_positive_entropy}
\bysame, \emph{Explicit examples of rational and {C}alabi-{Y}au threefolds with
  primitive automorphisms of positive entropy}, J. Math. Sci. Univ. Tokyo
  \textbf{22} (2015), no.~1, 361--385. \MR{3329200}

\bibitem[RS97]{russakovskii_shiffman}
Alexander Russakovskii and Bernard Shiffman, \emph{Value distribution for
  sequences of rational mappings and complex dynamics}, Indiana Univ. Math. J.
  \textbf{46} (1997), no.~3, 897--932. \MR{1488341}

\bibitem[SU03]{shestakov}
Ivan~P Shestakov and Ualbai~U Umirbaev, \emph{The nagata automorphism is wild},
  Proceedings of the National Academy of Sciences \textbf{100} (2003), no.~22,
  12561--12563.

\bibitem[Tru16]{truong_pseudo_aut}
Tuyen~Trung Truong, \emph{Some dynamical properties of pseudo-automorphisms in
  dimension 3}, Trans. Amer. Math. Soc. \textbf{368} (2016), no.~1, 727--753.
  \MR{3413882}

\bibitem[Tru17]{truong_aut_blowup_threefolds}
\bysame, \emph{Automorphisms of blowups of threefolds being {F}ano or having
  {P}icard number 1}, Ergodic Theory Dynam. Systems \textbf{37} (2017), no.~7,
  2255--2275. \MR{3693127}

\bibitem[Ure16]{urech}
Christian Urech, \emph{Remarks on the degree growth of birational
  transformations}, arXiv preprint 1606.04822 (2016).

\bibitem[Vaq00]{vaquie}
Michel Vaqui\'e, \emph{Valuations}, Resolution of singularities ({O}bergurgl,
  1997), Progr. Math., vol. 181, Birkh\"auser, Basel, 2000, pp.~539--590.
  \MR{1748635}

\bibitem[Wri15]{wright_tame_A3}
David Wright, \emph{The generalized amalgamated product structure of the tame
  automorphism group in dimension three}, Transform. Groups \textbf{20} (2015),
  no.~1, 291--304. \MR{3317803}

\end{thebibliography}

\end{document}